\documentclass[10pt, a4paper]{article}
\usepackage[lmargin=1in,rmargin=1in,bottom=1in,top=1in]{geometry}
\usepackage[english]{babel}
\usepackage[T1]{fontenc}
\usepackage[utf8]{inputenc}
\usepackage[square,numbers]{natbib}
\usepackage{doi}
\usepackage{preamble}

\providecommand{\keywords}[1]
{
  \small	
  \textbf{{Key words: }} #1
}
\providecommand{\msccodes}[1]
{
  \small	
  \textbf{{MSC codes: }} #1
}

\title{A theoretical study on the effect of mass lumping on the discrete frequencies in immersogeometric analysis}
\author{%
  Ivan Bioli\textsuperscript{a,b,c,1,}\thanks{Corresponding author. \textit{Email addresses}: \href{mailto:ivan.bioli@unipv.it}{ivan.bioli@unipv.it} (Ivan Bioli), \href{mailto:yannis.voet@epfl.ch}{yannis.voet@epfl.ch} (Yannis Voet).} \textsuperscript{,}\thanks{Authors in alphabetical order.}%
  \and Yannis Voet\textsuperscript{a,}\footnotemark[2]%
  }
\date{
    \footnotesize
    \textit{
    \textsuperscript{a} MNS, Institute of Mathematics, École polytechnique fédérale de Lausanne, Station 8, CH-1015 Lausanne, Switzerland\\
    \textsuperscript{b} Dipartimento di Matematica, Università degli Studi di Pavia, Via A. Ferrata 5, 27100 Pavia, Italy\\
    \textsuperscript{c} Dipartimento di Ingegneria Civile e Architettura, Università degli Studi di Pavia, Via A. Ferrata 3, 27100 Pavia, Italy\\
    }
}

\begin{document}
\maketitle
\begin{abstract}
  In structural dynamics, mass lumping techniques are commonly employed for improving the efficiency of  explicit time integration schemes and increasing their critical time step constrained by the largest discrete frequency of the system. For immersogeometric methods, Leidinger \cite{leidinger2020explicit} first showed in 2020 that for sufficiently smooth spline discretizations, the largest frequency was not affected by small trimmed elements if the mass matrix was lumped, a finding later supported by several independent numerical studies. This article provides a rigorous theoretical analysis aimed at unraveling this property. By combining linear algebra with functional analysis, we derive sharp analytical estimates capturing the behavior of the largest discrete frequency for lumped mass approximations and various trimming configurations. Additionally, we also provide estimates for the smallest discrete frequency, which has lately drawn closer scrutiny. Our estimates are then confirmed numerically for 1D and 2D problems.
\end{abstract}


\vspace{2ex}
\msccodes{65M60, 65M15, 65F15}


\sloppy
\section{Introduction}
Isogeometric analysis is a spline-based discretization technique for solving partial differential equations (PDEs) \cite{hughes2005isogeometric,cottrell2009isogeometric}. Conceived as a generalization of classical finite element methods (FEM), it relies on spline functions from Computer-Aided-Design (CAD), such as B-splines and Non-Uniform Rational B-splines (NURBS) for both representing the geometry and approximating the solution. Isogeometric analysis offers many distinctive advantages, including exact representation of common geometries, smooth solutions and better approximation properties \cite{bazilevs2006isogeometric,bressan2019approximation,sande2020explicit}. It has proven itself in many applications, including fluid-structure interaction \cite{hsu2015dynamic,nitti2020immersed}, biomedical applications \cite{morganti2015patient,lorenzo2019computer} and fracture mechanics \cite{borden2014higher} to name just a few. In particular, isogeometric analysis rapidly stood out in spectral analysis, as it became known for removing the so-called ``optical branches'' from the discrete spectrum characterizing classical finite element discretizations \cite{cottrell2006isogeometric,cottrell2007studies,hughes2014finite}. Although isogeometric analysis allows approximating a large fraction of the eigenvalues, a few highly inaccurate discrete eigenvalues persist and were called \emph{outlier} eigenvalues for this reason \cite{cottrell2006isogeometric}. Outlier eigenvalues are especially a hurdle for explicit time integration schemes, whose critical time step is typically inversely proportional to the largest discrete frequency. For instance, the critical time step of the undamped central difference method is given by
\begin{equation}
    \label{eq:CFL_central_difference}
    \Delta t_c= \frac{2}{\sqrt{\lambda_n}}
\end{equation}
where $\sqrt{\lambda_n}$ is the largest frequency of the discrete system \cite{hughes2012finite,bathe2006finite}. Since the early days of isogeometric analysis, removing the outliers from the discrete spectrum has been of great interest. Some of the outlier removal techniques proposed in the literature include nonlinear parametrizations \cite{cottrell2006isogeometric,chan2018multi}, strong \cite{hiemstra2021removal,manni2022application,manni2023outlier} or weak \cite{deng2021boundary} imposition of additional boundary conditions (also related to $n$-width optimal spline spaces) and eigenvalue deflation techniques \cite{tkachuk2014local,gonzalez2020large,voet2025mass}.

In explicit dynamics, outliers are not the only concern. Due to the inherent cost of solving linear systems with the mass matrix ``exactly'' \cite{collier2012cost,collier2013cost}, mass lumping techniques are often employed. Mass lumping consist in directly substituting the consistent mass matrix with an ad hoc approximation, generally diagonal. Mass lumping is strongly rooted in the engineering community and accurate techniques have been known ever since the 1970s for classical FEM \cite{fried1975finite,cohen1994higher}. Unfortunately, they explicitly rely on the interpolatory nature of the basis functions, a property that is generally lost in isogeometric analysis. In this context, standard mass lumping techniques such as the row-sum only deliver second order convergent eigenvalues \cite{cottrell2006isogeometric}, independently of the spline order. Although a general proof is still lacking, this property has been systematically observed, also for generalizations of the row-sum technique \cite{voet2023mathematical,voet2025mass}. Instead, high order mass lumping techniques either rely on (approximate) dual bases \cite{anitescu2019isogeometric,nguyen2023towards,hiemstra2025higher} or interpolatory spline bases \cite{li2022significance,li2024interpolatory}. However, none of these approaches are perfect: while the former is being held back by several technical difficulties, primarily related to the imposition of boundary conditions, the latter leads to globally supported basis functions and dense stiffness matrices. Mass lumping techniques not only alleviate the burden of solving linear systems but are also often praised for increasing the critical time step, a property formally proved for the row-sum technique and nonnegative matrices \cite{voet2023mathematical}.

For immersogeometric analysis, this effect is sometimes extreme. Immersed methods consist in embedding a complex domain into a much simpler fictitious domain that is easily meshed. In isogeometric analysis, trimming curves (or surfaces) cut across the fictitious domain and delimit the physical domain. While trimming drastically simplifies mesh generation, it also introduces new difficulties, including integration on trimmed elements and boundaries, imposition of essential boundary conditions and severe stability and conditioning issues resulting from small trimmed elements \cite{de2017condition,de2019preconditioning,de2023stability}. Unless some form of stabilization is employed, the stable step size of explicit time integration methods with the consistent mass may become arbitrarily small \cite{de2023stability}. However, Leidinger \cite{leidinger2020explicit} first showed that for sufficiently smooth spline discretizations, mass lumping removed the dependency of the critical time step on the size of the trimmed element. This intriguing finding was later confirmed by several independent numerical studies in \cite{messmer2021isogeometric,messmer2022efficient,coradello2021accurate,stoter2023critical,radtke2024analysis}. Yet, only few authors have explored this phenomenon from a theoretical perspective.
Leidinger himself first studied the differences between $C^0$ and maximally smooth discretizations of degree $1$, $2$ and $3$ by explicitly computing some of the entries of the system matrices for a simple 1D problem. He concluded that maximally smooth discretizations of degree $p \geq 2$ guaranteed bounded eigenvalues.

More recently, Stoter et al. \cite{stoter2023critical} thoroughly analyzed the behavior of the largest eigenvalue by tracking down the dependency on the degree, dimension, boundary conditions and trimming configuration but mostly focused on the maximally smooth case. Their analysis consists in evaluating the Rayleigh quotient for clever guesses of the largest eigenfunction, thereby producing estimates of the largest eigenvalue that are later verified numerically. However, these estimates are only lower bounds. Neither can they guarantee that the largest eigenvalue remains bounded nor necessarily predict how fast it blows up. This article extends the previous study by providing an upper bound. Using sharp eigenvalue bounds, we derive theoretical estimates that capture the behavior of the largest discrete eigenvalue and support previous numerical findings. As we will show, this eigenvalue may exhibit distinctive behaviors that critically depend on the smoothness and trimming configuration (in higher dimensions). In addition to natural generalizations of 1D trimming configurations to single directions (such as ``corner cuts'' and ``sliver cuts'' in \cite{stoter2023critical}), our analysis has uncovered a new trimming configuration specific to 2D geometries and yielding novel insights on the behavior of the largest eigenvalue. We provide a separate study for this case, thereby complementing existing literature. Moreover, contrary to most existing work, our analysis seamlessly covers for all intermediate continuities ranging from $C^0$ to maximal smoothness.

Furthermore, the effect of mass lumping on the largest discrete eigenvalue might only be the tip of the iceberg. In the meantime, some of the more subtle intertwining of mass lumping and trimming has emerged and it appears that instead the smallest eigenvalue may converge to zero, thereby introducing spurious eigenvalues and modes in the low-frequency spectrum \cite{coradello2021accurate,radtke2024analysis}. Contrary to common belief, small spurious eigenvalues might have deleterious consequences for transient problems \cite{voet2025mass}.
Thus, we complement our findings with theoretical estimates for the smallest discrete eigenvalue. Our estimates confirm that, if the degree is large enough, this eigenvalue converges to zero.

The outline for the rest of the article is as follows: after recalling the context and standard notations in \Cref{se:model_problem}, we first derive sharp algebraic eigenvalue bounds in \Cref{se:eig_bounds}. These bounds later serve to deduce analytical estimates in \Cref{se:1D_trimming} capturing the eigenvalues' asymptotic behavior for the 1D Laplace problem. Our estimates are later extended to 2D in \Cref{se:2D_trimming} by considering various trimming configurations. Clues for further extending them to the bi-Laplace problem are also provided. Our estimates are confirmed by a series of numerical experiments for 1D and 2D problems. Finally, \Cref{se:conclusion} summarizes our findings and draws conclusions.

\section{Model problem and discretization}
\label{se:model_problem}
We consider as model problem the classical wave equation. Let $\Omega \subset \mathbb{R}^d$ be an open, connected physical domain with Lipschitz boundary $\partial \Omega = \overline{\partial \Omega_D \cup \partial \Omega_N}$ where $\partial \Omega_D \cap \partial \Omega_N = \emptyset$ and $\partial \Omega_D \neq \emptyset$. The physical domain $\Omega \subset \Omega_0$ is embedded in a fictitious domain $\Omega_0$ (also sometimes called extended domain or ambient domain). Let $[0,T]$ be the time domain with $T>0$ denoting the final time. We look for $u \colon \Omega \times [0,T] \to \mathbb{R}$ such that
\begin{align}
    \partial_{tt} u(\bm{x},t)-\Delta u(\bm{x},t) & =f(\bm{x},t) &  & \text{ in } \Omega \times (0,T], \label{eq:wave_equation}           \\
    u(\bm{x},t)                                  & =0           &  & \text{ on } \partial \Omega_D \times (0,T], \label{eq:dirichlet_bc} \\
    \partial_n u(\bm{x},t)                       & =h(\bm{x},t) &  & \text{ on } \partial \Omega_N \times (0,T], \label{eq:neumann_bc}   \\
    u(\bm{x},0)                                  & =u_0(\bm{x}) &  & \text{ in } \Omega,  \label{eq:initial_disp}                        \\
    \partial_t u(\bm{x},0)                       & =v_0(\bm{x}) &  & \text{ in } \Omega. \label{eq:initial_vel}
\end{align}
Equation \eqref{eq:wave_equation} describes the wave propagation in a medium and is complemented with Dirichlet \eqref{eq:dirichlet_bc} and Neumann \eqref{eq:neumann_bc} boundary conditions. Equations \eqref{eq:initial_disp} and \eqref{eq:initial_vel} impose the initial conditions. For the sake of the presentation, we prescribe homogeneous Dirichlet boundary conditions and assume that $\partial \Omega_D \subset \partial \Omega_0$. Our results, however, also apply to more general cases where such boundary conditions are imposed weakly on the trimmed boundary. Denoting
\begin{equation*}
    V = \{v \in \HoneO \colon v|_{\partial \Omega_D} = 0\}
\end{equation*}
and defining the Bochner space
\begin{equation*}
    V_T = \{v \in \Ltwo((0,T); V) \colon \ \partial_t v \in \Ltwo((0,T); \LtwoO), \ \partial_{tt} v \in \Ltwo((0,T); V^*) \},
\end{equation*}
the weak form of the problem reads: find $u \in V_T$ such that for almost every time $t \in (0,T)$
\begin{align*}
    \langle\partial_{tt} u(t), v\rangle + (\nabla u(t), \nabla v)_{\LtwoO^d} & = (f(t),v)_{\LtwoO} + (h(t),v)_{\Ltwo(\partial \Omega_N)} & \forall v \in V, \label{eq:continuous_weak_form} \\
    (u(0),v)_{\LtwoO}                                                        & = (u_0,v)_{\LtwoO}                                        & \forall v \in V, \nonumber                       \\
    \langle \partial_t u(0),v \rangle                                        & = \langle v_0,v \rangle                                   & \forall v \in V, \nonumber
\end{align*}
where we assume that $f \in \Ltwo((0,T); \LtwoO)$, $h \in \Ltwo((0,T); \Ltwo(\partial \Omega_N))$, $u_0 \in \LtwoO$, $v_0 \in V^*$ and $\langle .,.\rangle$ denotes the duality pairing between $V^*$ and $V$ (see e.g. \cite{evans2022partial}). The Galerkin method then seeks an approximate solution $u_h(.,t)$ in a finite dimensional subspace $V_h \subset V$. In isogeometric analysis, $V_h$ is chosen to be a spline space; i.e. a space of piecewise polynomials. For trimmed geometries, a basis for such a space is obtained by first constructing a tensor product B-spline basis $\widehat{\Phi}=\{\varphi_i\}_{i=1}^{\hat{n}}$ over a background mesh $\widehat{\mathcal{T}}_h$ and then retaining all basis functions whose support intersects the physical domain $\Omega$. In 1D, the B-spline basis functions are constructed recursively from a sequence of non-decreasing real numbers $\Xi:=(\xi_1,\dots,\xi_{\hat{n}+p+1})$, called a knot vector \cite{hughes2005isogeometric}. The integers $p$ and $\hat{n}$ denote the spline degree and spline space dimension, respectively. In practice, so-called open knot vectors are commonly used, where the first and last knot is repeated $p+1$ times. Varying the multiplicity $1 \leq m \leq p$ of internal knots allows constructing $C^k$ continuous spline spaces, denoted $\mathcal{S}^{k}_{p,\Xi}$, where $k=p-m$. The $i$th B-spline only depends on the knots $\xi_i, \xi_{i+1}, \dots, \xi_{i+p+1}$ and its regularity is recalled in the next theorem.

\begin{theorem}[{\cite[Theorem~3.19]{lyche2018spline}}]
    \label{thm:spline_continuity}
    Suppose that the knot $\xi$ occurs $m$ times among the knots $\xi_i, \xi_{i+1}, \dots, \xi_{i+p+1}$, defining the B-spline $\varphi_i$. Then, if $1 \leq m \leq p$,
    \begin{equation*}
        \varphi_i \in C^{p-m}(\xi) \setminus C^{p+1-m}(\xi).
    \end{equation*}
\end{theorem}
The basis construction is extended to dimension $d \geq 2$ by a tensor product argument and leads to a tensor product spline space $\mathcal{S}^{\bm{k}}_{\bm{p},\bm{\Xi}}$ where the dependency on the knot vectors $\Xi_1,\dots,\Xi_d$ is specified by $\bm{\Xi}$ and the degrees, space dimensions and continuities along each direction are collected in the vectors $\bm{p}=(p_1,p_2,\dots,p_d)$, $\hat{\bm{n}}=(\hat{n}_1,\hat{n}_2,\dots,\hat{n}_d)$ and $\bm{k}=(k_1,k_2,\dots,k_d)$, respectively. Tensor product basis functions are commonly labeled with multi-indices $\bm{i}=(i_1,i_2,\dots,i_d)$ and we allow a slight abuse of notation by identifying them with ``linear'' indices in the global numbering:
\begin{equation*}
    \varphi_i=\varphi_{\bm{i}}=\varphi_{1, i_1}\varphi_{2, i_2}\dots \varphi_{d, i_d},
\end{equation*}
where $\varphi_{l,j}$ denotes the $j$th function in the $l$th direction and $1 \leq i \leq \hat{n}:=\prod_{l=1}^d \hat{n}_l$ is a global index, which only depends on $\bm{i}$ and $\hat{\bm{n}}$. The so-called ``active mesh'' and ``active basis'' over the trimmed domain are then defined as
\begin{equation*}
    \mathcal{T}_h = \{ T \in \widehat{\mathcal{T}}_h \colon T \cap \Omega \neq \emptyset \} \subseteq \widehat{\mathcal{T}}_h,
\end{equation*}
and
\begin{equation*}
    \Phi = \{\varphi_i \in \widehat{\Phi} \colon \mathrm{supp}(\varphi_i) \cap \Omega \neq \emptyset \} \subseteq \widehat{\Phi}.
\end{equation*}
Let $I$ denote its index set, $V_h = \Span(\Phi)$ the resulting space and $n = \dim (V_h)$ its dimension. To simplify the presentation, we assume the basis functions are relabeled such that $I=\{1,2,\dots,n\}$. Denoting
\begin{equation*}
    V_{h,T} = \{v \in \Ltwo((0,T); V_h) \colon \ \partial_t v \in \Ltwo((0,T); V_h), \ \partial_{tt} v \in \Ltwo((0,T); V_h) \},
\end{equation*}
the Galerkin method seeks an approximate solution $u_h \in V_{h,T}$ such that for almost every time $t \in (0,T)$
\begin{align}
    b(\partial_{tt} u_h(t), v_h) + a(u_h(t),v_h) & = F(v_h)             & \forall v_h \in V_h, \label{eq:weak_form} \\
    (u_h(0),v_h)_{\LtwoO}                        & = (u_0,v_h)_{\LtwoO} & \forall v_h \in V_h, \nonumber            \\
    (\partial_t u_h(0),v_h)_{\LtwoO}             & = (v_0,v_h)_{\LtwoO} & \forall v_h \in V_h, \nonumber
\end{align}
where $a,b \colon V_h \times V_h \to \mathbb{R}$ are symmetric bilinear forms defined as
\begin{equation}
    \label{eq:global_contribution}
    a(u_h,v_h) = (\nabla u_h, \nabla v_h)_{\LtwoO^d}, \quad \text{and} \quad b(u_h,v_h) = (u_h,v_h)_{\LtwoO}
\end{equation}
and $F \colon V_h \to \mathbb{R}$ is the linear functional
\begin{equation*}
    F(v_h) = (f(t),v_h)_{\LtwoO} + (h(t),v_h)_{\Ltwo(\partial \Omega_N)}.
\end{equation*}
In this simplified presentation, the bilinear forms $a$ and $b$ induce the $\Hone$ seminorm and $\Ltwo$ norm as
\begin{equation}
    \label{eq:norms}
    \abs{u_h}_{\HoneO}^2 = a(u_h,u_h) \quad \text{and} \quad \norm{u_h}_{\LtwoO}^2 = b(u_h,u_h)
\end{equation}
but in practice, both might account for additional penalization terms \cite{stoter2023critical}. After expanding the approximate solution in the B-spline basis $\Phi$ and testing against every basis function in $\Phi$, equation \eqref{eq:weak_form} leads to a semi-discrete formulation, which is a system of ordinary differential equations \cite{hughes2012finite,quarteroni2009numerical}
\begin{align}
    \label{eq:semi_discrete_pb}
    \begin{split}
        M\ddot{\bm{u}}(t) + K\bm{u}(t) & = \bm{f}(t) \qquad \text{for } t \in [0,T], \\
        \bm{u}(0)                      & = \bm{u}_0,                                 \\
        \dot{\bm{u}}(0)                & = \bm{v}_0.
    \end{split}
\end{align}
where $\bm{u}(t)$ is the coefficient vector of $u_h(\bm{x},t)$ in the basis $\Phi$,
\begin{equation}
    \label{eq:stiffness_mass_matrix}
    K = (a(\varphi_i,\varphi_j))_{i,j=1}^n \quad \text{and} \quad M = (b(\varphi_i,\varphi_j))_{i,j=1}^n
\end{equation}
are the stiffness and mass matrices (i.e. the Gram matrices of $a(u_h,v_h)$ and $b(u_h,v_h)$ in the basis $\Phi$) and
\begin{equation*}
    \bm{f}(t) = (F(\varphi_i))_{i=1}^n.
\end{equation*}
From equations \eqref{eq:norms} and \eqref{eq:stiffness_mass_matrix}, it follows that the stiffness and mass matrices are symmetric and $M$ is always positive definite while $K$ is only positive definite if $\partial \Omega_D \neq \emptyset$ such that the $\Hone$ seminorm actually becomes a norm. Moreover, owing to the nonnegativity of the B-spline basis, $M$ is additionally nonnegative.

There exist scores of methods for integrating equation \eqref{eq:semi_discrete_pb} in time; see e.g. \cite{hughes2012finite,bathe2006finite}. In structural dynamics, explicit time integration schemes are by far the most popular for various reasons. Firstly, the step size is primarily constrained by the physics of the phenomenon that is simulated. Fast dynamic processes such as blast and impacts require exceedingly small step sizes. This restriction prevents using larger step sizes within unconditionally stable implicit methods, although stability alone would allow it. This holds especially true for nonlinear problems, where nonlinear system solvers might even fail to converge. Secondly, explicit methods offer the possibility of lumping the mass matrix, thereby saving up on both storage and floating point operations, which drastically enhances performance. As a matter of fact, commercial codes are often completely matrix-free as the stiffness matrix is never explicitly assembled. Despite significant advances in mass lumping techniques \cite{anitescu2019isogeometric,nguyen2023towards,voet2025mass}, only very few are applicable to immersed methods and even if they are, the lumped mass matrix may not be completely diagonal. Thus, in many instances the row-sum technique remains the method of choice, especially in isogeometric analysis where positive definite lumped mass matrices are guaranteed owing to the nonnegativity of the B-splines. The row-sum lumped mass matrix is defined algebraically as
\begin{equation} \label{eq:lumped_mass}
    \LM = \diag(d_1,\dots,d_n),
\end{equation}
where $d_i = \sum_{j=1}^n M_{ij}$. Finite element procedures usually ignore the contribution of Dirichlet boundary conditions when initially assembling all system matrices. We assume mass lumping also takes place at this stage and leverages the partition-of-unity property of the basis. Hence,
\begin{equation*}
    d_i = \sum_{j=1}^n M_{ij} = \sum_{j=1}^n b(\varphi_i, \varphi_j) = b(\varphi_i, 1) = \int_\Omega \varphi_i = \norm{\varphi_i}_{\LoneO},
\end{equation*}
where the last equality follows from the nonnegativity of the B-splines. Practically speaking, substituting $M$ with $\hat{M}$ consists in replacing the bilinear form $b(u_h,v_h)$ in \eqref{eq:weak_form} by the discrete one $\hat{b} \colon V_h \times V_h \to \mathbb{R}$ defined as
\begin{equation*}
    \hat{b}(u_h,v_h) = \bm{u}^\top \hat{M} \bm{v},
\end{equation*}
where $\bm{u}$ and $\bm{v}$ are the coefficient vectors of $u_h$ and $v_h$, respectively, in the B-spline basis $\Phi$. Hence, the lumped mass explicitly depends on the B-spline basis, which undoubtedly bears consequences for the analysis.

In the engineering literature, the mass matrix is always lumped and the notation adopted therein rarely distinguishes it from the consistent mass. However, the distinction is crucial within this article. When explicit methods are applied to equation \eqref{eq:semi_discrete_pb}, their critical time step is inversely proportional to $\sqrt{\lambda_n(K,M)}$, where $\lambda_n(K,M)$ denotes the largest generalized eigenvalue of $(K,M)$; see e.g. equation \eqref{eq:CFL_central_difference} for the central difference method in the undamped case. Substituting the consistent mass $M$ with the lumped mass $\LM$ alters the critical time step of the method and since $\lambda_n(K,\LM) \leq \lambda_n(K,M)$ \cite[Corollary 3.10]{voet2023mathematical}, the row-sum technique increases the critical time step. For trimmed geometries, this effect is significant: whereas $\lambda_n(K,M)$ becomes unbounded for arbitrarily small trimmed elements \cite{de2023stability}, several numerical studies have indicated that $\lambda_n(K,\LM)$ does not depend on the size of the trimmed element for sufficiently smooth isogeometric discretizations \cite{messmer2021isogeometric,messmer2022efficient,coradello2021accurate,stoter2023critical,radtke2024analysis}. However, these results have overshadowed another much more subtle effect: while the largest eigenvalue $\lambda_n(K,\LM)$ may appear unaffected, the smallest one $\lambda_1(K,\LM)$ typically tends to zero as the trimmed elements get smaller. This creates severe inaccuracies in the low-frequency spectrum and their consequences on the dynamics have been thoroughly investigated in separate work \cite{voet2025stabilization}. This article will exclusively focus on better understanding the behavior of the extreme eigenvalues of $(K,\LM)$, regardless of their impact on the quality of the solution.

\section{Eigenvalue bounds}
\label{se:eig_bounds}
In this section, we focus on deriving general bounds on the smallest and largest eigenvalues of $(K,\LM)$ and later specialize our results to the Laplace operator. Firstly, from the min-max characterization of eigenvalues \cite[Theorem 4.2.6]{horn2012matrix}, a trivial upper bound for the smallest eigenvalue is given by:
\begin{equation}
    \label{eq:upperbound_lambda1}
    \lambda_1(K, \LM) = \min_{\bm{x}\in\R^n} \frac{\bm{x}^\top K \bm{x}}{\bm{x}^\top \LM \bm{x}} \leq \min_{i} \frac{\bm{e}_i^\top K \bm{e}_i}{\bm{e}_i^\top \LM \bm{e}_i} = \min_{i} \frac{K_{ii}}{\LM_{ii}},
\end{equation}
where $\bm{e}_i$ is the $i$th vector of the canonical basis of $\mathbb{R}^n$. Similarly, a lower bound for the largest eigenvalue is:
\begin{equation}
    \label{eq:lowerbound_lambdan}
    \lambda_n(K, \LM) = \max_{\bm{x}\in\R^n} \frac{\bm{x}^\top K \bm{x}}{\bm{x}^\top \LM \bm{x}} \geq \max_{i} \frac{\bm{e}_i^\top K \bm{e}_i}{\bm{e}_i^\top \LM \bm{e}_i} = \max_{i} \frac{K_{ii}}{\LM_{ii}}.
\end{equation}
Moreover, since $\LM$ is diagonal, we derive the following upper bound:
\begin{equation}
    \label{eq:upperbound_lambdan_Bdiag_infnorm}
    \lambda_n(K, \LM) = \lambda_n(\LM^{-\frac{1}{2}} K \LM^{-\frac{1}{2}}) \leq \norm{\LM^{-\frac{1}{2}} K \LM^{-\frac{1}{2}}}_{\infty} = \max_{i} \sum_{j=1}^n \frac{\abs{K_{ij}}}{\sqrt{\LM_{ii}\LM_{jj}}}.
\end{equation}
Experimentally, equations \eqref{eq:lowerbound_lambdan} and \eqref{eq:upperbound_lambdan_Bdiag_infnorm} provided exceedingly sharp and computationally cheap bounds on $\lambda_n(K, \LM)$. The upper bound \eqref{eq:upperbound_lambdan_Bdiag_infnorm} is particularly useful given that it produces a sharp and conservative estimate for the stable step size. It appears that symmetrically scaling $K$ on both sides by $\LM^{-\frac{1}{2}}$ is critical for sharpness. Indeed, as noted in \cite{cocchetti2013selective}, only scaling on one side by $\LM^{-1}$ tends to produce a loose overestimate, which in our context did not even capture the behavior of $\lambda_n(K, \LM)$ in all cases.

Although computationally appealing, these bounds are not convenient for deriving theoretical estimates. The next lemma provides a working basis.

\begin{lemma}
    \label{lem:bounds_lambdan_lumped}
    The largest eigenvalue $\lambda_n(K, \LM)$ satisfies
    \begin{equation}
        \label{eq:bounds_lambdan_lumped}
        \max_{i} \frac{K_{ii}}{\LM_{ii}} \leq \lambda_n(K, \LM) \leq c(\bm{p}, d) \max_{i} \frac{K_{ii}}{\LM_{ii}}.
    \end{equation}
    where $c(\bm{p}, d) \leq \prod_{k=1}^d (2p_k+1)$ only depends on the dimension and spline degree.
\end{lemma}
\begin{proof}
    The result only requires reworking the upper bound \eqref{eq:upperbound_lambdan_Bdiag_infnorm}. Since the stiffness matrix $K$ is positive semidefinite, $|K_{ij}| \leq \sqrt{K_{ii}K_{jj}}$ (see e.g. \cite[Problem 7.1.P1]{horn2012matrix}) and consequently
    \begin{equation}
        \label{eq:upperbound_lambdan_lumped}
        \begin{split}
            \lambda_n(K, \LM) & \leq \max_{i} \sum_{j=1}^n \frac{\abs{K_{ij}}}{\sqrt{\LM_{ii}\LM_{jj}}} \leq c(\bm{p}, d) \max_{i,j} \frac{\abs{K_{ij}}}{\sqrt{\LM_{ii}\LM_{jj}}} \\ &\leq c(\bm{p}, d) \max_{i,j} \frac{\sqrt{K_{ii}K_{jj}}}{\sqrt{\LM_{ii}\LM_{jj}}} = c(\bm{p}, d) \max_{i} \frac{K_{ii}}{\LM_{ii}}
        \end{split}
    \end{equation}
    where $c(\bm{p}, d)$ denotes the maximum number of non-zero entries per row in $K$ and is at most $\prod_{k=1}^d (2p_k+1)$ for maximal regularity splines.
\end{proof}

\begin{rmk}
    Although we have the upper bound \eqref{eq:upperbound_lambda1}, it is generally not possible to derive a lower bound on $\lambda_1(K, \LM)$ of the form $\tilde{c}(\bm{p}, d) \, \min_{i} K_{ii} /\LM_{ii}$. Indeed, the matrix $K$ can be singular even if $K_{ii}\neq 0$ for all $i=1,\dots, n$. Thus, while the largest eigenvalue always behaves as the maximum ratio, the smallest eigenvalue may not behave as the minimum ratio. This key difference will later transpire in the experiments.
\end{rmk}

The bounds of \Cref{lem:bounds_lambdan_lumped} also hold in more general situations. As a matter of fact, the only requirements on $\LM$ and $K$ are that:
\begin{enumerate}[noitemsep]
    \item $\LM$ is diagonal and positive definite,
    \item $K$ is symmetric positive semidefinite and its maximum number of non-zero entries per row is bounded independently of the mesh size $h$.
\end{enumerate}

In practice, the bilinear form $a(u_h,v_h)$ and therefore the stiffness matrix $K$ often account for additional penalization terms originating from the weak imposition of Dirichlet boundary conditions \cite{stoter2023critical}. Fortunately, if the symmetric Nitsche method is chosen together with a suitable penalization parameter, the above requirements are met and \Cref{lem:bounds_lambdan_lumped} stills holds. However, adding a ghost penalization term to $\LM$, as originally envisioned in \cite{stoter2023critical}, ruins the diagonal structure of the lumped mass and violates the first requirement. Thus, our bounds are not applicable in this case. Formulating our bounds in terms of Rayleigh quotients is more convenient for deriving analytical estimates. The Rayleigh quotient $\hat{R} \colon V_h \setminus \{0\} \to \mathbb{R}^+$ is defined as
\begin{equation}
    \label{eq:rayleigh_quotient}
    \hat{R}(u_h) = \frac{a(u_h,u_h)}{\hat{b}(u_h,u_h)}.
\end{equation}

\begin{corollary}
    \label{corollary:eigen_laplace}
    Let $K$ and $\LM$ be the stiffness matrix and lumped mass matrix, respectively. The following bounds on the eigenvalues hold:
    \begin{align}
        \lambda_1(K, \LM)                                  & \leq \min_{i} \hat{R}(\varphi_i) \label{eq:bound_1_laplace},             \\
        \max_{i} \hat{R}(\varphi_i) \leq \lambda_n(K, \LM) & \leq c(\bm{p}, d) \max_{i}\hat{R}(\varphi_i) \label{eq:bound_n_laplace}.
    \end{align}
\end{corollary}
\begin{proof}
    The proof immediately follows from combining equations \eqref{eq:upperbound_lambda1} and \eqref{eq:bounds_lambdan_lumped} with the definition of the Rayleigh quotient \eqref{eq:rayleigh_quotient}.
\end{proof}

\begin{rmk}
    From the construction of the lumped mass matrix, we had already anticipated a relation between spectral and basis properties and indeed, \Cref{corollary:eigen_laplace} shows that the maximum eigenvalue behaves as the maximum of all evaluations of the Rayleigh quotient \eqref{eq:rayleigh_quotient} at the basis functions. In other words, we have replaced the difficult maximization problem
    \begin{equation*}
        \max_{\substack{v_h \in V_h \\ v_h \neq 0}} \hat{R}(v_h) \quad \text{by} \quad \max_{i} \hat{R}(\varphi_i),
    \end{equation*}
    which is much simpler. For the smallest eigenvalue, however, the minimum of all evaluations only provides an upper bound, which may not be sharp. In contrast, \cite{stoter2023critical} evaluates the Rayleigh quotient \eqref{eq:rayleigh_quotient} for clever guesses of the eigenfunctions, yielding lower bounds on the largest eigenvalue, which was the only quantity of interest therein. These estimates are insufficient for guaranteeing that the largest eigenvalue remains bounded and are later complemented with numerical experiments. \Cref{corollary:eigen_laplace} is particularly advantageous, as it allows us to exclusively focus on individual basis functions. Thus, analytical estimates for the evaluation of the Rayleigh quotient at the basis functions will yield valuable insights on the behavior of the largest eigenvalue. This was already conjectured in \cite{stoter2023critical} for the so-called first and second ``corner-cut'' functions. Therefore, some of the estimates derived therein (that also account for weakly imposed Dirichlet boundary conditions) remain applicable in this work. Nevertheless, we must stress that \Cref{corollary:eigen_laplace} critically relies on the assumption that $\hat{M}$ is diagonal, and neither holds for the consistent mass $M$ nor penalized versions of the lumped mass as proposed in \cite{stoter2023critical}.
\end{rmk}

Although \Cref{corollary:eigen_laplace} allows restricting our attention to individual basis functions, evaluating the minimum and maximum still requires identifying a subset of the basis functions that could cause trouble in the event of small trimmed elements. For this purpose, we first introduce a partition of the mesh into large (or ``good'') and small (or ``bad'') elements as in \cite{buffa2020minimal,burman2023extension}. Given a parameter $\gamma \in [0,1]$, the set of large elements $\mathcal{T}^L_h$ is defined as
\begin{equation*}
    \mathcal{T}^L_h = \{T \in \mathcal{T}_h \colon |T \cap \Omega | \geq \gamma |T| \}.
\end{equation*}
Its complement $\mathcal{T}^S_h = \mathcal{T}_h \setminus \mathcal{T}^L_h$ is the set of small elements:
\begin{definition}
    \label{def:bad_element}
    An element $T \in \mathcal{T}_h$ is called large (or good) if $T \in \mathcal{T}^L_h$ and is called small (or bad) otherwise.
\end{definition}

The actual value of $\gamma$ partitioning the mesh in good and bad elements does not matter much in the context of this work. An element $T$ such that $|T \cap \Omega | \to 0$ will eventually be labeled as ``bad'', regardless of how small $\gamma$ is. This partition allows defining the set of ``large'' basis functions $\Phi^L$ as
\begin{equation*}
    \Phi^L = \bigcup_{T \in \mathcal{T}_h^L} \Phi_{T}.
\end{equation*}
where
\begin{equation*}
    \Phi_{T} = \{\varphi_i \in \Phi \colon T \subseteq \supp(\varphi_i) \}.
\end{equation*}
Its complement $\Phi^S = \Phi \setminus \Phi^L$ defines the set of ``small'' basis functions:
\begin{definition}
    \label{def:bad_basis_function}
    A basis function $\varphi \in \Phi$ is called large (or good) if $\varphi \in \Phi^L$ and is called small (or bad) otherwise.
\end{definition}
Let $V_h^L=\Span(\Phi^L)$ and $V_h^S=\Span(\Phi^S)$ denote the subspaces spanned by the large and small basis functions, respectively, such that
\begin{equation*}
    V_h = V_h^L \oplus V_h^S.
\end{equation*}
The support of large basis functions contains at least one large element. Conversely, the support of small basis functions only contains small elements. The active support of a basis function $\varphi_i \in \Phi$ is defined as
\begin{equation*}
    \supp_\Omega(\varphi_i) = \Omega \cap \supp(\varphi_i)
\end{equation*}
and measures how ``small'' this function really is. For a bad basis function $\varphi_i \in \Phi^S$, $s_i = |\supp_\Omega(\varphi_i)|$ may become arbitrarily small and since
\begin{equation*}
    \lim_{s_i \to 0} a(\varphi_i,\varphi_i) = \lim_{s_i \to 0} \hat{b}(\varphi_i,\varphi_i) = 0,
\end{equation*}
the Rayleigh quotient in equation \eqref{eq:rayleigh_quotient} may either blow up, converge to zero or take a finite positive value. The behavior of $a(\varphi_i,\varphi_i)$ and $\hat{b}(\varphi_i,\varphi_i)$ depends not only on $s_i$, but also on the spline degree, continuity, boundary conditions and the trimming configuration (in higher dimensions).

Tracking down all these dependencies is a daunting task and the most complete analysis we are aware of is reported in \cite{stoter2023critical}. The results therein account for the spline degree, dimension and boundary conditions for various trimming configurations and tackle both second and fourth order differential operators. However, the computations (whose details were omitted) merely consist in evaluating the Rayleigh quotient for specific choices of functions. Thus, the estimates are only lower bounds on the largest eigenvalue and do not necessarily capture its behavior. Additionally, critical trimming configurations might have been forgotten and the continuity is only addressed as a side note. Yet, the latter is central for unraveling the behavior of the largest eigenvalue. Thus, the next section proposes a more rigorous analysis, although in a simplified setting by only prescribing Neumann boundary conditions on the trimmed boundary such that
\begin{equation*}
    a(\varphi_i,\varphi_i)=\abs{\varphi_i}_{\HoneO}^2 \quad \text{and} \quad \hat{b}(\varphi_i,\varphi_i)=\norm{\varphi_i}_{\LoneO}.
\end{equation*}
This simplification reduces the number of terms and therefore the length of the computations. Nevertheless, our bounds are sometimes applicable to more general cases considered in \cite{stoter2023critical} and nicely complement their findings. We will now focus on studying the ratio
\begin{equation}
    \label{eq:ratio}
    \hat{R}(\varphi_i) = \frac{\abs{\varphi_i}_{\HoneO}^2}{\norm{\varphi_i}_{\LoneO}}
\end{equation}
for basis functions $\varphi_i \in \Phi^S$. In the sequel, $C$ will denote a constant that may take different values at different locations.

\section{Trimming in 1D}
\label{se:1D_trimming}
We consider the physical domain $\Omega$ embedded within the fictitious domain $\Omega_0 = (0, 1)$ whose right end is trimmed off between knots $\xi_l$ and $\xi_{l+1}$. Let $\Omega_\ell$ denote the trimmed element intersected with the physical domain whose measure is $|\Omega_\ell|=\delta$. The trimming process is illustrated in \Cref{fig:1D_trimming}, along with the notation used in the subsequent analysis.
Our focus is on understanding the effects of trimming on $\lambda_n(K,\LM)$ and $\lambda_1(K,\LM)$ and more specifically, whether $\lambda_n(K,\LM)$ diverges and/or $\lambda_1(K,\LM)$ approaches $0$, as the trimming parameter $\delta$ goes to $0$. Based on Corollary~\ref{corollary:eigen_laplace}, we shift our attention to the behavior of the quantities $\min_i \hat{R}(\varphi_i)$ and $\max_i \hat{R}(\varphi_i)$.

\begin{figure}[h!]
    \centering
    \includegraphics[width=0.5\textwidth]{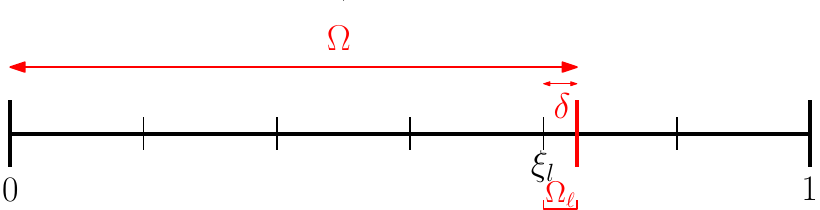}
    \caption{Visual representation of 1D trimming.}
    \label{fig:1D_trimming}
\end{figure}

As explained earlier, when $\delta\to 0$, the Rayleigh quotient \eqref{eq:ratio} can either blow up or tend to zero only if $\varphi_i \in \Phi^S$. In the 1D configuration depicted in \Cref{fig:1D_trimming}, $|\Phi^S|=m_l$, where $1 \leq m_l \leq p$ is the multiplicity of knot $\xi_l$, and $\supp_\Omega(\varphi_i)=\Omega_\ell \ \forall \varphi_i \in \Phi^S$. In the upcoming analysis, $k$ denotes the smoothness of $\varphi_i$, and might change from one function to another. Suppose $\varphi_i$ is $C^k$ across the knot $\xi_l$, but not $C^{k+1}$. Then $\restr{\varphi_i}{\Omega_\ell}$ is a polynomial of degree $p$ such that $\D^m_+\restr{\varphi_i}{\Omega_\ell}(\xi_l) = 0$ for all $m = 0, \dots, k$, and $\D^{k+1}_+\restr{\varphi_i}{\Omega{_\ell}}(\xi_l) \neq 0$. Denoting $\alpha_m = \frac{1}{m!}\D^m_+\varphi_i(\xi_l)$, we obtain
\[
    \restr{\varphi_i}{\Omega_\ell}(\xi) = \sum_{m=k+1}^{p} \alpha_m (\xi - \xi_l)^m.
\]
Consequently,
\[
    \begin{split}
        K_{ii} & = \int_\Omega \norm{\nabla \varphi_i}^2 = \int_{\Omega_\ell} \norm{\nabla \varphi_i}^2 = \int_{\xi_l}^{\xi_l+\delta} \left( \sum_{m=k+1}^{p} m\alpha_m (\xi - \xi_l)^{m-1} \right)^2 \dxi \\
               & = \int_{\xi_l}^{\xi_l+\delta} (\xi - \xi_l)^{2k}\left( \sum_{m=0}^{p-k-1} (k+1+m)\alpha_{k+1+m} (\xi - \xi_l)^{m} \right)^2 \dxi                                                          \\
               & = \int_{0}^{\delta} \xi^{2k}\left(C+ \calO(\xi)\right) \dxi \sim \delta^{2k+1} \qquad \text{as } \delta\to 0,
    \end{split}
\]

and similarly

\[
    \begin{split}
        \LM_{ii} & = \int_\Omega  \varphi_i = \int_{\Omega_\ell} \varphi_i = \int_{\xi_l}^{\xi_l+\delta}  \sum_{m=k+1}^{p} \alpha_m (\xi - \xi_l)^{m} \dxi \\
                 & = \int_{0}^{\delta} \xi^{k+1} \left(C+ \calO(\xi)\right) \dxi \sim \delta^{k+2} \qquad \text{as } \delta\to 0.
    \end{split}
\]
Therefore,
\[
    \frac{K_{ii}}{\LM_{ii}} \sim \frac{\delta^{2k+1}}{\delta^{k+2}} \sim \delta^{k-1} \qquad \text{as } \delta\to 0.
\]

By \Cref{thm:spline_continuity}, if the knot $\xi_l$ has multiplicity $m_l = p - \kappa_l$, there exists a spline $\varphi_i$ as above for each $k = \kappa_l, \dots, p-1$. The largest eigenvalue $\lambda_n(K, \LM)$ can blow up only if $k = 0$, thus
\begin{equation}
    \label{eq:lambdan_1d}
    \lambda_n(K, \LM) \sim
    \begin{cases*}
        \calO\left(\delta^{-1}\right) & \text{if  $\kappa_l = 0$} \\
        \calO\left( 1 \right)         & \text{otherwise}
    \end{cases*} \qquad \text{as } \delta\to 0.
\end{equation}
Conversely, the quickest decay to $0$ is obtained for $k = p-1$ (i.e. the spline of maximal regularity). Thus, we have the following upper bound on the smallest eigenvalue
\begin{equation}
    \label{eq:lambda1_1d}
    \lambda_1(K, \LM) \leq \min_i \frac{K_{ii}}{\LM_{ii}} =
    \begin{cases*}
        \calO\left(1\right)              & \text{if  $p \leq 2$} \\
        \calO\left(\delta^{p - 2}\right) & \text{otherwise}
    \end{cases*} \qquad \text{as } \delta\to 0.
\end{equation}
Although increasing the knot multiplicity $m_l$ decreases the global regularity of the spline space, it produces $m_l$ splines starting at $\xi_l$ of increasing regularity ranging from $p-m_l$ to $p-1$. The spline responsible for the quickest decay to $0$ is the one of maximal regularity, which always exists, independently of the global regularity of the space. \Cref{fig:1D_trimming_whichspline} confirms our findings by showing the B-splines attaining the maximum and the minimum. Remarkably, the behavior of $\lambda_n(K, \LM)$ depends solely on the local continuity of the spline basis across $\xi_l$, but not on the global continuity: it scales as $\calO\left(\delta^{-1}\right)$ for $C^0$ continuity and $\calO\left( 1 \right)$ otherwise. For the $C^0$ case, the maximum is attained for the spline of minimal regularity starting at $\xi_l$. This finding is also consistent with the explicit computations of Leidinger \cite[Section 4.3.1]{leidinger2020explicit}. Conversely, the upper bound on $\lambda_1(K, \LM)$ is attained for the spline of maximal regularity, which only depends on the degree of the splines and not on the (global) continuity.

\begin{figure}[h!]
    \centering
    \begin{subfigure}[b]{0.49\textwidth}
        \centering
        \includegraphics[width=\textwidth]{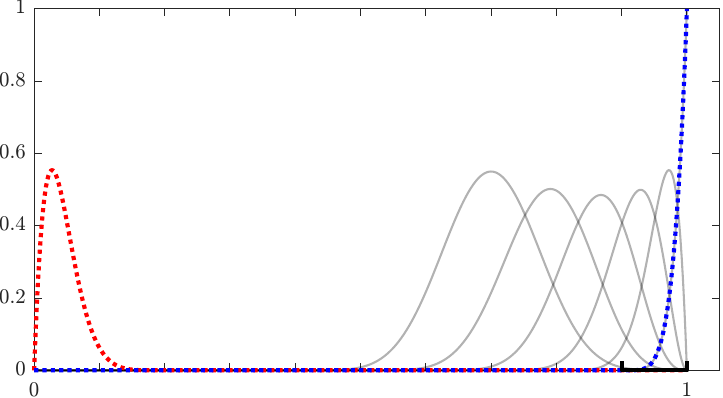}
        \caption{$C^{p-1}$ continuity.}
    \end{subfigure}
    \hfill
    \begin{subfigure}[b]{0.49\textwidth}
        \centering
        \includegraphics[width=\textwidth]{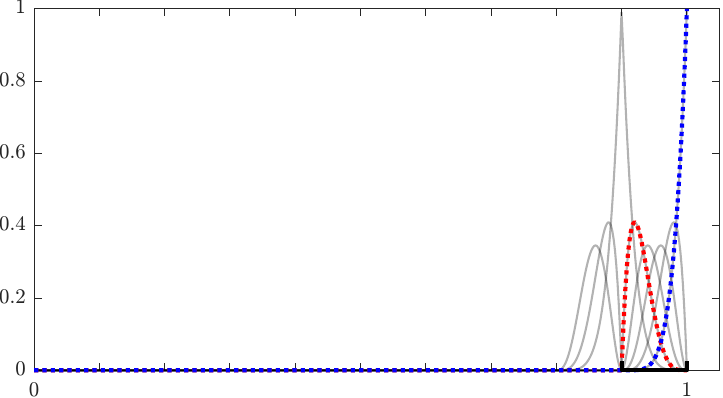}
        \caption{$C^0$ continuity.}
    \end{subfigure}
    \caption{B-splines attaining $\max_i K_{ii} / \LM_{ii}$ (red) and $\min_i K_{ii} / \LM_{ii}$ (blue), for degree $p = 5$ splines. The black segment represents the trimmed element.}
    \label{fig:1D_trimming_whichspline}
\end{figure}

\begin{rmk}
    In his Ph.D. thesis, Leidinger \cite{leidinger2020explicit} used a different argument to understand whether the largest eigenvalue blew up or remained bounded. By explicitly computing some of the entries of $\LM^{-1}K$ and from the equalities
    \begin{equation*}
        \sum_{i=1}^n \lambda_i(K,\LM) = \mathrm{trace}(\LM^{-1}K) = \sum_{i=1}^n \frac{K_{ii}}{\LM_{ii}},
    \end{equation*}
    he concluded that $\lambda_n(K,\LM)$ blew up if any of the diagonal entries of $\LM^{-1}K$ blew up. Conversely, it remained bounded if all diagonal entries of $\LM^{-1}K$ remained bounded. Our results are more precise: firstly, we provide both lower and upper bounds specifically on $\lambda_n(K,\LM)$ and secondly provide an upper bound on $\lambda_1(K,\LM)$. Additionally, both bounds accurately capture the dependency on the continuity.
\end{rmk}

\begin{rmk}
    Analogous computations for the consistent mass matrix yield
    \[
        M_{ii} = \int_\Omega \varphi_i^2 = \int_0^{\delta} \left( \xi^{k+1} \left(C + \calO(\xi)\right) \right)^2 \dxi \sim \delta^{2k+3} \qquad \text{as } \delta \to 0.
    \]
    Thus, if $\varphi_i \in \Phi^S$, we obtain the following lower bound for the largest eigenvalue when using the consistent mass matrix:
    \begin{equation} \label{eq:lambdan_consistent_1d}
        \lambda_n(K,M) \geq \frac{K_{ii}}{M_{ii}} \sim \frac{\delta^{2k+1}}{\delta^{2k+3}} \sim \delta^{-2} \qquad \text{as } \delta \to 0.
    \end{equation}
    Thus, with the consistent mass matrix, the largest eigenvalue grows at least as $\calO(\delta^{-2})$, independently of the spline degree and continuity. This result is well-known in the community and has been independently shown in \cite{stoter2023critical}. Note that the upper bound \eqref{eq:upperbound_lambdan_Bdiag_infnorm} does not hold for the consistent mass due to its non-diagonal structure. Instead, the result follows from a standard inverse inequality. The argument also holds in higher dimensions and is postponed to the next section.
\end{rmk}

\subsection{Numerical experiments}\label{sec:exp_1d}
We consider the fictitious domain $\Omega_0 = (0, 1)$, discretized using B-splines and $N = 1/h$ subdivisions, where $h$ is the mesh size. Let $p$ denote the spline degree and $k = p - m$ the continuity. We solve the Laplace eigenvalue problem on the trimmed domain $\Omega = (0, 1-h+\delta)$, where the last element on the right boundary is trimmed by $h-\delta$, similarly to \cref{fig:1D_trimming}. Here, we consider $N = 2^7$ and $\delta = 10^{-3}, 10^{-4}, \dots, 10^{-8}$. Homogeneous Dirichlet boundary conditions are strongly imposed on the left boundary, while homogeneous Neumann boundary conditions are applied on the right boundary.

As shown in \Cref{fig:1Dtrimming_lambdan}, the lower bound \eqref{eq:lowerbound_lambdan} and the upper bound \eqref{eq:upperbound_lambdan_Bdiag_infnorm} not only accurately capture the behavior of $\lambda_n(K, \LM)$ as a function of the trimming parameter $\delta$ but also numerically yield exceedingly sharp bounds. Additionally, these experimental results perfectly match \eqref{eq:lambdan_1d}. Specifically, the behavior of $\lambda_n(K, \LM)$ depends solely on the continuity of the spline basis, not its degree: it scales as $\calO\left(\delta^{-1}\right)$ for $C^0$ inter-element continuity and remains nearly constant otherwise. On the other hand, $\lambda_n(K, M)$ scales as $\calO(\delta^{-2})$ regardless of the degree or continuity of the splines, consistently with \eqref{eq:lambdan_consistent_1d}.

\begin{figure}[h!]
    \centering
    \begin{subfigure}[t]{0.49\textwidth}
        \centering
        \includegraphics[width=\textwidth]{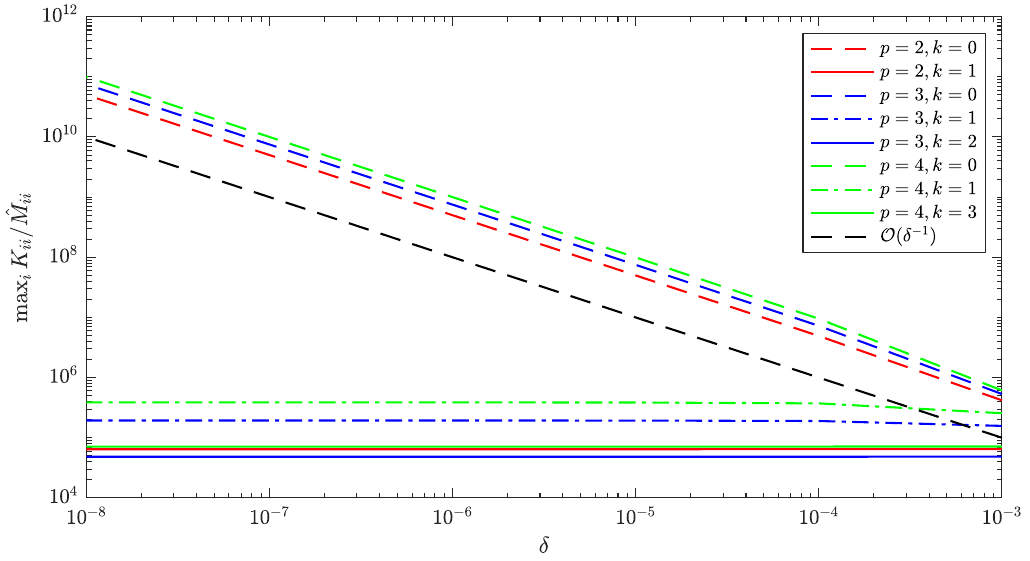}
        \caption{Lower bound $\max_i K_{ii} / \LM_{ii}$.}
    \end{subfigure}
    \hfill
    \begin{subfigure}[t]{0.49\textwidth}
        \centering
        \includegraphics[width=\textwidth]{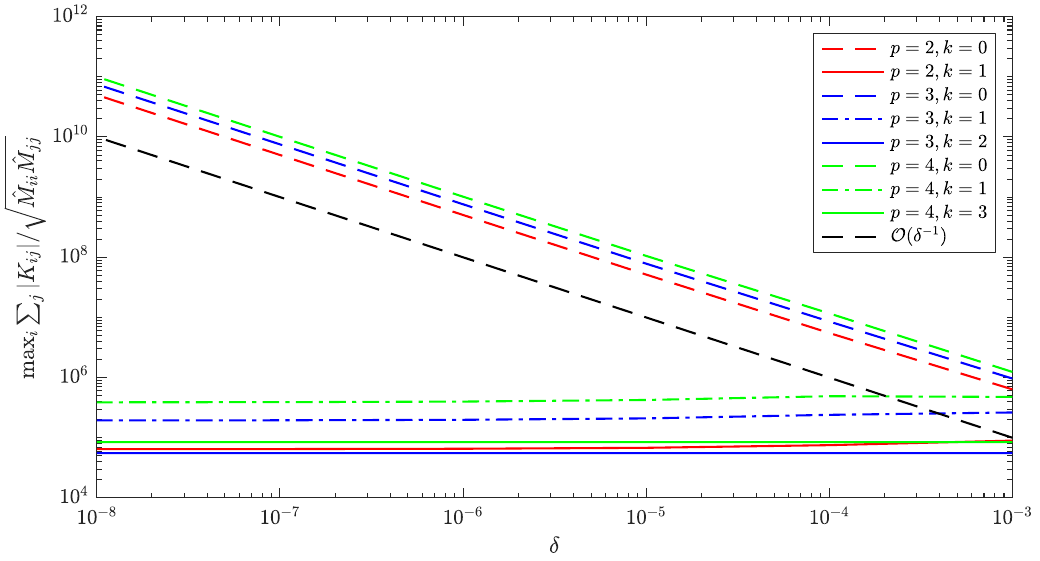}
        \caption{Upper bound $\max_{i} \sum_{j} \abs{K_{ij}} / \sqrt{\LM_{ii}\LM_{jj}}$.}
    \end{subfigure}
    \hfill
    \begin{subfigure}[b]{0.49\textwidth}
        \centering
        \includegraphics[width=\textwidth]{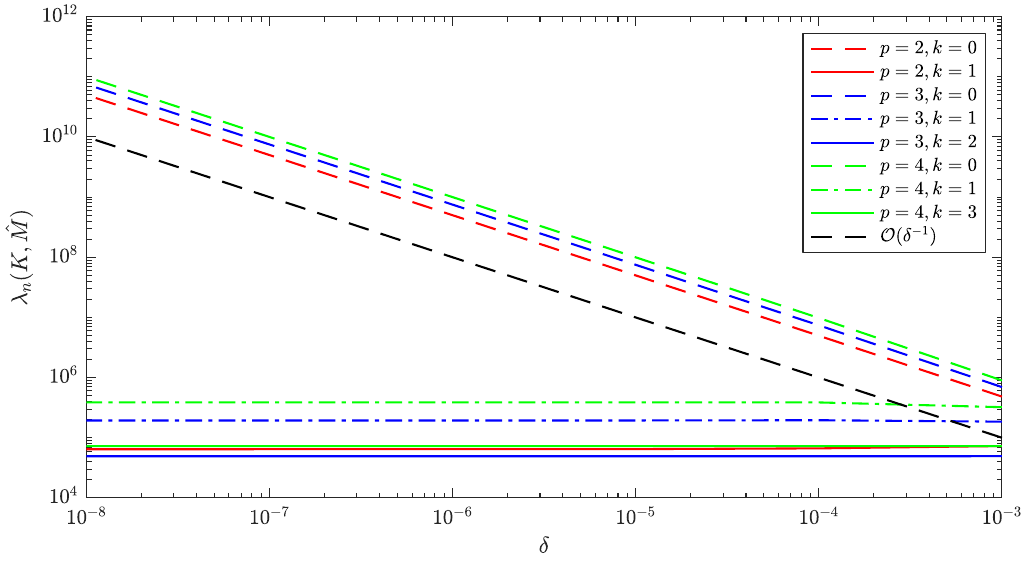}
        \caption{Largest eigenvalue with lumped mass.}
    \end{subfigure}
    \hfill
    \begin{subfigure}[b]{0.49\textwidth}
        \centering
        \includegraphics[width=\textwidth]{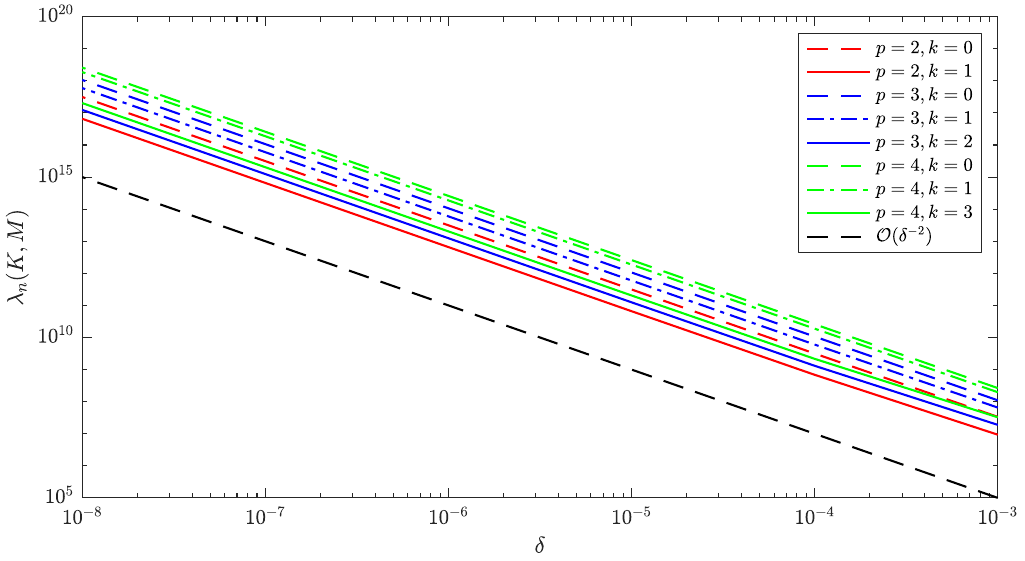}
        \caption{Largest eigenvalue with consistent mass.}
    \end{subfigure}
    \caption{Behavior of the largest eigenvalue with consistent and lumped mass matrices, the lower bound \eqref{eq:lowerbound_lambdan}, and the upper bound \eqref{eq:upperbound_lambdan_Bdiag_infnorm}, as a function of the trimming parameter $\delta$.}
    \label{fig:1Dtrimming_lambdan}
\end{figure}

\Cref{fig:1Dtrimming_lambda1} illustrates the behavior of the smallest eigenvalue of the matrix pair $(K, \LM)$ alongside the upper bound \eqref{eq:upperbound_lambda1}. The decay rate of $\min_i \hat{R}(\varphi_i)$ aligns perfectly with \eqref{eq:lambda1_1d}, and eventually, so does $\lambda_1(K, \LM)$. Notably, their asymptotic decay rate depends exclusively on the degree of the spline basis, rather than the continuity. However, contrary to the lower and upper bounds \eqref{eq:lowerbound_lambdan} and \eqref{eq:upperbound_lambdan_Bdiag_infnorm} on the largest eigenvalue, the upper bound $\min_i \hat{R}(\varphi_i)$ on the smallest eigenvalue is not tight for splines of smoothness $k \leq p-2$. Most likely, the actual eigenfunction $u_1 \in V_h$ corresponding to $\lambda_1(K, \LM)$ is a linear combination of small basis functions $\varphi_i \in \Phi^S$ and $\hat{R}(u_1)$ is significantly smaller than $\min_i \hat{R}(\varphi_i)$, although it decays at the same rate.

\begin{figure}[h!]
    \centering
    \begin{subfigure}[b]{0.49\textwidth}
        \centering
        \includegraphics[width=\textwidth]{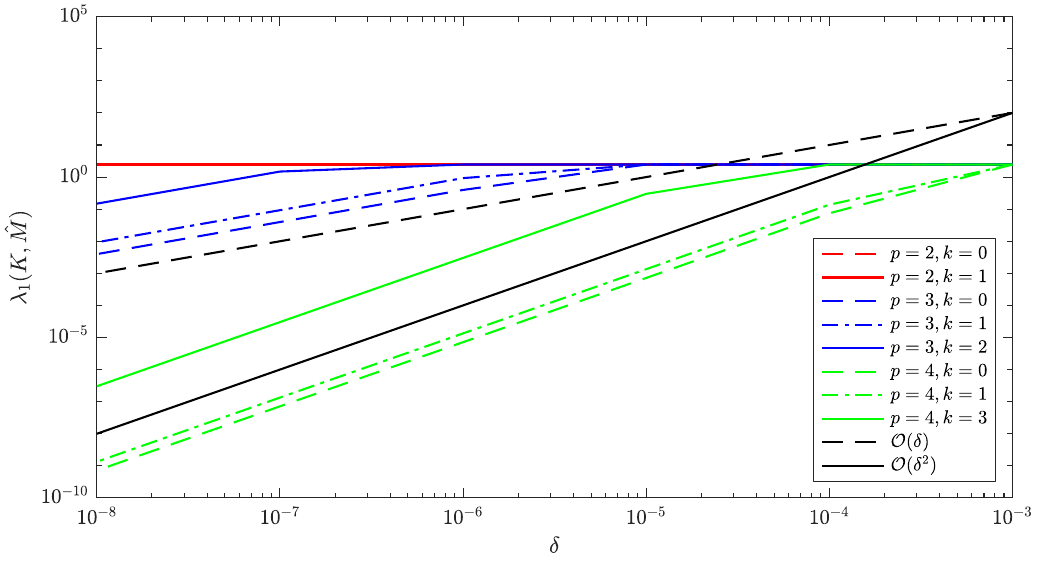}
        \caption{Smallest eigenvalue $\lambda_1(K, \LM)$.}
    \end{subfigure}
    \hfill
    \begin{subfigure}[b]{0.49\textwidth}
        \centering
        \includegraphics[width=\textwidth]{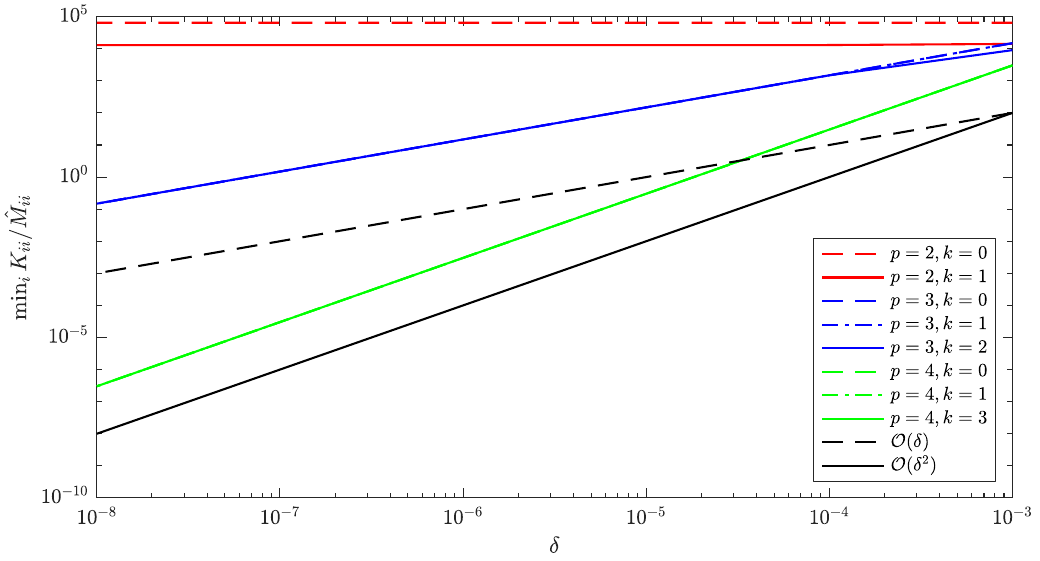}
        \caption{Upper bound $\min_i K_{ii} / \LM_{ii}$.}
    \end{subfigure}
    \caption{Behavior of the smallest eigenvalue of the matrix pair $(K, \LM)$ and the upper bound $\min_i K_{ii} / \LM_{ii}$, with respect to the trimming parameter $\delta$.}
    \label{fig:1Dtrimming_lambda1}
\end{figure}

\Cref{fig:1D_trimming_eigen} compares the B-splines that minimize (resp. maximize) $\hat{R}(\varphi_i)$ to the actual eigenfunctions minimizing (resp. maximizing) the Rayleigh quotient \eqref{eq:rayleigh_quotient}. In most cases, the eigenfunctions closely follow the basis functions maximizing or minimizing $\hat{R}(\varphi_i)$. Although the largest eigenfunction departs from the basis function maximizing $\hat{R}(\varphi_i)$ in the $C^0$ case, it is merely a linear combination of small basis functions $\varphi_i \in \Phi^S$. Having such a good agreement between eigenfunctions and basis functions is not entirely surprising since the lumped mass matrix $\hat{M}$ is basis-dependent and so are the generalized eigenvalues of $(K,\hat{M})$, contrary to those for a consistent mass.

\begin{figure}[h!]
    \centering
    \begin{subfigure}[b]{0.49\textwidth}
        \centering
        \includegraphics[width=\textwidth]{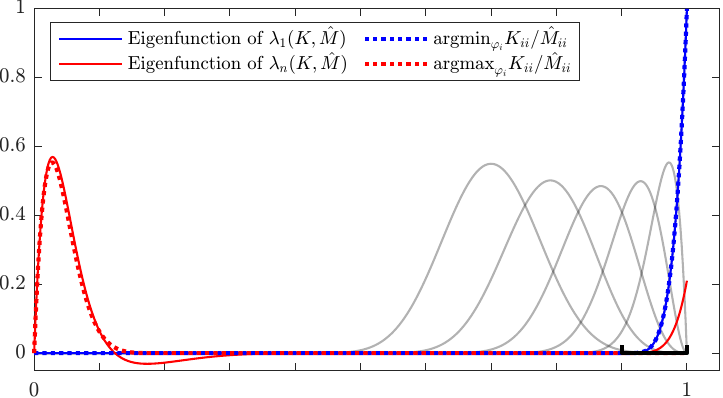}
        \caption{$C^{p-1}$ continuity.}
        \label{fig:1D_trimming_eigen_Cp-1}
    \end{subfigure}
    \hfill
    \begin{subfigure}[b]{0.49\textwidth}
        \centering
        \includegraphics[width=\textwidth]{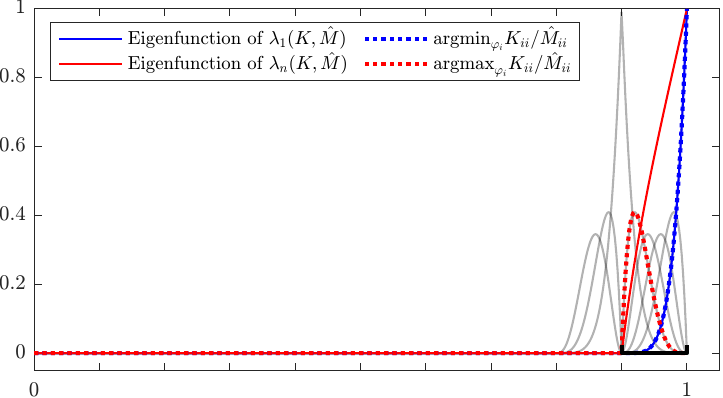}
        \caption{$C^0$ continuity.}
        \label{fig:1D_trimming_eigen_C0}
    \end{subfigure}
    \caption{Eigenfunctions corresponding to the smallest and largest eigenvalue, for degree $p = 5$ splines. The black segment represents the trimmed element.}
    \label{fig:1D_trimming_eigen}
\end{figure}

For $C^{p-1}$ continuity, the maximum eigenvalue is similar to an outlier eigenvalue for a boundary-fitted discretization. Indeed, \Cref{fig:1D_trimming_eigen} reveals that the largest eigenfunction is supported near the non-trimmed boundary and closely resembles one of the B-splines originating from the open knot vector. From this observation, Meßmer et al. \cite{messmer2021isogeometric,messmer2022efficient} suggested artificially embedding the physical domain in a larger fictitious domain in order to trim off both sides of the knot vector and remove outlier eigenvalues originating from open knot vectors.

\section{Trimming in 2D}
\label{se:2D_trimming}
Compared to the 1D case, 2D trimming configurations feature an additional layer of complexity. Indeed, unlike its one-dimensional counterpart, a curve in 2D may arbitrarily trim the support of a basis function in different ways and bad basis functions may exhibit widely different behaviors. We have identified three representative ways in which the support of a basis function is commonly trimmed. Those cases, depicted in \Cref{fig:2Dtrimming_allconfigurations} for quadratic $C^1$ splines in some idealized configurations, are not exhaustive but cover a broad range of possibilities. Essentially, a trimming curve may cut through the support of a basis function:
\begin{enumerate}[label=(\alph*), itemsep=0pt, parsep=0pt]
    \item Along one direction only (\Cref{fig:trimming2d_1direction}),
    \item In two directions but not around a corner of the support (\Cref{fig:trimming2d_angle}),
    \item In two directions and around a corner of the support (\Cref{fig:trimming2d_2directions}).
\end{enumerate}
Although a trimming curve may create completely arbitrary trimming configurations, critical cut configurations should essentially fall in either one of the configurations listed above. \Cref{fig:trimming2d_categorization}, for instance, shows a trimmed geometry where all three cases are practically encountered. In practice, those configurations often arise simultaneously. For example, the left corner of the domain in \Cref{fig:trimming2d_categorization} gives rise to both configurations \ref{fig:trimming2d_angle} and \ref{fig:trimming2d_2directions}. In such cases, we must retain the most critical configuration among all those occurring. Since all our analysis (see \Cref{corollary:eigen_laplace}) revolves around basis functions instead of single elements, all trimming configurations (including those above) are described in terms of how the trimming curve cuts through the support of basis functions. This distinction is crucial and is a sharp contrast to most existing work that focuses on individual elements (e.g. for setting up quadrature rules). Indeed, the properties of the basis, and in particular its smoothness, are just as important as the actual geometry. Thus, we must analyze configurations that combine both aspects. Configurations \ref{fig:trimming2d_1direction}-\ref{fig:trimming2d_2directions} either differ in the geometry of the trimmed element $\Omega_\ell$ or in the behavior of the basis function $\varphi(\xi_1, \xi_2) = \varphi_1(\xi_1)\varphi_2(\xi_2)$ that is supported on it. More specifically, denoting $\rho_{\min}$ and $\rho_{\max}$ the radii of the largest and smallest balls such that $\mathcal{B}_{\min} \subseteq \Omega_\ell \subseteq \mathcal{B}_{\max}$,
\begin{itemize}[noitemsep]
    \item In configuration \ref{fig:trimming2d_1direction}, $\varphi_1(\xi_1) \sim 1$ and $\varphi_2(\xi_2) \ll 1$ while $\rho_{\min} \sim \delta$ and $\rho_{\max} \sim h$,
    \item In configuration \ref{fig:trimming2d_angle}, $\varphi_1(\xi_1) \sim 1$ and $\varphi_2(\xi_2) \ll 1$ while $\rho_{\min} \sim \rho_{\max} \sim \delta$,
    \item In configuration \ref{fig:trimming2d_2directions}, $\varphi_1(\xi_1) \ll 1$ and $\varphi_2(\xi_2) \ll 1$ while $\rho_{\min} \sim \rho_{\max} \sim \delta$.
\end{itemize}

Those differences will later transpire in the analysis. The following lemma holds for any small basis function independently of the trimming configuration (\ref{fig:trimming2d_1direction}-\ref{fig:trimming2d_2directions}) and will be useful for simplifying the analysis.


\begin{figure}[h!]
    \centering
    \begin{subfigure}[t]{0.32\textwidth}
        \centering
        \includegraphics[width=\textwidth]{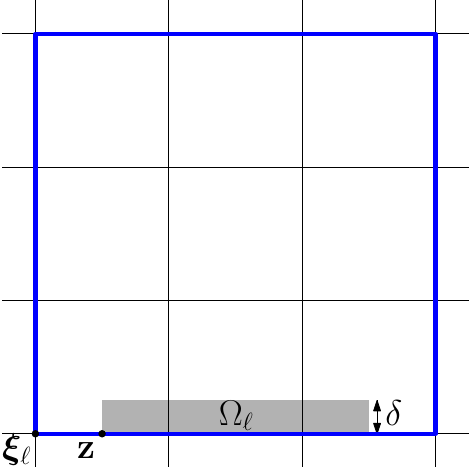}
        \caption{Trimming curve cutting along one direction only.}
        \label{fig:trimming2d_1direction}
    \end{subfigure}
    \hfill
    \begin{subfigure}[t]{0.32\textwidth}
        \centering
        \includegraphics[width=\textwidth]{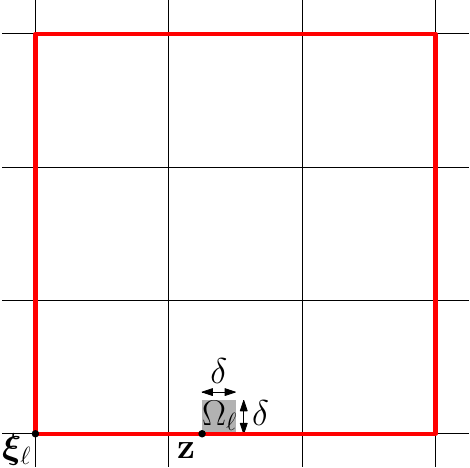}
        \caption{Trimming curve cutting in two directions but not around a corner of the support.}
        \label{fig:trimming2d_angle}
    \end{subfigure}
    \hfill
    \begin{subfigure}[t]{0.32\textwidth}
        \centering
        \includegraphics[width=\textwidth]{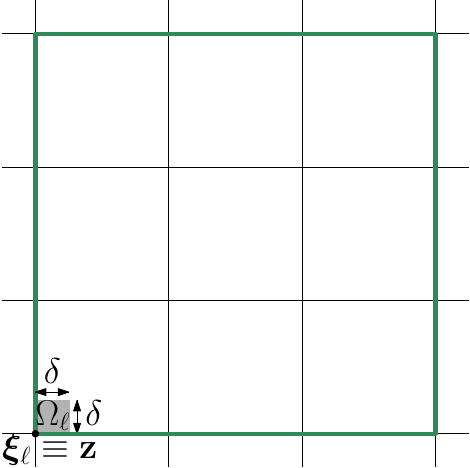}
        \caption{Trimming curve cutting in two directions and around a corner of the support.}
        \label{fig:trimming2d_2directions}
    \end{subfigure}
    \hfill
    \begin{subfigure}[b]{0.5\textwidth}
        \centering
        \includegraphics[width=\textwidth]{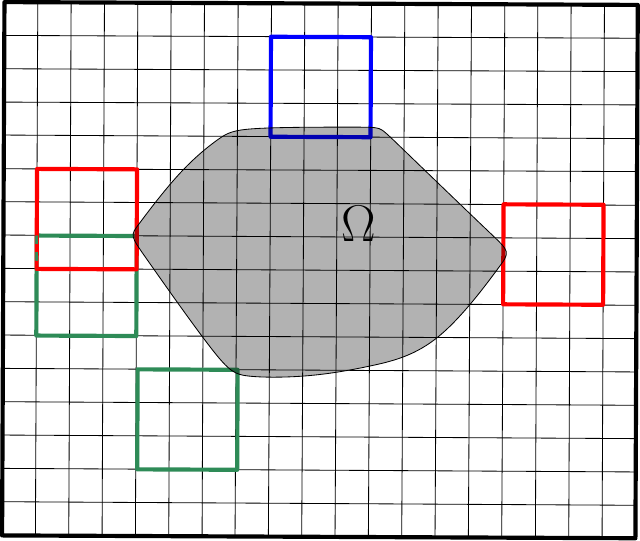}
        \caption{Support of several bad basis functions.}
        \label{fig:trimming2d_categorization}
    \end{subfigure}
    \caption{Idealized configurations for some bad basis functions in 2D.}
    \label{fig:2Dtrimming_allconfigurations}
\end{figure}

\begin{appendixlemmarep}
    \label{lem: L2_H1_eq}
    If $\varphi \in \Phi^S$, then, for any of the trimming configurations \ref{fig:trimming2d_1direction}-\ref{fig:trimming2d_2directions},
    \begin{equation}\label{eq:badlytrimmed_h1l2}
        \delta^{-2}\norm{\varphi}_{\LtwoO}^2 \lesssim \abs{\varphi}_{\HoneO}^2 \lesssim \delta^{-2}\norm{\varphi}_{\LtwoO}^2, \qquad \text{as } \delta \to 0.
    \end{equation}
\end{appendixlemmarep}
\begin{proofsketch}
    The upper bound is derived using elementwise inverse inequalities, while the lower bound is obtained by rescaling the Poincaré inequality to the unit square. A detailed proof is presented in Appendix \ref{proof:lemma_poincare}.
\end{proofsketch}
\begin{appendixproof}
    \label{proof:lemma_poincare}
    The upper bound follows from standard elementwise inverse inequalities:
    \begin{equation*}
        \abs{\varphi}_{\HoneO}^2 = \sum_{\substack{T \in \mathcal{T}_h}} |\varphi |^2_{\Hone(T \cap \Omega)} \lesssim \sum_{\substack{T \in \mathcal{T}_h}} \delta^{-2} \|\varphi \|^2_{\Ltwo(T \cap \Omega)} = \delta^{-2}\norm{\varphi}_{\LtwoO}^2.
    \end{equation*}
    The lower bound is obtained by rescaling the Poincaré inequality. In all trimming configurations~\ref{fig:trimming2d_1direction}-\ref{fig:trimming2d_2directions}, the basis function $\varphi$ is zero on the bottom edge of its active support $\Omega_\ell = \supp_\Omega(\varphi)$. Let $D_\delta$ be the linear transformation that rescales $\Omega_\ell$ to a rectangle with sides of length $\calO(h)$, i.e. $D_\delta = \mathrm{diag}(1, h/\delta)$ for configuration~\ref{fig:trimming2d_1direction} and $D_\delta = \mathrm{diag}(h/\delta, h/\delta)$ for configurations~\ref{fig:trimming2d_angle}, \ref{fig:trimming2d_2directions}. Then:
    \begin{equation*}
        \begin{split}
            \norm{\varphi}_{\LtwoO}^2 & = \int_{\Omega_\ell} \varphi^2(\bm{\xi})\mathrm{d}\bm{\xi} = \int_{D_\delta(\Omega_\ell-\bm{z})} \varphi^2(D_\delta^{-1}\bm{\zeta}+\bm{z})\abs{\mathrm{det}(D_\delta)}^{-1}\mathrm{d}\bm{\zeta}               \\
                                      & \lesssim \int_{D_\delta(\Omega_\ell-\bm{z})}\norm{D_\delta^{-1}\nabla\varphi(D_\delta^{-1}\bm{\zeta}+\bm{z})}^2\abs{\mathrm{det}(D_\delta)}^{-1}\mathrm{d}\bm{\zeta}                                          \\
                                      & = \int_{\Omega_\ell} \norm{D_\delta^{-1}\nabla \varphi(\bm{\xi})}^2\mathrm{d}\bm{\xi} \lesssim \delta^2 \int_{\Omega_\ell} \norm{\nabla \varphi(\bm{\xi})}^2\mathrm{d}\bm{\xi} \qquad\text{as } \delta \to 0.
        \end{split}
    \end{equation*}
    The last inequality is trivial for configurations~\ref{fig:trimming2d_angle}-\ref{fig:trimming2d_2directions}, where $D_\delta = \mathrm{diag}(h/\delta, h/\delta)$. In the trimming configuration~\ref{fig:trimming2d_1direction}, denote $\bm{\xi}_\ell = (\xi_{\ell_1}^{(1)}, \xi_{\ell_2}^{(2)})$ and $\varphi(\xi_1, \xi_2) = \varphi_{i_1}(\xi_1)\varphi_{i_2}(\xi_2)$. Its active support $\Omega_\ell$ is the rectangle $\Omega_\ell = \Omega_1 \times \Omega_2 = [c_1, c_2] \times [\xi_{\ell_2}^{(2)}, \xi_{\ell_2}^{(2)}+\delta]$, where $c_1 < c_2$ do not depend on $\delta$. Thus, if $\varphi_{i_2}$ is $C^{k_2}$ across the knot $\xi_{\ell_2}^{(2)}$ but not $C^{k_2+1}$, then, analogously to the one-dimensional case, we obtain:
    \begin{equation*}
        \begin{split}
             & \int_{\Omega_\ell} \norm{D_\delta^{-1}\nabla \varphi(\bm{\xi})}^2\mathrm{d}\bm{\xi} = \int_{\Omega_\ell}\left(\varphi_{i_1}'(\xi_1)\varphi_{i_2}(\xi_2)\right)^2 + \left(\frac{\delta}{h}\varphi_{i_1}(\xi_1)\varphi_{i_2}'(\xi_2)\right)^2\, \dxi_1\dxi_2                                                                                              \\
             & = \int_{c_1}^{c_2} (\varphi_{i_1}'(\xi_1))^2 \, \mathrm{d}\xi_1 \int_{\xi_{\ell_2}^{(2)}}^{\xi_{\ell_2}^{(2)}+\delta} (\varphi_{i_2}(\xi_2))^2 \, \mathrm{d}\xi_2 + \frac{\delta^2}{h^2}\int_{c_1}^{c_2} (\varphi_{i_1}(\xi_1))^2 \, \mathrm{d}\xi_1 \int_{\xi_{\ell_2}^{(2)}}^{\xi_{\ell_2}^{(2)}+\delta} (\varphi_{i_2}'(\xi_2))^2 \, \mathrm{d}\xi_2 \\
             & \sim \int_{\xi_{\ell_2}^{(2)}}^{\xi_{\ell_2}^{(2)}+\delta} (\varphi_{i_2}(\xi_2))^2 \, \mathrm{d}\xi_2 + \delta^2\int_{\xi_{\ell_2}^{(2)}}^{\xi_{\ell_2}^{(2)}+\delta} (\varphi_{i_2}'(\xi_2))^2 \, \mathrm{d}\xi_2                                                                                                                                     \\
             & \sim \int_0^\delta \xi_2^{2k_2+2} (C+\calO(\xi_2)) \, \mathrm{d}\xi_2 + \delta^2 \int_0^\delta \xi_2^{2k_2} (C+\calO(\xi_2)) \, \mathrm{d}\xi_2 \sim \delta^{2k_2+3}.
        \end{split}
    \end{equation*}
    and similarly
    \begin{equation*}
        \int_{\Omega_\ell} \norm{\nabla \varphi(\bm{\xi})}^2\mathrm{d}\bm{\xi} \sim \int_0^\delta \xi_2^{2k_2+2} (1+o(1)) \, \mathrm{d}\xi_2 + \int_0^\delta \xi_2^{2k_2} (1+o(1)) \, \mathrm{d}\xi_2 \sim \delta^{2k_2+1}.
    \end{equation*}
    This concludes the proof.
\end{appendixproof}
\begin{toappendix}
    Note that the lower bound does not always follow from a standard Poincaré inequality: in configuration~\ref{fig:trimming2d_1direction} it critically relies on the smoothness of the function. Thus, the validity of the lower bound must be studied on a case-by-case basis, whereas the upper bound, which only relies on a standard inverse inequality, holds more generally for functions $v_h \in V_h^S$.
\end{toappendix}

We must stress that Lemma~\ref{lem: L2_H1_eq} only holds for small basis functions $\varphi \in \Phi^S$ and configurations \ref{fig:trimming2d_1direction}-\ref{fig:trimming2d_2directions}. For such functions, the ratio $K_{ii}/M_{ii}$ behaves as $\delta^{-2}$ and consequently, for the largest eigenvalue with the consistent mass matrix,
\[
    \lambda_n(K,M) \geq \max_{i=1,\dots,n} \frac{\abs{\varphi_i}_{\HoneO}^2}{\norm{\varphi_i}_{\LtwoO}^2} \sim \delta^{-2}
\]
analogously to the 1D case \eqref{eq:lambdan_consistent_1d}. Once again, the upper bound follows from a standard inverse inequality, as in the proof of Lemma~\ref{lem: L2_H1_eq}. However, contrary to the lumped mass, the largest eigenvalue for the consistent mass does not always behave as the maximum ratio of diagonal entries $\max_i K_{ii}/M_{ii}$ in the most general (nearly non-physical) trimming configurations. Although the forthcoming discussion relies on Lemma~\ref{lem: L2_H1_eq}, it merely simplifies the analysis, since, for the lumped mass matrix,
\[
    \frac{K_{ii}}{\LM_{ii}} = \frac{\abs{\varphi_i}_{\HoneO}^2}{\norm{\varphi_i}_{\LoneO}} =\frac{\abs{\varphi_i}_{\HoneO}^2}{\norm{\varphi_i}_{\LtwoO}^2}\frac{\norm{\varphi_i}_{\LtwoO}^2}{\norm{\varphi_i}_{\LoneO}} \sim \delta^{-2} \frac{\norm{\varphi_i}_{\LtwoO}^2}{\norm{\varphi_i}_{\LoneO}},
\]
and the behavior of the extreme eigenvalues as $\delta \to 0$ can be deduced from a Taylor expansion of $\varphi_i$ in each variable and bypasses cumbersome derivative computations. However, if Lemma~\ref{lem: L2_H1_eq} does not apply, computing the ratio $K_{ii}/\LM_{ii}$ from scratch might be inevitable. In the next subsections, configurations \ref{fig:trimming2d_1direction}-\ref{fig:trimming2d_2directions} are analyzed separately.

\subsection*{Configuration~\ref{fig:trimming2d_1direction}} Denote $\bm{\xi}_\ell = (\xi_{\ell_1}^{(1)}, \xi_{\ell_2}^{(2)})$ the knot in configuration~\ref{fig:trimming2d_1direction} and let $\varphi_i(\xi_1, \xi_2) = \varphi_{i_1}(\xi_1)\varphi_{i_2}(\xi_2) \in \Phi^S$ be a bad basis function. Its active support $\Omega_\ell$ is the rectangle $\Omega_1 \times \Omega_2 = [c_1, c_2] \times [\xi_{\ell_2}^{(2)}, \xi_{\ell_2}^{(2)}+\delta]$, where $c_1 < c_2$ do not depend on $\delta$. Thus, if $\varphi_{i_2}$ is $C^{k_2}$ across the knot $\xi_{\ell_2}^{(2)}$ but not $C^{k_2+1}$, then, analogously to the one-dimensional case, we obtain
\[
    \begin{split}
        \norm{\varphi_i}_{\LoneO} & = \int_{\Omega_\ell} \varphi_i(\bm{\xi}) \, \mathrm{d}\bm{\xi} = \int_{c_1}^{c_2} \varphi_{i_1}(\xi_1) \, \mathrm{d}\xi_1 \int_{\xi_{\ell_2}^{(2)}}^{\xi_{\ell_2}^{(2)}+\delta} \varphi_{i_2}(\xi_2) \, \mathrm{d}\xi_2 \\
                                  & \sim \int_{\xi_{\ell_2}^{(2)}}^{\xi_{\ell_2}^{(2)}+\delta} \varphi_{i_2}(\xi_2) \, \mathrm{d}\xi_2 \sim \int_0^\delta \xi_2^{k_2+1} (C+\calO(\xi_2)) \, \mathrm{d}\xi_2 \sim \delta^{k_2+2}.
    \end{split}
\]
Similarly, for the $\Ltwo$ norm:
\[
    \norm{\varphi_i}_{\LtwoO}^2 \sim \int_0^\delta \xi_2^{2(k_2+1)} (1+o(1)) \, \mathrm{d}\xi_2 \sim \delta^{2k_2+3}.
\]
Therefore,
\[
    \frac{\abs{\varphi_i}_{\HoneO}^2}{\norm{\varphi_i}_{\LoneO}} \sim \delta^{-2} \frac{\norm{\varphi_i}_{\LtwoO}^2}{\norm{\varphi_i}_{\LoneO}} \sim \delta^{-2} \frac{\delta^{2k_2+3}}{\delta^{k_2+2}} = \delta^{k_2-1}.
\]
By \Cref{thm:spline_continuity}, there exists such a spline $\varphi_{i_2}$ for each $k_2 = \kappa_2, \dots, p_2-1$. Thus, the largest eigenvalue $\lambda_n(K, \LM)$ can blow up due to a trimming configuration as in~\ref{fig:trimming2d_1direction} only if $k_2 = 0$, resulting in
\begin{equation}\label{eq:case1_lambdan}
    \lambda_n(K, \LM) \sim
    \begin{cases*}
        \calO\left(\delta^{-1}\right) = \calO\left(\abs{\Omega_\ell}^{-1}\right) & \text{if  $\kappa_2 = 0$} \\
        \calO\left( 1 \right)                                                    & \text{otherwise}
    \end{cases*} \qquad \text{as } \delta \to 0.
\end{equation}
Conversely, the quickest decay to $0$ is obtained for $k_2 = p_2 - 1$, leading to
\begin{equation}\label{eq:case1_lambda1}
    \lambda_1(K, \LM) \leq
    \begin{cases*}
        \calO\left(1\right)                                                              & \text{if  $p_2 \leq 2$} \\
        \calO\left(\delta^{p_2 - 2}\right) = \calO\left(\abs{\Omega_\ell}^{p_2-2}\right) & \text{otherwise}
    \end{cases*} \qquad \text{as } \delta \to 0.
\end{equation}
This behavior perfectly mirrors the 1D case described in \Cref{se:1D_trimming}.

\subsection*{Configuration~\ref{fig:trimming2d_angle}}
For the trimming configuration~\ref{fig:trimming2d_angle}, the active support of a bad basis function $\varphi_i(\xi_1, \xi_2) = \varphi_{i_1}(\xi_1)\varphi_{i_2}(\xi_2) \in \Phi^S$ is $\Omega_\ell = [z_1,z_1+\delta] \times [\xi_{\ell_2}^{(2)}, \xi_{\ell_2}^{(2)}+\delta]$. Assuming $\varphi_{i_2}$ is $C^{k_2}$ across the knot $\xi_{\ell_2}^{(2)}$ but not $C^{k_2+1}$, then
\[
    \begin{split}
        \norm{\varphi_i}_{\LoneO} & = \int_{\Omega_\ell} \varphi_i(\bm{\xi}) \, \mathrm{d}\bm{\xi} = \int_{z_1}^{z_1+\delta} \varphi_{i_1}(\xi_1) \, \mathrm{d}\xi_1 \int_{\xi_{\ell_2}^{(2)}}^{\xi_{\ell_2}^{(2)}+\delta} \varphi_{i_2}(\xi_2) \, \mathrm{d}\xi_2 \\
                                  & \sim \delta \int_{\xi_{\ell_2}^{(2)}}^{\xi_{\ell_2}^{(2)}+\delta} \varphi_{i_2}(\xi_2) \, \mathrm{d}\xi_2 \sim \delta \int_0^\delta \xi_2^{k_2+1} (C+\calO(\xi_2)) \, \mathrm{d}\xi_2 \sim \delta^{k_2+3},
    \end{split}
\]
where the second step follows from a Taylor expansion of $\varphi_{i_1}$ around $z_1$. Analogously, for the $\Ltwo$ norm:
\[
    \norm{\varphi_i}_{\LtwoO}^2 \sim \delta \int_0^\delta \xi_2^{2(k_2+1)} (1+o(1)) \, \mathrm{d}\xi_2 \sim \delta^{2k_2+4}.
\]
Since this configuration raises a power of $\delta$ for both the numerator and denominator, its effect cancels out when computing the ratio and similarly as above, the largest eigenvalue $\lambda_n(K, \LM)$ can blow up in trimming configuration~\ref{fig:trimming2d_angle} only if $k_2 = 0$. This results in
\begin{equation}\label{eq:case2_lambdan}
    \lambda_n(K, \LM) \sim
    \begin{cases*}
        \calO\left(\delta^{-1}\right) = \calO\left(\abs{\Omega_\ell}^{-\frac{1}{2}}\right) & \text{if  $\kappa_2 = 0$} \\
        \calO\left( 1 \right)                                                              & \text{otherwise}
    \end{cases*} \qquad \text{as } \delta \to 0.
\end{equation}
Similarly, the smallest eigenvalue is bounded by
\begin{equation}\label{eq:case2_lambda1}
    \lambda_1(K, \LM) \leq
    \begin{cases*}
        \calO\left(1\right)                                                                        & \text{if  $p_2 \leq 2$} \\
        \calO\left(\delta^{p_2 - 2}\right) = \calO\left(\abs{\Omega_\ell}^{\frac{p_2-2}{2}}\right) & \text{otherwise}
    \end{cases*} \qquad \text{as } \delta \to 0.
\end{equation}

\subsection*{Configuration~\ref{fig:trimming2d_2directions}}
Finally, let us consider a bad basis function $\varphi_i(\xi_1, \xi_2) = \varphi_{i_1}(\xi_1)\varphi_{i_2}(\xi_2) \in \Phi^S$, whose support is trimmed according to configuration~\ref{fig:trimming2d_2directions}. Suppose that $\varphi_{i_j}$ is $C^{k_j}$ across the knot $\xi_{l_j}^{(j)}$ but not $C^{k_j+1}$ for $j=1,2$. Then, for the $\Lone$ norm,
\[
    \begin{split}
        \norm{\varphi_i}_{\LoneO} & = \int_{\Omega_\ell} \varphi_i(\bm{\xi}) \, \mathrm{d}\bm{\xi}
        \sim \int_{\xi_{\ell_1}^{(1)}}^{\xi_{\ell_1}^{(1)}+\delta} \varphi_{i_1}(\xi_1) \, \mathrm{d}\xi_1 \int_{\xi_{\ell_2}^{(2)}}^{\xi_{\ell_2}^{(2)}+\delta} \varphi_{i_2}(\xi_2) \, \mathrm{d}\xi_2 \\
                                  & \sim \int_0^\delta \xi_1^{k_1+1} \mathrm{d}\xi_1 \int_0^\delta \xi_2^{k_2+1} \mathrm{d}\xi_2 \sim \delta^{k_1+k_2+4}.
    \end{split}
\]
Similarly, for the $\LtwoO$ norm:
\[
    \norm{\varphi_i}_{\LtwoO}^2 \sim \int_0^\delta \xi_1^{2(k_1+1)}\, \mathrm{d}\xi_1 \int_0^\delta \xi_2^{2(k_2+1)}\, \mathrm{d}\xi_2 \sim \delta^{2k_1 + 2k_2 +6}.
\]
As a result,
\[
    \frac{\abs{\varphi_i}_{\HoneO}^2}{\norm{\varphi_i}_{\LoneO}} \sim \delta^{-2} \frac{\norm{\varphi_i}_{\LtwoO}^2}{\norm{\varphi_i}_{\LoneO}} \sim \delta^{-2} \frac{\delta^{2k_1+2k_2+6}}{\delta^{k_1+k_2+4}} = \delta^{k_1 + k_2}.
\]
Note that this type of badly trimmed basis function cannot cause the largest eigenvalue to blow up, but can cause the smallest eigenvalue to tend to zero. Thus, for configuration~\ref{fig:trimming2d_2directions}:
\begin{equation}\label{eq:case3_lambdan}
    \lambda_n(K, \LM) \sim \calO\left( 1 \right) \qquad \text{as } \delta \to 0,
\end{equation}
and
\begin{equation}\label{eq:case3_lambda1}
    \lambda_1(K, \LM) \leq
    \begin{cases*}
        \calO\left(1\right)                                                                                & \text{if  $p_1 = p_2 = 1$} \\
        \calO\left(\delta^{p_1+p_2 - 2}\right) = \calO\left(\abs{\Omega_\ell}^{\frac{p_1+p_2-2}{2}}\right) & \text{otherwise}
    \end{cases*} \qquad \text{as } \delta \to 0.
\end{equation}
Thus, the smallest eigenvalue tends to zero whenever $\max(p_1, p_2) > 1$. The results for the three trimming configurations \ref{fig:trimming2d_1direction}-\ref{fig:trimming2d_2directions} are summarized in \Cref{tab:config_eig}. In practice, these cases often appear simultaneously within a trimmed geometry and one should only retain the critical ones, as in the upcoming examples. For instance, configuration~\ref{fig:trimming2d_2directions} often arises in conjunction with configuration~\ref{fig:trimming2d_angle} and the largest eigenvalue may then blow up due to the latter.

\begin{table}[hpt]
    \centering
    \caption{Bounds on the smallest and largest eigenvalues of $(K, \LM)$ as $\delta \to 0$.}
    \label{tab:config_eig}
    \begin{tabular}{l l l}
        \toprule
        Config.                                                                                                                                                                              & $\lambda_1(K, \LM)$ & $\lambda_n(K, \LM)$ \\ \midrule
        \ref{fig:trimming2d_1direction}                                                                                                                                                      &
        $\leq
        \begin{cases*}
                \calO\left(1\right)                                                            & \text{if  $p_2 \leq 2$} \\
                \calO\left(\delta^{p_2 - 2}\right)=\calO\left(\abs{\Omega_\ell}^{p_2-2}\right) & \text{otherwise}
            \end{cases*}$                                                                          &
        $\sim
        \begin{cases*}
                \calO\left(\delta^{-1}\right)=\calO\left(\abs{\Omega_\ell}^{-1}\right) & \text{if  $\kappa_2 = 0$} \\
                \calO\left( 1 \right)                                                  & \text{otherwise}
            \end{cases*}$                                                                                                                               \\
        \ref{fig:trimming2d_angle}                                                                                                                                                           &
        $\leq
        \begin{cases*}
                \calO\left(1\right)                                                                      & \text{if  $p_2 \leq 2$} \\
                \calO\left(\delta^{p_2 - 2}\right)=\calO\left(\abs{\Omega_\ell}^{\frac{p_2-2}{2}}\right) & \text{otherwise}
            \end{cases*}$        &
        $\sim
        \begin{cases*}
                \calO\left(\delta^{-1}\right)=\calO\left(\abs{\Omega_\ell}^{-\frac{1}{2}}\right) & \text{if  $\kappa_2 = 0$} \\
                \calO\left( 1 \right)                                                            & \text{otherwise}
            \end{cases*}$                                                       \\
        \ref{fig:trimming2d_2directions}                                                                                                                                                     &
        $\leq
        \begin{cases*}
                \calO\left(1\right)                                                                              & \text{if  $p_1 = p_2 = 1$} \\
                \calO\left(\delta^{p_1+p_2 - 2}\right)=\calO\left(\abs{\Omega_\ell}^{\frac{p_1+p_2-2}{2}}\right) & \text{otherwise}
            \end{cases*}$ &
        $\sim \calO\left( 1 \right)$                                                                                                                                                                                                     \\
        \bottomrule
    \end{tabular}
\end{table}

The analysis carried out in Sections \ref{se:1D_trimming} and \ref{se:2D_trimming} indicates that at least quadratic splines are required for $\lambda_n(K, \LM)$ to become independent of $\delta$, unless a trimmed geometry exclusively features configuration~\ref{fig:trimming2d_2directions}. For the bi-Laplacian, the minimal degree increases to $4$ and more generally, we expect it matches the order of the differential operator. Generalizing Lemma~\ref{lem: L2_H1_eq} would allow straightforwardly extending our results to the bi-Laplacian and other high-order problems but we will not detail it here.

\subsection{Numerical experiments}
We fix here the notation and common setup for the numerical experiments that follow. The fictitious domain $\Omega_0 = (0, 1)^2$ is discretized in both directions using $N = 1/h$ subdivisions and degree $p$ splines with continuity $k = p-m$. Here, $N = 2^4$ is fixed, and we vary the trimming parameter $\delta$ on $7$ logarithmically spaced values between $10^{-2}$ and $10^{-5}$. The Laplace eigenvalue problem is solved on the trimmed domain $\Omega \subset \Omega_0$, which exemplarily shows the isolated or combined effect of the different trimming configurations~\ref{fig:trimming2d_1direction}-\ref{fig:trimming2d_2directions}.

For trimmed geometries, accurately computing the eigenvalues and eigenvectors is often quite challenging due to the ill-conditioning of system matrices. For a matrix pair $(A,B)$, the conditioning of $B$ heavily impacts the stability of eigensolvers. Following the remark in \cite[Example 5.8]{voet2025mass}, we have instead computed the eigenpairs of the equivalent matrix pair $(DAD,DBD)$, where $D=\diag(d_1,\dots,d_n)$ with $d_i=1/\sqrt{b_{ii}}$ for $i=1,\dots,n$ is a Jacobi preconditioner for $B$ \cite{de2017condition}. The eigenvalues of $(A,B)$ and $(DAD,DBD)$ are the same and their eigenvectors are simply related \cite{saad2011numerical}. Yet, the eigenpairs of the latter are computed more accurately since the conditioning of $DBD$ is oftentimes orders of magnitude smaller than the one of $B$ \cite{de2017condition}. We occasionally also relied on different eigensolvers for enhanced stability.

\begin{ex}[Configurations \ref{fig:trimming2d_1direction} and \ref{fig:trimming2d_2directions}]
    \label{ex:trimmed_square_2_sides}
    The trimmed geometry $\Omega = (0, 1-h+\delta)^2 \subset \Omega_0$ is an immediate extension of the 1D numerical experiment discussed in \Cref{sec:exp_1d} and generates the first and third trimming configurations previously analyzed. Dirichlet boundary conditions are strongly imposed on the left and bottom edges and homogeneous Neumann boundary conditions on the top and right edges. A visual representation is provided in \Cref{fig:2D_trimming_square_domain}.

    The results shown in \Cref{fig:2Dtrimming_square} perfectly align with equations \eqref{eq:case1_lambdan} and \eqref{eq:case3_lambdan}: as $\delta \to 0$, the largest eigenvalue for the lumped mass is tightly squeezed between the upper and lower bounds, scaling as $\calO(\delta^{-1})$ for $C^0$ inter-element continuity, while remaining nearly constant otherwise. Moreover, the decay rate of the smallest eigenvalue $\lambda_1(K, \hat{M})$ depends only on the degree $p$ and is asymptotically aligned with its upper bound $\min_i K_{ii}/\LM_{ii}$, which scales as $\calO(\delta^{2p-2})$, as predicted by equation \eqref{eq:case3_lambda1}. Additionally, the largest eigenvalue with a consistent mass matrix scales as $\calO(\delta^{-2})$, independently of the basis degree and continuity.

    Based on the analysis carried out in \Cref{se:2D_trimming}, we anticipate the following results: For $C^{p-1}$ continuity, configuration~\ref{fig:trimming2d_2directions} is responsible for the minimum eigenvalue while the maximum is independent of $\delta$. For $C^0$ continuity, configuration~\ref{fig:trimming2d_2directions} is still responsible for the minimum eigenvalue while the maximum eigenvalue originates from configuration~\ref{fig:trimming2d_1direction}. Figures \ref{fig:2Dtrimming_square_eigenk2} and \ref{fig:2Dtrimming_square_eigenk0} confirm our speculations by comparing the B-splines achieving the minimum and maximum ratio $K_{ii}/\LM_{ii}$ with the actual eigenfunctions corresponding to the extreme eigenvalues, for both $C^{p-1}$ and $C^0$ splines. Notably, the basis function attaining the minimum ratio closely resembles the eigenfunction corresponding to the spurious minimum eigenvalue $\lambda_1(K, \LM)$, and both appear unaffected by the continuity. In contrast, the spline attaining the maximum ratio $K_{ii}/\LM_{ii}$ has untrimmed support for $C^{p-1}$ continuity, whereas for $C^0$ continuity, its support is trimmed in one direction, as in configuration~\ref{fig:trimming2d_1direction}. The same behavior is observed for the eigenfunction corresponding to $\lambda_n(K, \LM)$, although expressed as a linear combination of several bad basis functions from configuration~\ref{fig:trimming2d_1direction}.

    \begin{ex}[Configurations \ref{fig:trimming2d_angle} and \ref{fig:trimming2d_2directions}]
        \label{ex:house_c2_c3}
        As physical domain we consider the house-like structure depicted in \Cref{fig:2D_trimming_house_c2_c3} whose ridge cuts through a horizontal grid line forming two symmetric triangles and roughly reproducing configurations~\ref{fig:trimming2d_angle} and \ref{fig:trimming2d_2directions}. Dirichlet boundary conditions are strongly imposed on the bottom edge and Neumann boundary conditions on the rest of the boundary. On the one hand, configuration \ref{fig:trimming2d_2directions} is critical for the minimum ratio and produces an $\calO(\delta^{2p-2})$ dependency on the trimming parameter. On the other hand, configuration \ref{fig:trimming2d_angle} is critical for the largest ratio, which blows up if $k_2=0$. Those results are confirmed in \Cref{fig:2Dtrimming_ex4_geo2}. For $C^{p-1}$ splines, \Cref{fig:2Dtrimming_ex4_eigenk2} indicates that the eigenfunction corresponding the smallest eigenvalue is essentially a linear combination of the two bad basis functions from configuration \ref{fig:trimming2d_2directions}, while the one corresponding the largest eigenvalue is supported near the non-trimmed boundary. This last observation is reminiscent of the 1D case in \Cref{fig:1D_trimming_eigen_Cp-1}. Our discussion for the minimum ratio also holds for the $C^0$ case in \Cref{fig:2Dtrimming_ex4_eigenk0}. Moreover, the maximum ratio is indeed attained for a basis function from configuration \ref{fig:trimming2d_angle} while the actual eigenfunction is again a linear combination of basis functions from the same configuration.

        \begin{figure}[p]
            \centering
            \includegraphics[width=\textwidth]{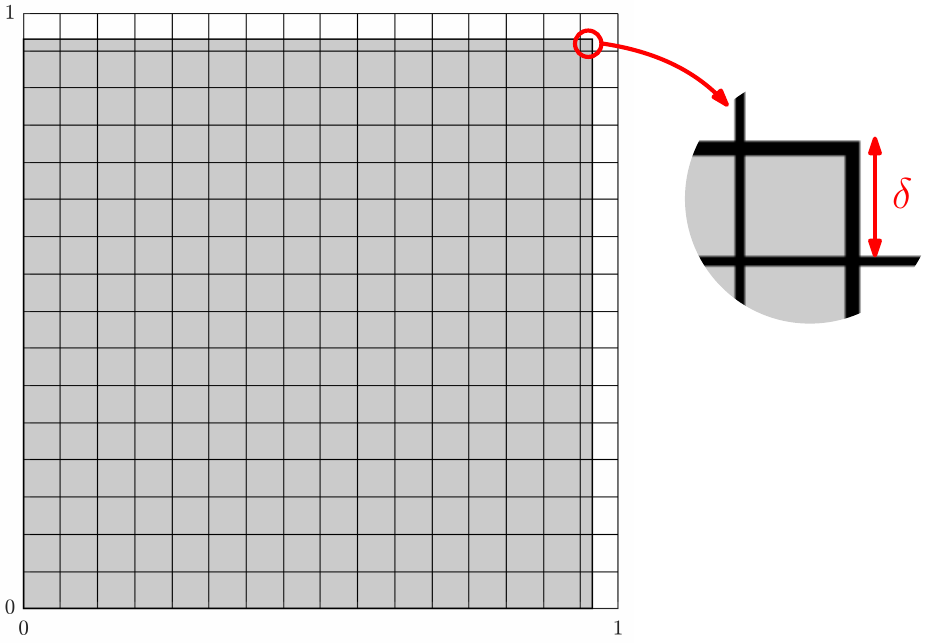}
            \caption{Fictitious (white) and trimmed (grey) domain for Example~\ref{ex:trimmed_square_2_sides}.}
            \label{fig:2D_trimming_square_domain}
        \end{figure}

        \begin{figure}[p]
            \centering
            \includegraphics[width=\textwidth]{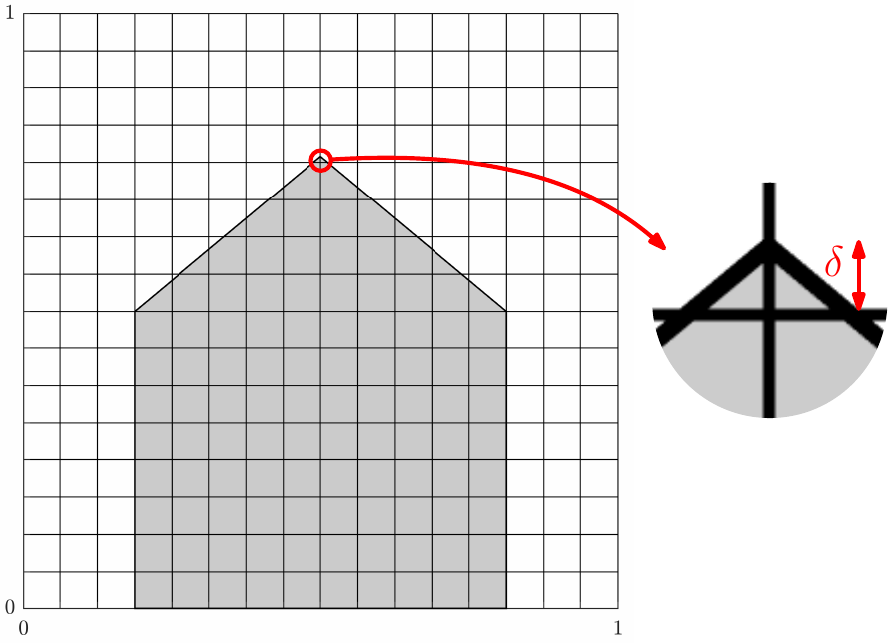}
            \caption{Fictitious (white) and trimmed (grey) domain for Example~\ref{ex:house_c2_c3}.}
            \label{fig:2D_trimming_house_c2_c3}
        \end{figure}

        \begin{figure}[p]
            \centering
            \begin{subfigure}[t]{0.49\textwidth}
                \centering
                \includegraphics[width=\textwidth]{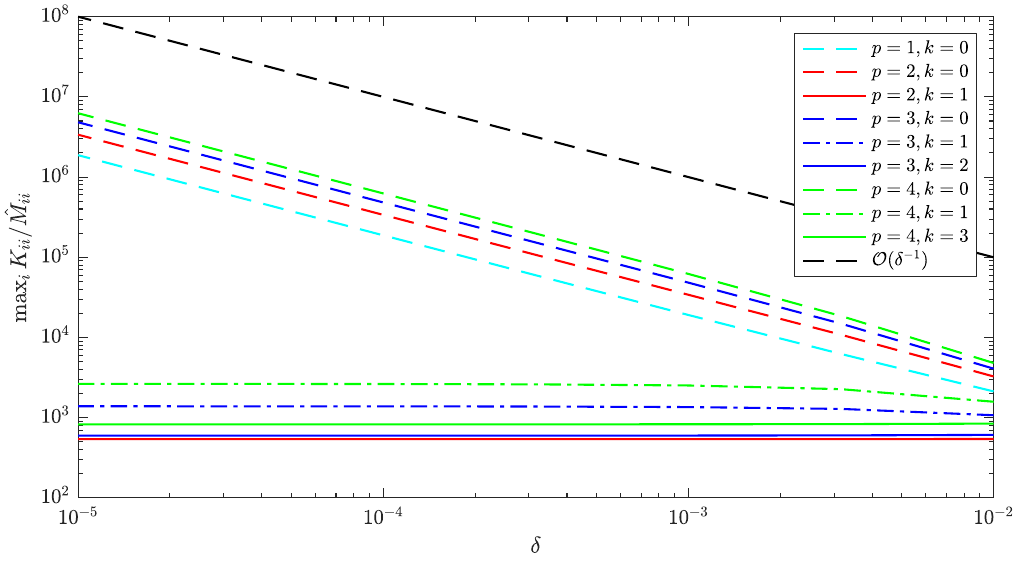}
                \caption{$\max_i K_{ii} / \LM_{ii}$.}
            \end{subfigure}
            \hfill
            \begin{subfigure}[t]{0.49\textwidth}
                \centering
                \includegraphics[width=\textwidth]{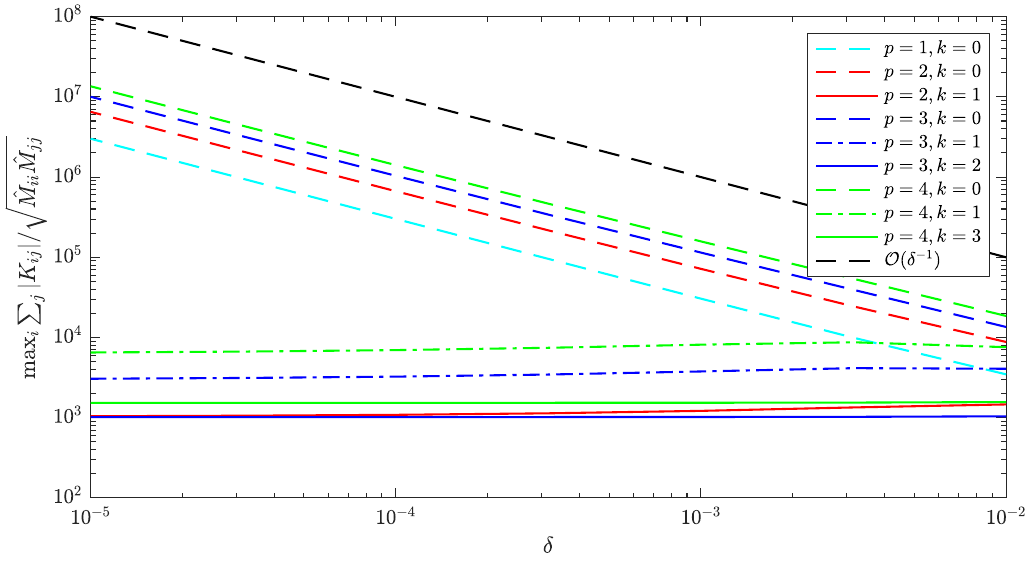}
                \caption{$\max_{i} \sum_{j} \abs{K_{ij}} / \sqrt{\LM_{ii}\LM_{jj}}$.}
            \end{subfigure}
            \hfill
            \begin{subfigure}[t]{0.49\textwidth}
                \centering
                \includegraphics[width=\textwidth]{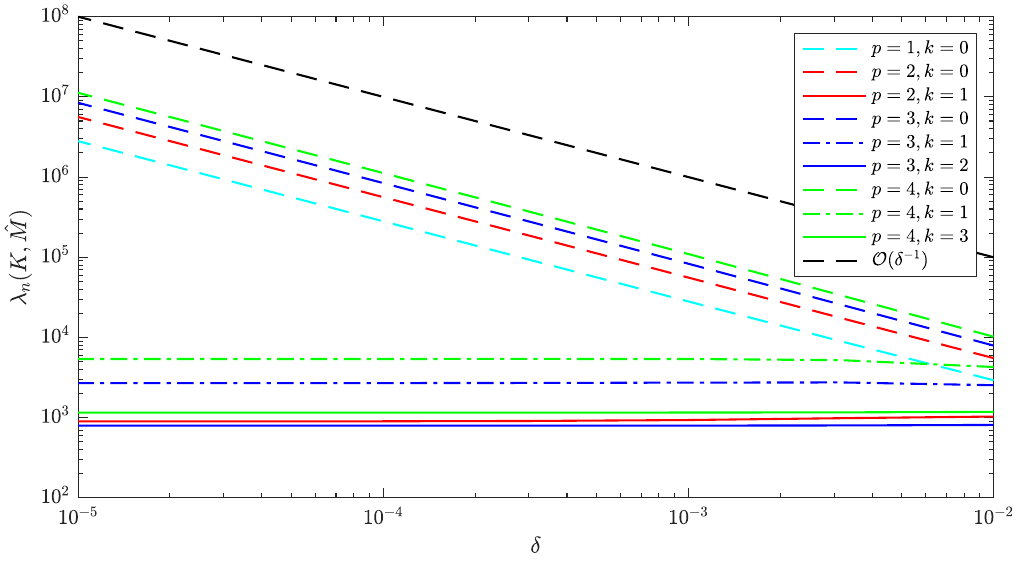}
                \caption{Largest eigenvalue with lumped mass.}
            \end{subfigure}
            \hfill
            \begin{subfigure}[t]{0.49\textwidth}
                \centering
                \includegraphics[width=\textwidth]{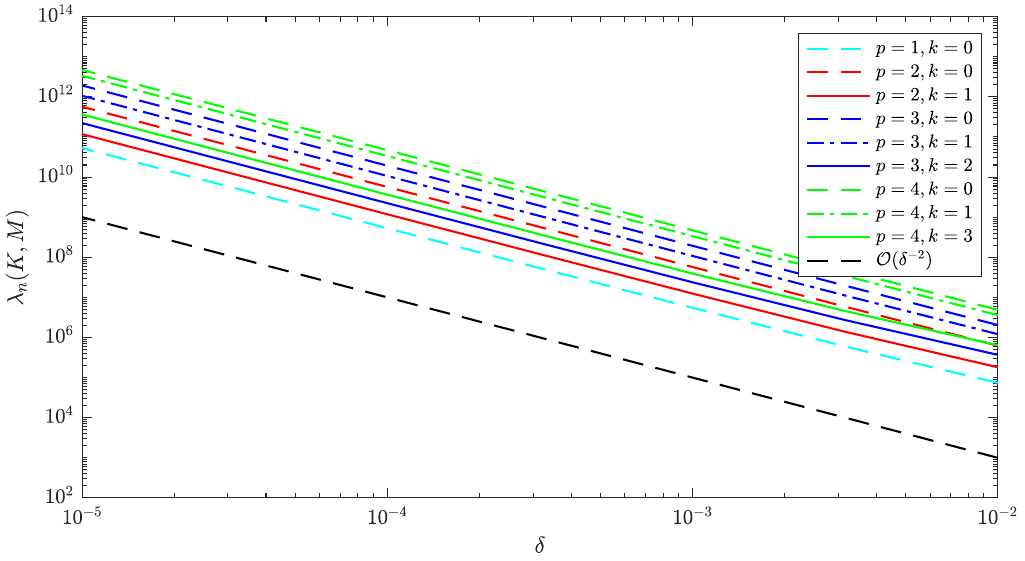}
                \caption{Largest eigenvalue with consistent mass.}
            \end{subfigure}
            \hfill
            \begin{subfigure}[t]{0.49\textwidth}
                \centering
                \includegraphics[width=\textwidth]{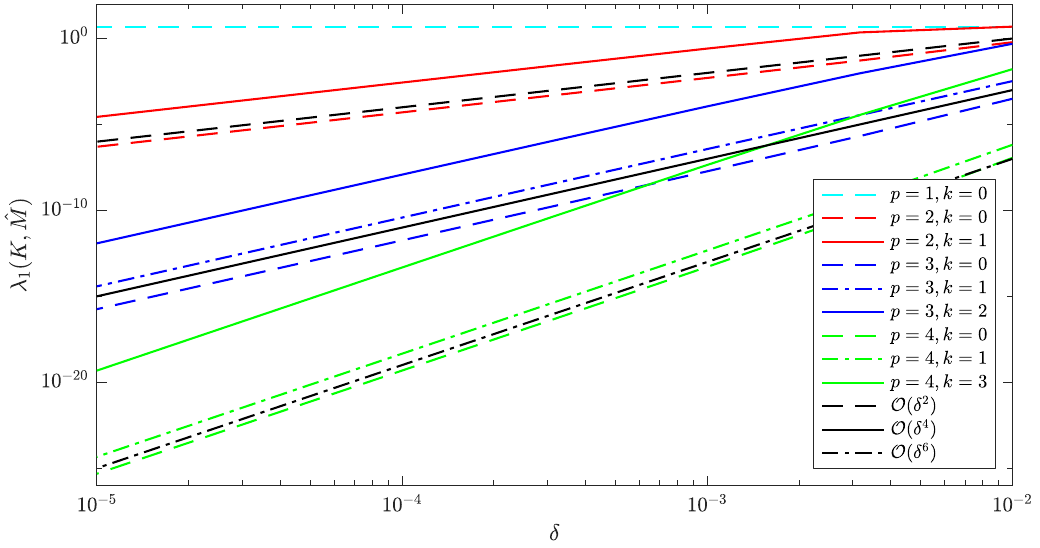}
                \caption{Smallest eigenvalue with lumped mass.}
            \end{subfigure}
            \hfill
            \begin{subfigure}[t]{0.49\textwidth}
                \centering
                \includegraphics[width=\textwidth]{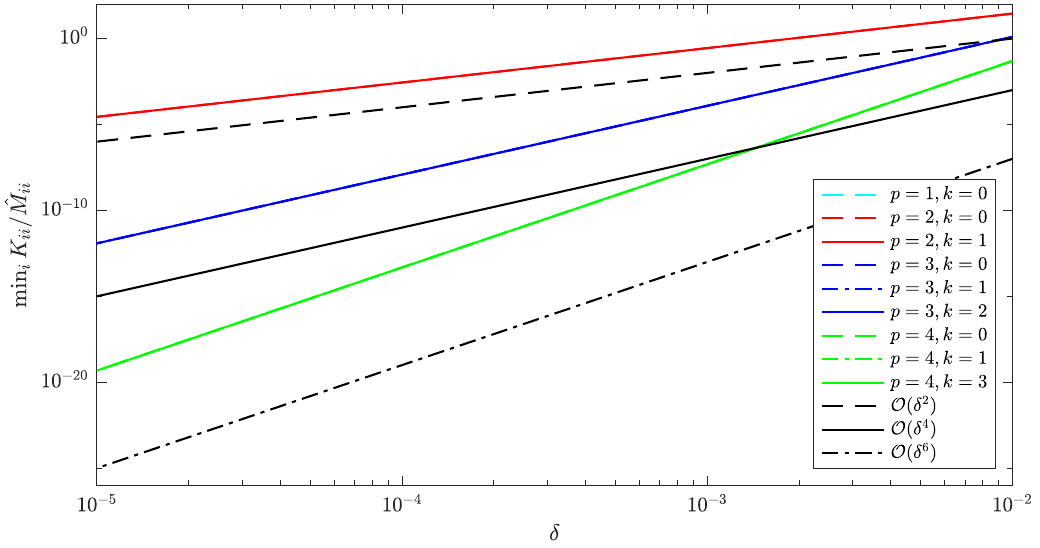}
                \caption{$\min_i K_{ii} / \LM_{ii}$.}
            \end{subfigure}
            \begin{subfigure}[t]{0.49\textwidth}
                \centering
                \includegraphics[width=0.55\textwidth]{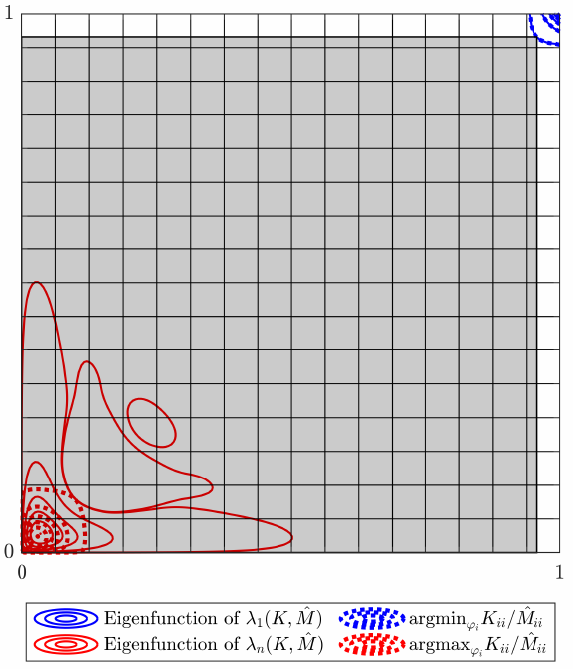}
                \caption{Degree $p = 3$ splines and $C^{p-1}$ continuity:  B-splines attaining the extreme values of $K_{ii} / \LM_{ii}$, and eigenfunctions corresponding to extreme eigenvalues, with trimming parameter $\delta=10^{-4}$.}
                \label{fig:2Dtrimming_square_eigenk2}
            \end{subfigure}
            \hfill
            \begin{subfigure}[t]{0.49\textwidth}
                \centering
                \includegraphics[width=0.55\textwidth]{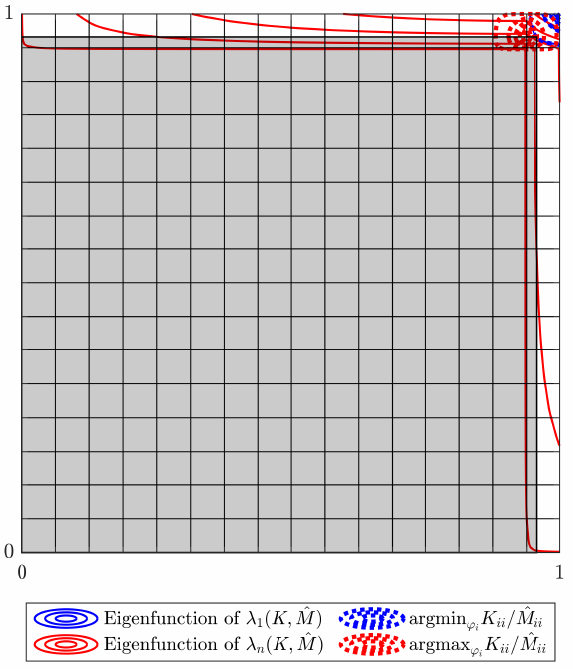}
                \caption{Degree $p = 3$ splines and $C^{0}$ continuity:  B-splines attaining the extreme values of $K_{ii} / \LM_{ii}$, and eigenfunctions corresponding to extreme eigenvalues, with trimming parameter $\delta=10^{-4}$.}
                \label{fig:2Dtrimming_square_eigenk0}
            \end{subfigure}
            \caption{Example~\ref{ex:trimmed_square_2_sides} of 2D trimming.}
            \label{fig:2Dtrimming_square}
        \end{figure}
    \end{ex}

    \begin{figure}[p]
        \centering
        \begin{subfigure}[t]{0.49\textwidth}
            \centering
            \includegraphics[width=\textwidth]{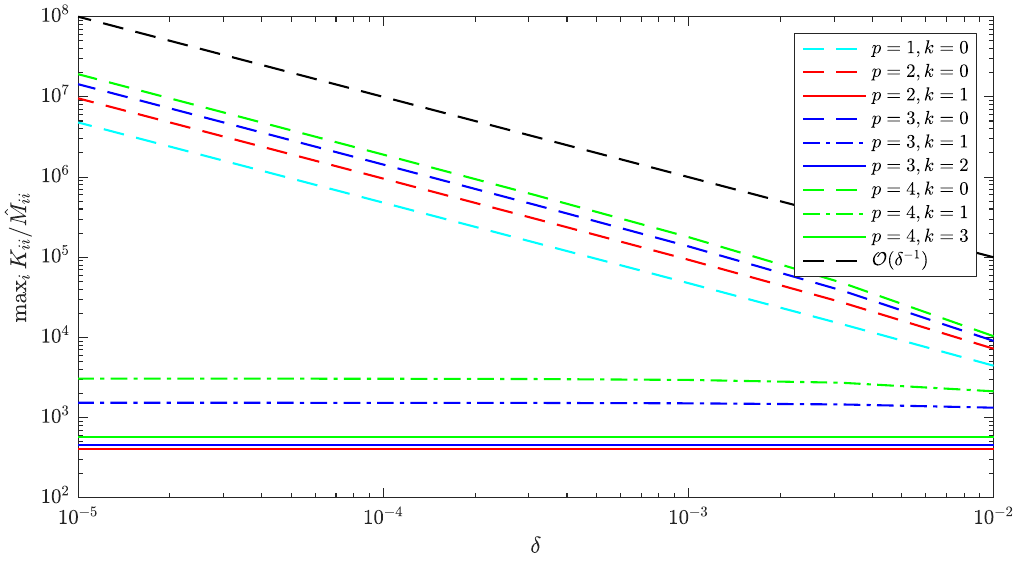}
            \caption{$\max_i K_{ii} / \LM_{ii}$.}
        \end{subfigure}
        \hfill
        \begin{subfigure}[t]{0.49\textwidth}
            \centering
            \includegraphics[width=\textwidth]{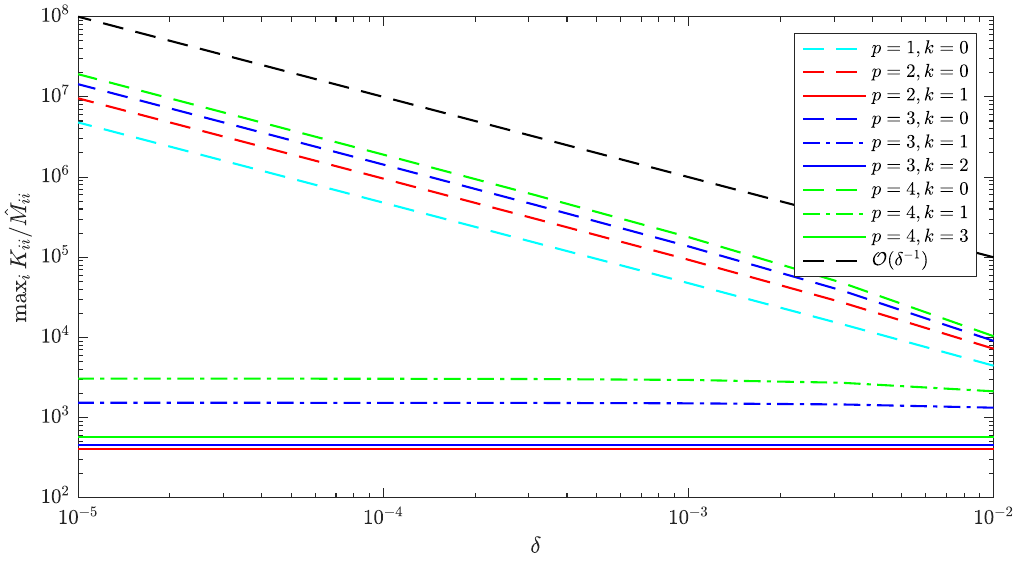}
            \caption{$\max_{i} \sum_{j} \abs{K_{ij}} / \sqrt{\LM_{ii}\LM_{jj}}$.}
        \end{subfigure}
        \hfill
        \begin{subfigure}[t]{0.49\textwidth}
            \centering
            \includegraphics[width=\textwidth]{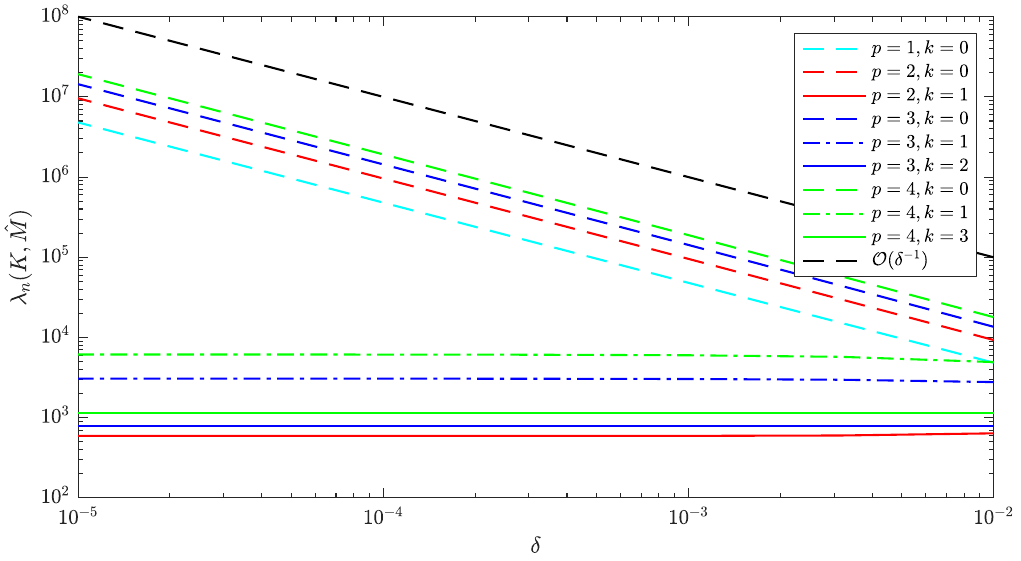}
            \caption{Largest eigenvalue with lumped mass.}
        \end{subfigure}
        \hfill
        \begin{subfigure}[t]{0.49\textwidth}
            \centering
            \includegraphics[width=\textwidth]{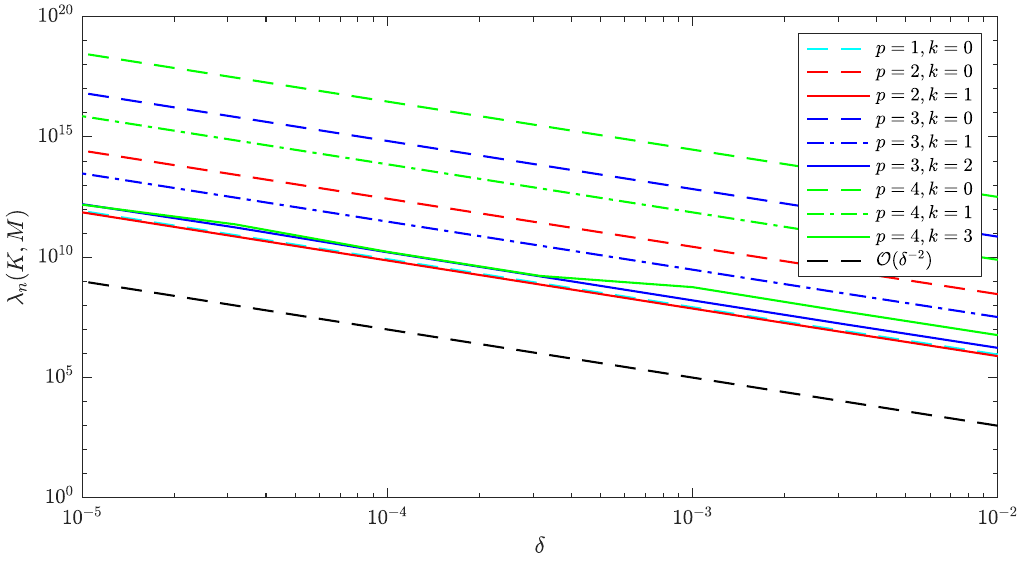}
            \caption{Largest eigenvalue with consistent mass.}
        \end{subfigure}
        \hfill
        \begin{subfigure}[t]{0.49\textwidth}
            \centering
            \includegraphics[width=\textwidth]{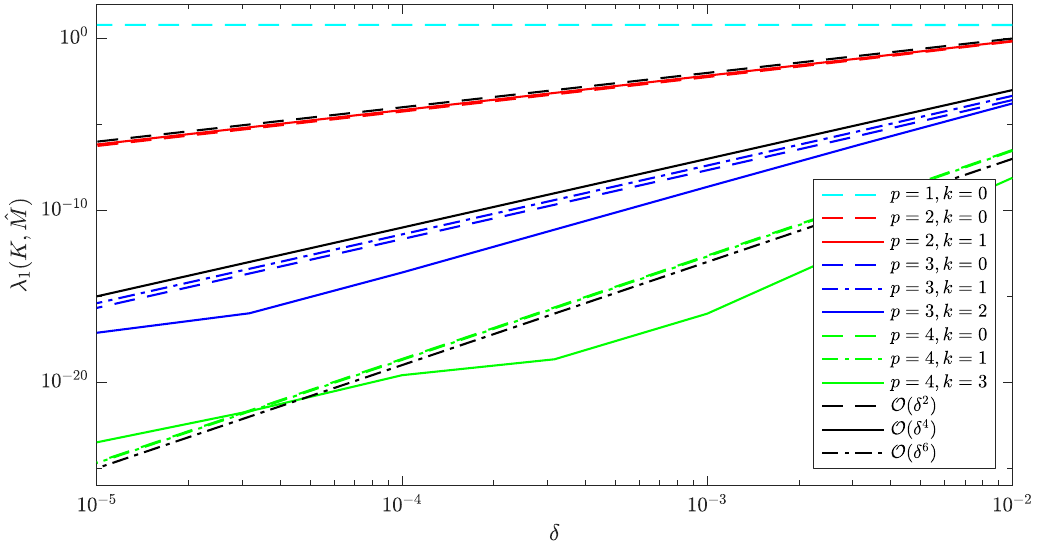}
            \caption{Smallest eigenvalue with lumped mass.}
        \end{subfigure}
        \hfill
        \begin{subfigure}[t]{0.49\textwidth}
            \centering
            \includegraphics[width=\textwidth]{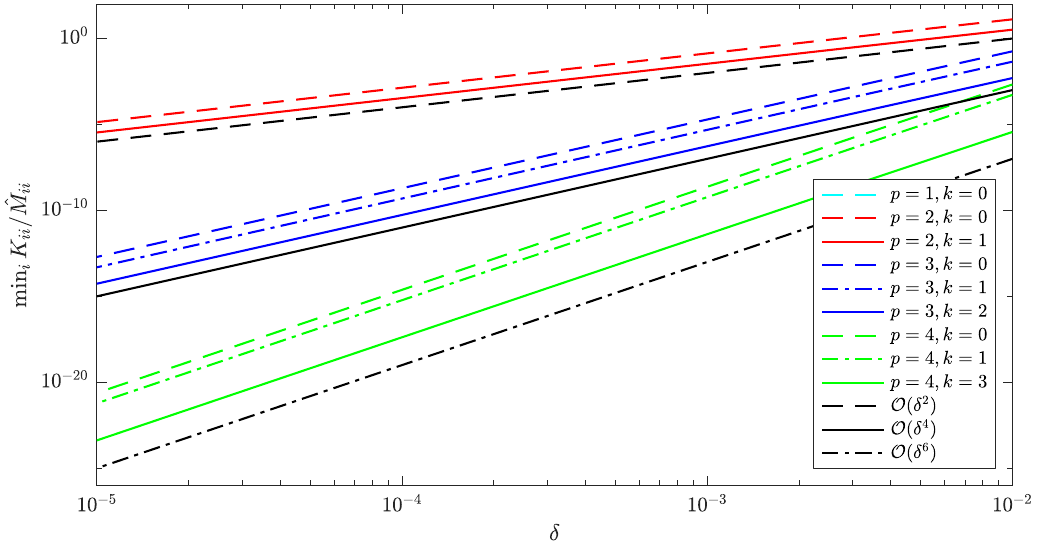}
            \caption{$\min_i K_{ii} / \LM_{ii}$.}
        \end{subfigure}
        \begin{subfigure}[t]{0.49\textwidth}
            \centering
            \includegraphics[width=0.7\textwidth]{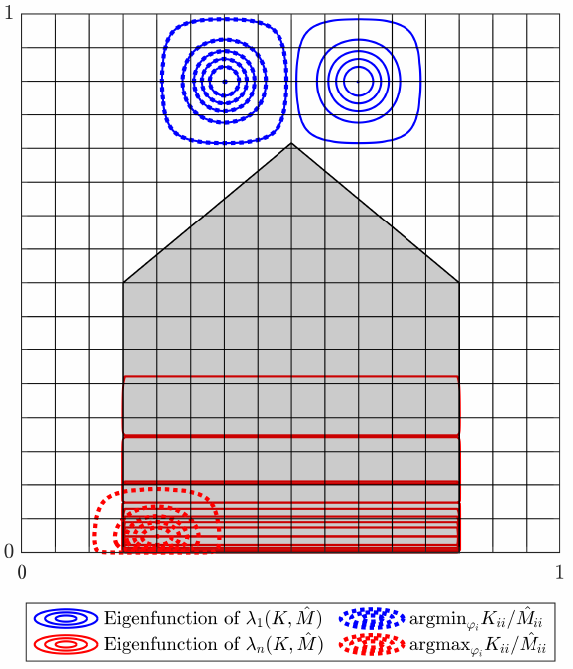}
            \caption{Degree $p = 3$ splines and $C^{p-1}$ continuity:  B-splines attaining the extreme values of $K_{ii} / \LM_{ii}$, and eigenfunctions corresponding to extreme eigenvalues, with trimming parameter $\delta=10^{-4}$.}
            \label{fig:2Dtrimming_ex4_eigenk2}
        \end{subfigure}
        \hfill
        \begin{subfigure}[t]{0.49\textwidth}
            \centering
            \includegraphics[width=0.7\textwidth]{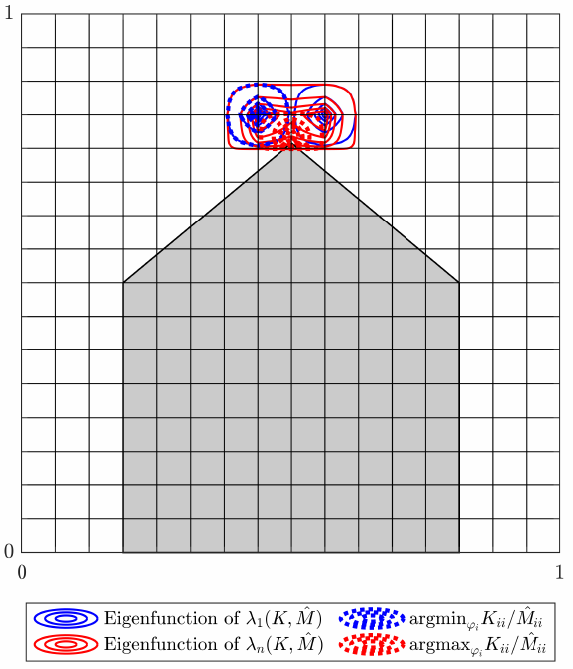}
            \caption{Degree $p = 3$ splines and $C^{0}$ continuity:  B-splines attaining the extreme values of $K_{ii} / \LM_{ii}$, and eigenfunctions corresponding to extreme eigenvalues, with trimming parameter $\delta=10^{-4}$.}
            \label{fig:2Dtrimming_ex4_eigenk0}
        \end{subfigure}
        \caption{Example~\ref{ex:house_c2_c3} of 2D trimming.}
        \label{fig:2Dtrimming_ex4_geo2}
    \end{figure}
\end{ex}

\begin{ex}[Configuration~\ref{fig:trimming2d_angle}]
    \label{ex:house_c2}
    In order to isolate configuration \ref{fig:trimming2d_angle} from \ref{fig:trimming2d_2directions}, the house's basis is enlarged such that a vertical grid line no longer cuts through the roof's ridge. Boundary conditions are imposed analogously to the previous example. While the maximum ratio remains mostly unaltered, the minimum ratio should now behave as $\calO(\delta^{p_2-2})$ for $p_2 \geq 2$ and $\calO(1)$ otherwise. Figure \ref{fig:2Dtrimming_ex5_geo2} confirms it. However, the smallest eigenvalue apparently decays much faster than its upper bound given by the minimum ratio. Figures \ref{fig:2Dtrimming_ex5_eigenk2} and \ref{fig:2Dtrimming_ex5_eigenk2_rotated} show the reason for this discrepancy: there exists a linear combination of bad basis functions that attains a significantly smaller minimum. This cannot be predicted by analyzing individual basis functions and is a clear limitation of our strategy. Although none of the basis functions vanish at the roof's ridge in the $\xi_1$ direction, it may be possible for a linear combination thereof, as visible from the perspective shown in \Cref{fig:2Dtrimming_ex5_eigenk2_rotated}. Also note that, despite of all the measures we took, the eigensolver sometimes failed to converge for the consistent mass, which is the reason for the discontinued lines in \Cref{fig:Trimming2D_ex5_lambdan_consistent}. Moreover, the computation of the smallest eigenvalue is particularly prone to numerical round-off errors. It may occasionally feature irregular jumps and its convergence is practically limited by machine precision.

\end{ex}

\begin{remark}
    The previous example shows that although the upper bound on the smallest eigenvalue behaves as predicted, the smallest eigenvalue itself may actually decay much faster. In such cases, the smallest eigenfunction is a linear combination of small basis functions. We have already provided graphical justification for this claim in \Cref{fig:2Dtrimming_ex5_eigenk2_rotated} and one may also verify it numerically. As a matter of fact, given that the upper bound
    \begin{equation}
        \label{eq: improved_bound}
        \min_{\substack{v_h \in V_h \\ v_h \neq 0}} \hat{R}(v_h) \leq \min_{\substack{v_h \in U_h \\ v_h \neq 0}} \hat{R}(v_h)
    \end{equation}
    holds for any subspace $U_h \subseteq V_h$, if the smallest eigenfunction is a linear combination of small basis functions, the upper bound obtained by choosing $U_h = V_h^S$ should accurately describe the behavior of the smallest eigenvalue. We could indeed verify this numerically on all the examples we tested and summarized the results in \Cref{app: improved_bound}. Note that computing the upper bound in \eqref{eq: improved_bound} boils down to computing the smallest eigenvalue of $(K^S,\hat{M}^S)$, where $K^S$ and $\hat{M}^S$ are the submatrices of the stiffness and lumped mass matrix corresponding to the small basis functions. Since those functions only form a small subset of all basis functions, computing the upper bound only introduces a minor computational overhead. However, the improved upper bound does not immediately yield \emph{analytical} estimates, which was the primary focus of this research.
\end{remark}

\begin{ex}[Configurations \ref{fig:trimming2d_angle} and \ref{fig:trimming2d_2directions}]
    \label{ex:rotated_square}
    We now consider as physical domain a square centered at $(0.5,0.5)$ and rotated $45$ degrees. Pure homogeneous Neumann boundary conditions are prescribed, along with the additional condition $\int_\Omega u = 0$ to ensure uniqueness of the solution. The condition $\int_\Omega u = \int_\Omega 1 \cdot u = 0$ means that we restrict the search space to the orthogonal complement of the kernel of the stiffness (spanned by the constant function $v=1$) and merely deflates the spectrum from the zero eigenvalue.

    This example and variants thereof were first introduced in \cite{de2017condition} and have become over the years a classical benchmark for stability and condition number analyses (see e.g. \cite{leidinger2020explicit,buffa2020minimal,voet2025mass}). Each corner of the square intersects a grid line forming a triangle of height $\delta$, which will serve as trimming measure (see \Cref{fig:2D_trimming_rotatedsquare}). Another grid line cuts through the triangle symmetrically, thereby roughly generating configurations \ref{fig:trimming2d_2directions} and \ref{fig:trimming2d_angle} discussed above. The former is responsible for the $\calO(\delta^{2p-2})$ decay of the smallest eigenvalue while the latter causes the largest eigenvalue to blow up if the regularity is $C^0$ in any direction, thereby generalizing Example \ref{ex:house_c2_c3}. Both trends are confirmed in \Cref{fig:2Dtrimming_rotated_square_geo1}. The smallest eigenvalue closely follows its upper bound given by the minimum ratio while the largest eigenvalue behaves (up to a constant) as the maximum ratio. Once again, the bounds for the latter are exceedingly tight even for this irregularly trimmed geometry. As in the other examples, the corresponding eigenfunctions are generally expressed as linear combinations of basis functions belonging to similar trimming configurations. However, since critical trimming configurations are essentially met all along the boundary, we do not plot the associated functions for clarity.

    \begin{figure}[p]
        \centering
        \includegraphics[width=\textwidth]{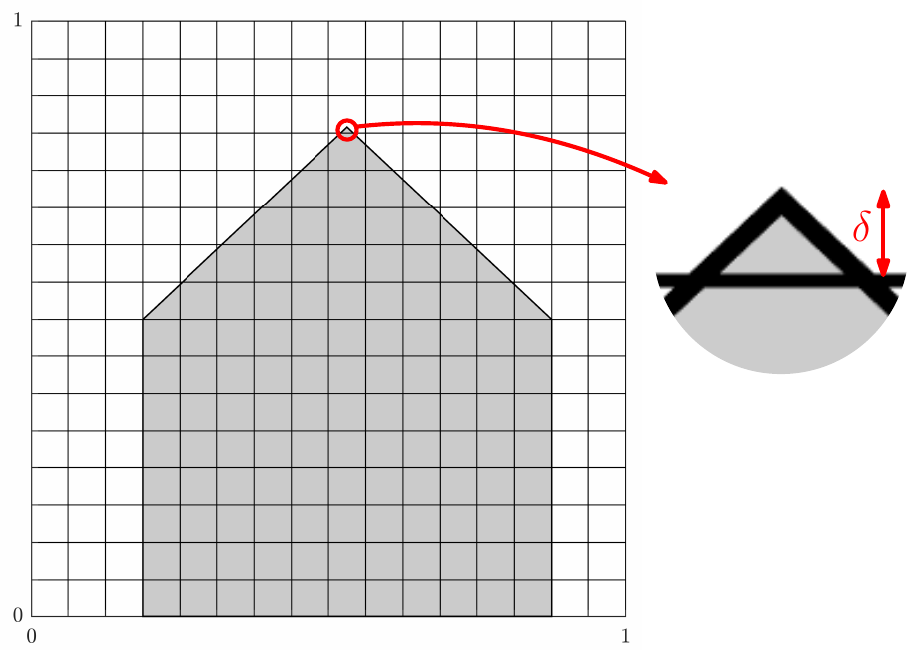}
        \caption{Fictitious (white) and trimmed (grey) domain for Example~\ref{ex:house_c2}.}
        \label{fig:2D_trimming_ex5}
    \end{figure}

    \begin{figure}[p]
        \centering
        \includegraphics[width=\textwidth]{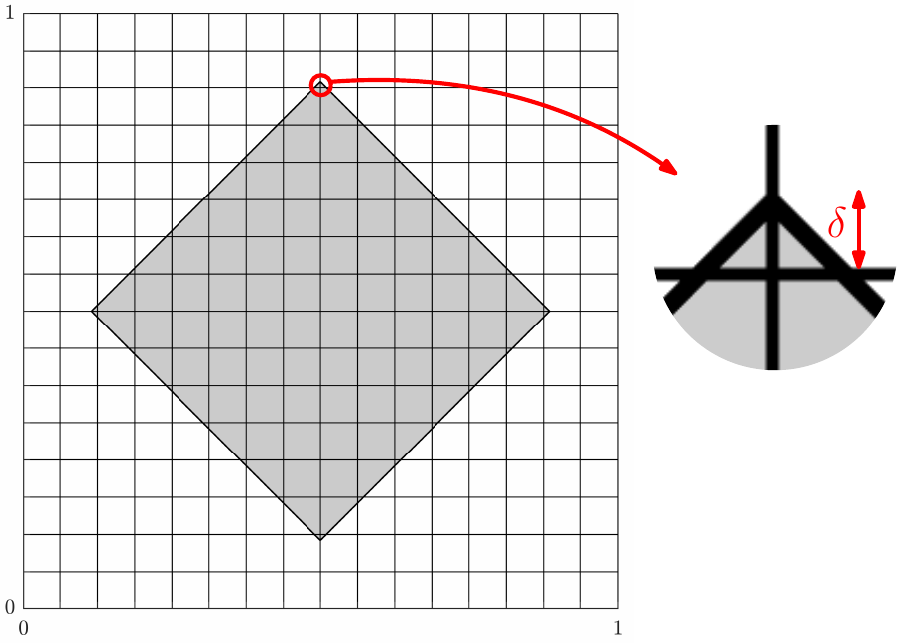}
        \caption{Fictitious (white) and trimmed (grey) domain for Example~\ref{ex:rotated_square}.}
        \label{fig:2D_trimming_rotatedsquare}
    \end{figure}

    \begin{figure}[p]
        \centering
        \begin{subfigure}[t]{0.49\textwidth}
            \centering
            \includegraphics[width=\textwidth]{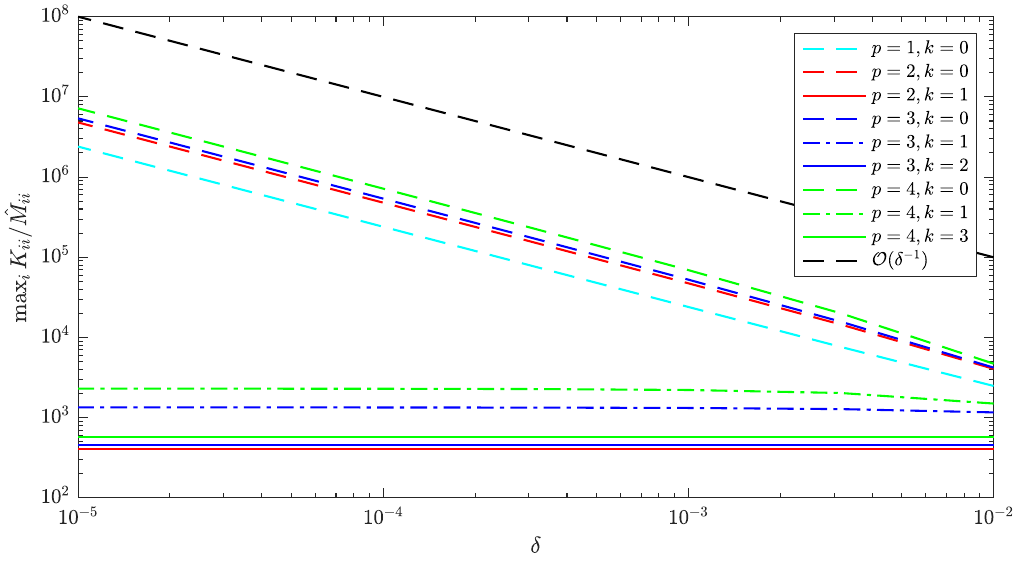}
            \caption{$\max_i K_{ii} / \LM_{ii}$.}
        \end{subfigure}
        \hfill
        \begin{subfigure}[t]{0.49\textwidth}
            \centering
            \includegraphics[width=\textwidth]{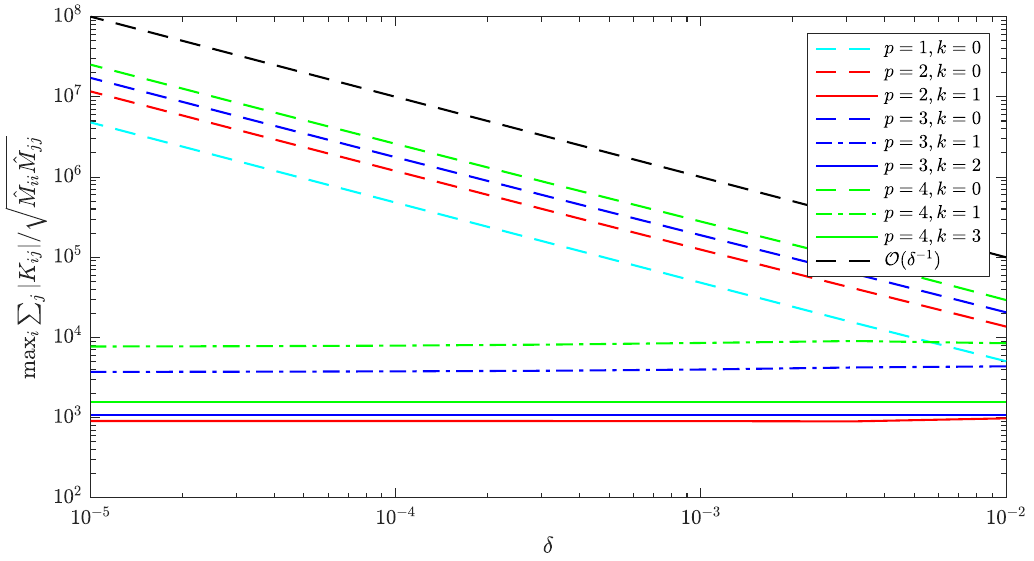}
            \caption{$\max_{i} \sum_{j} \abs{K_{ij}} / \sqrt{\LM_{ii}\LM_{jj}}$.}
        \end{subfigure}
        \hfill
        \begin{subfigure}[t]{0.49\textwidth}
            \centering
            \includegraphics[width=\textwidth]{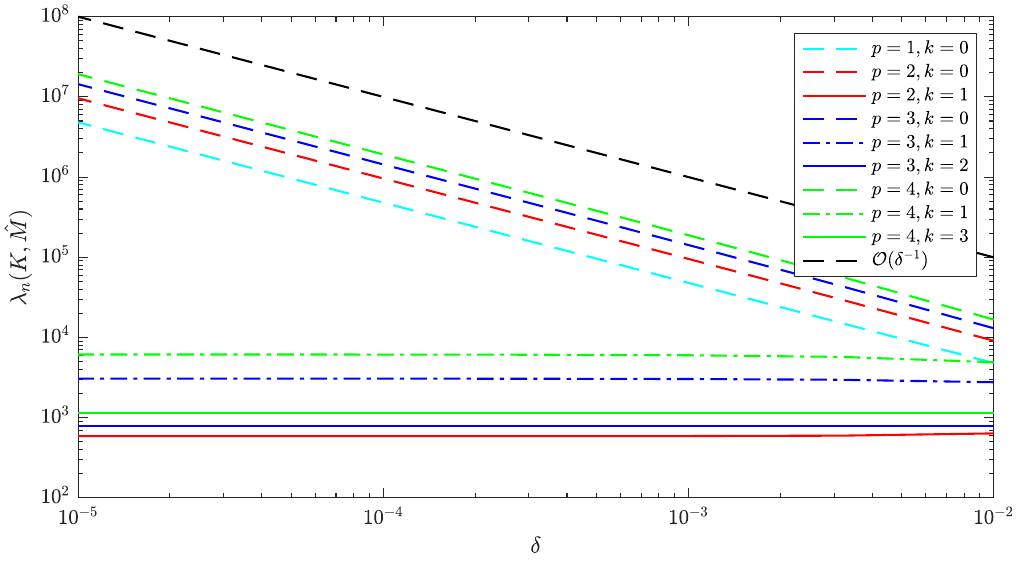}
            \caption{Largest eigenvalue with lumped mass.}
        \end{subfigure}
        \hfill
        \begin{subfigure}[t]{0.49\textwidth}
            \centering
            \includegraphics[width=\textwidth]{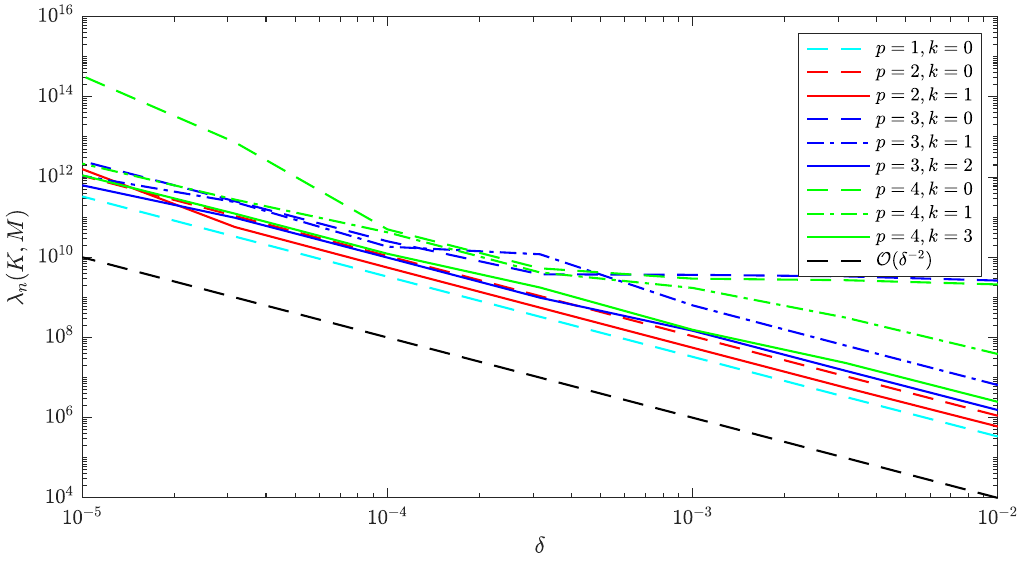}
            \caption{Largest eigenvalue with consistent mass.}
            \label{fig:Trimming2D_ex5_lambdan_consistent}
        \end{subfigure}
        \hfill
        \begin{subfigure}[t]{0.49\textwidth}
            \centering
            \includegraphics[width=\textwidth]{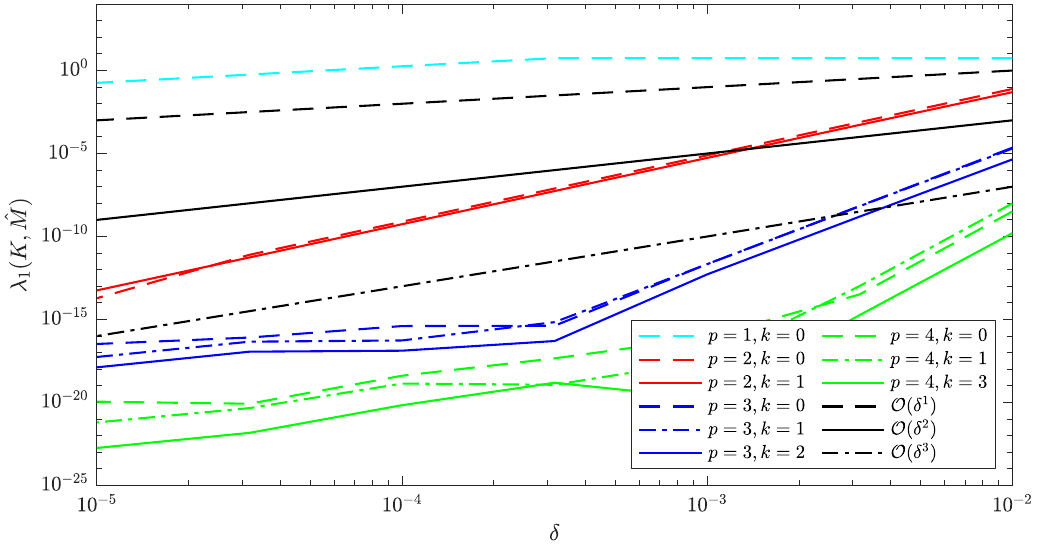}
            \caption{Smallest eigenvalue with lumped mass.}
        \end{subfigure}
        \hfill
        \begin{subfigure}[t]{0.49\textwidth}
            \centering
            \includegraphics[width=\textwidth]{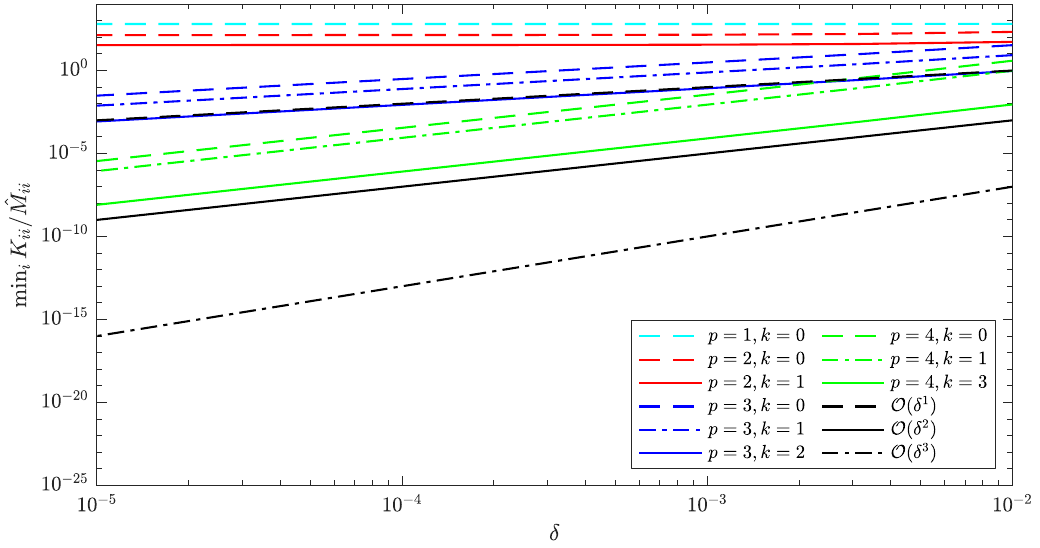}
            \caption{$\min_i K_{ii} / \LM_{ii}$.}
        \end{subfigure}
        \begin{subfigure}[t]{0.32\textwidth}
            \centering
            \includegraphics[width=0.75\textwidth]{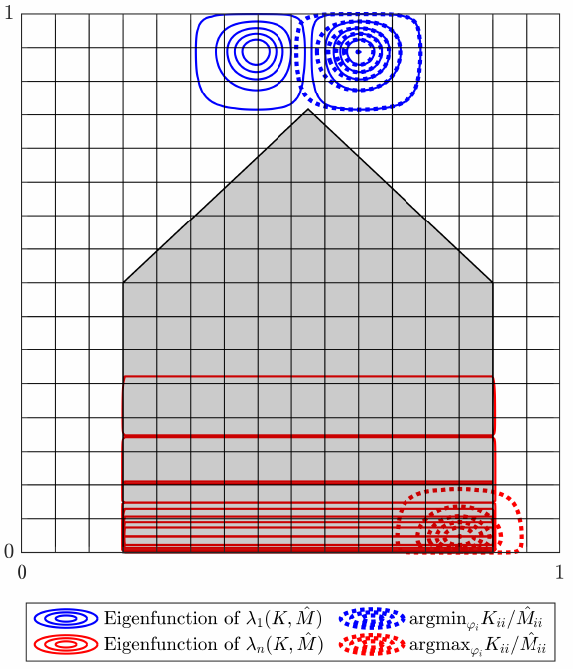}
            \caption{Degree $p = 3$ splines and $C^{p-1}$ continuity:  B-splines attaining the extreme values of $K_{ii} / \LM_{ii}$, and eigenfunctions corresponding to extreme eigenvalues, with trimming parameter $\delta=10^{-4}$.}
            \label{fig:2Dtrimming_ex5_eigenk2}
        \end{subfigure}
        \hfill
        \begin{subfigure}[t]{0.32\textwidth}
            \centering
            \includegraphics[width=0.75\textwidth]{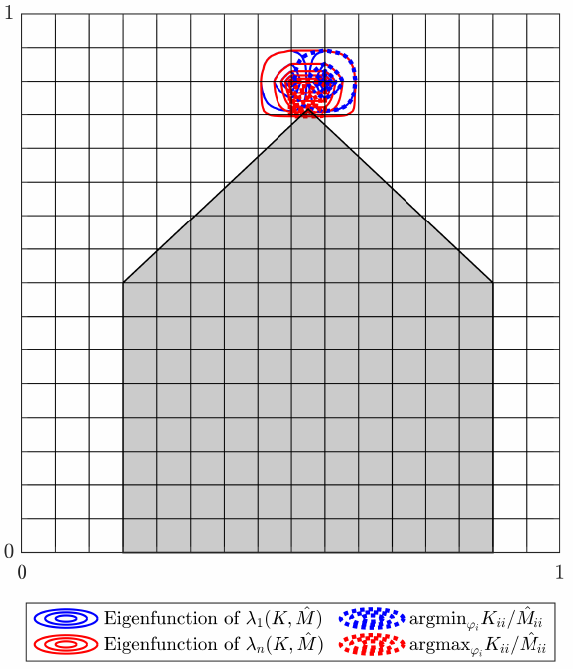}
            \caption{Degree $p = 3$ splines and $C^{0}$ continuity:  B-splines attaining the extreme values of $K_{ii} / \LM_{ii}$, and eigenfunctions corresponding to extreme eigenvalues, with trimming parameter $\delta=10^{-4}$.}
            \label{fig:2Dtrimming_ex5_eigenk0}
        \end{subfigure}
        \hfill
        \begin{subfigure}[t]{0.32\textwidth}
            \centering
            \includegraphics[width=0.9\textwidth]{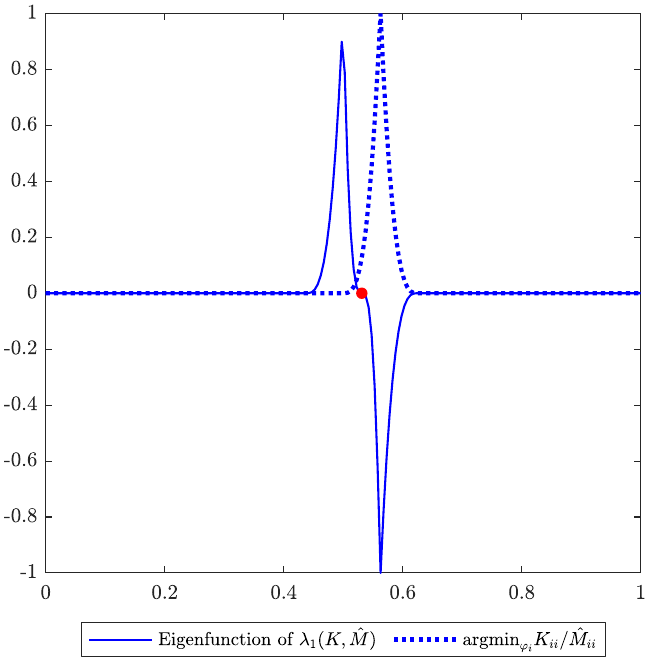}
            \caption{Degree $p = 3$ splines and $C^{0}$ continuity: rotated view of the B-spline attaining $\min_i K_{ii} / \LM_{ii}$, and of the eigenfunction corresponding to the smallest eigenvalue. The red dot locates the house tip.}
            \label{fig:2Dtrimming_ex5_eigenk2_rotated}
        \end{subfigure}
        \caption{Example~\ref{ex:house_c2} of 2D trimming.}
        \label{fig:2Dtrimming_ex5_geo2}
    \end{figure}

    \begin{figure}[h!]
        \centering
        \begin{subfigure}[t]{0.49\textwidth}
            \centering
            \includegraphics[width=\textwidth]{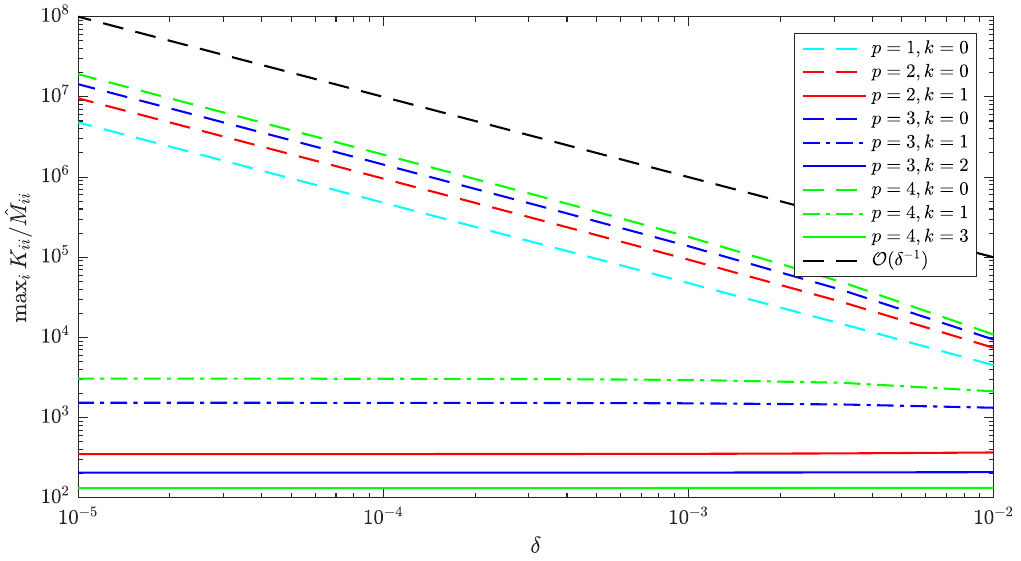}
            \caption{$\max_i K_{ii} / \LM_{ii}$.}
        \end{subfigure}
        \hfill
        \begin{subfigure}[t]{0.49\textwidth}
            \centering
            \includegraphics[width=\textwidth]{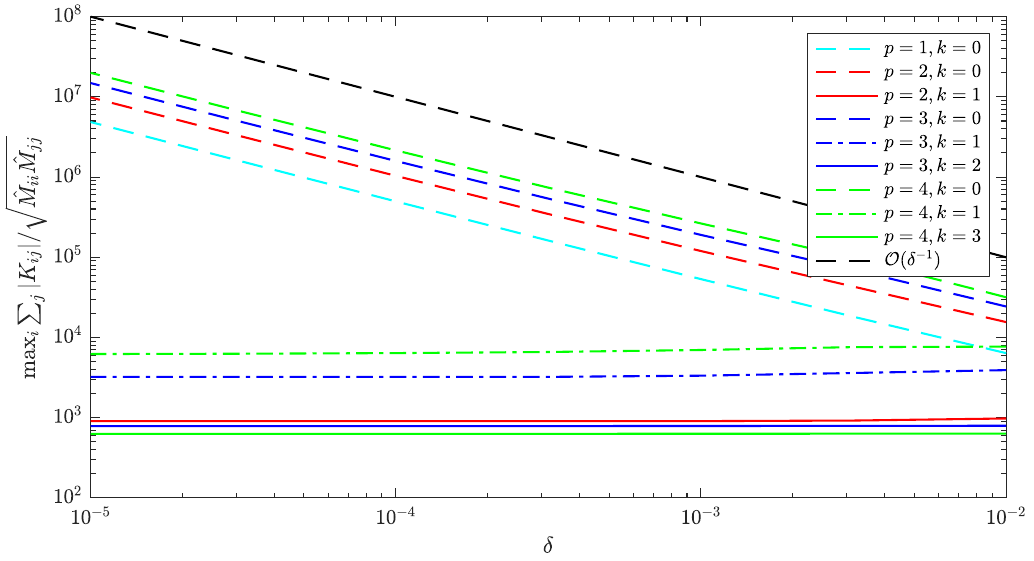}
            \caption{$\max_{i} \sum_{j} \abs{K_{ij}} / \sqrt{\LM_{ii}\LM_{jj}}$.}
        \end{subfigure}
        \hfill
        \begin{subfigure}[t]{0.49\textwidth}
            \centering
            \includegraphics[width=\textwidth]{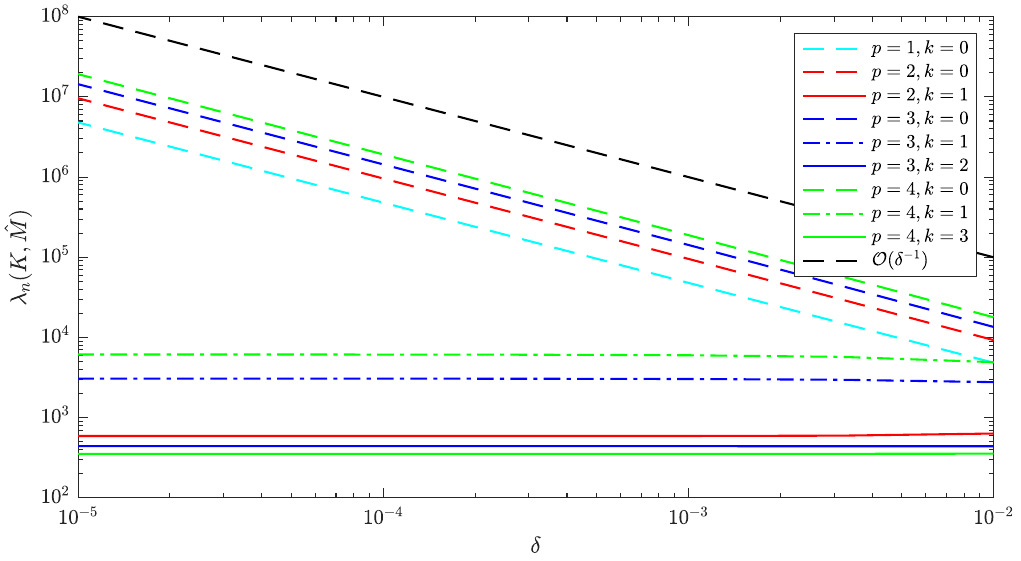}
            \caption{Largest eigenvalue with lumped mass.}
        \end{subfigure}
        \hfill
        \begin{subfigure}[t]{0.49\textwidth}
            \centering
            \includegraphics[width=\textwidth]{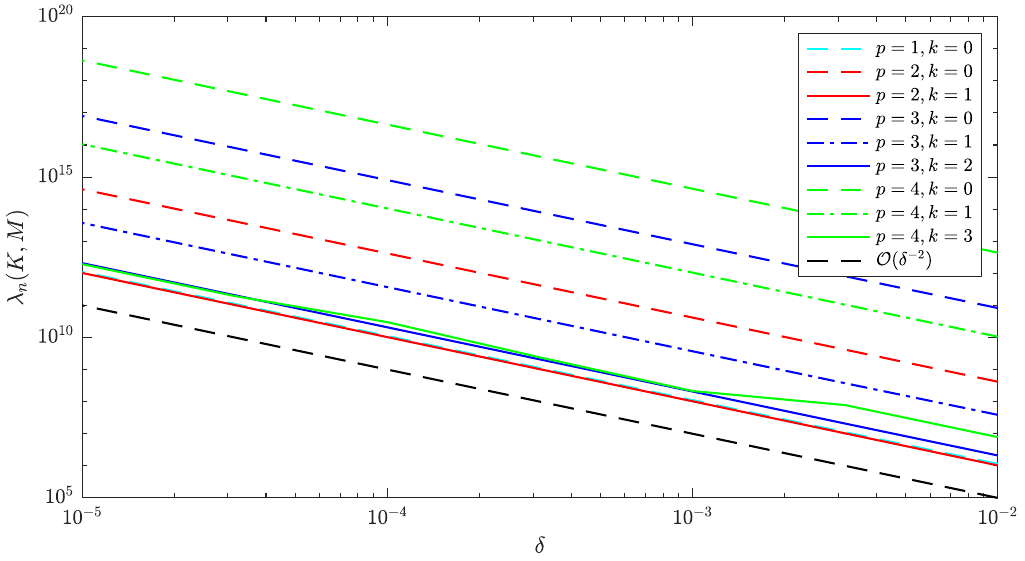}
            \caption{Largest eigenvalue with consistent mass.}
        \end{subfigure}
        \hfill
        \begin{subfigure}[t]{0.49\textwidth}
            \centering
            \includegraphics[width=\textwidth]{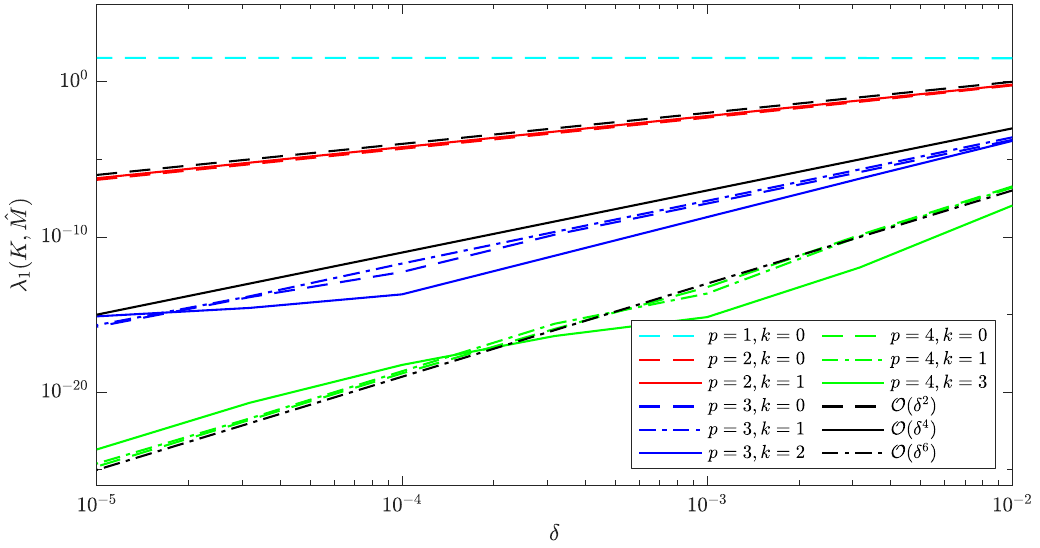}
            \caption{Smallest eigenvalue with lumped mass.}
        \end{subfigure}
        \hfill
        \begin{subfigure}[t]{0.49\textwidth}
            \centering
            \includegraphics[width=\textwidth]{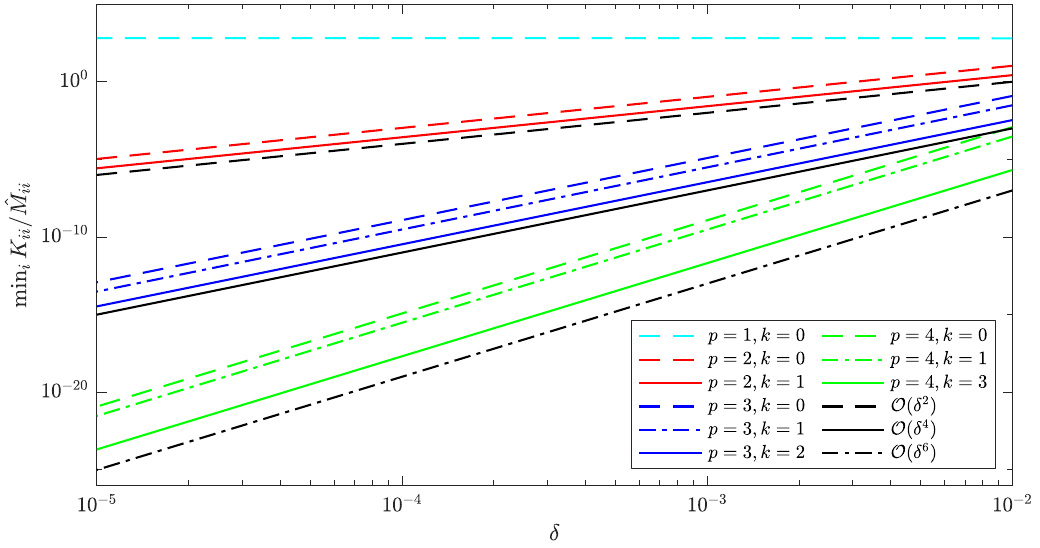}
            \caption{$\min_i K_{ii} / \LM_{ii}$.}
        \end{subfigure}
        \caption{Example~\ref{ex:rotated_square} of 2D trimming. }
        \label{fig:2Dtrimming_rotated_square_geo1}
    \end{figure}
\end{ex}

\begin{ex}[Configuration~\ref{fig:trimming2d_angle}]
    \label{ex:shifted_rotated_square}
    We now present an example generalizing \Cref{ex:house_c2} to multiple directions. The rotated square from \Cref{ex:rotated_square} is shifted along the horizontal to dissociate it from configuration \ref{fig:trimming2d_2directions}. Once again, pure homogeneous Neumann boundary conditions are prescribed, along with the additional condition $\int_\Omega u = 0$ to ensure uniqueness. The results from \Cref{ex:house_c2} are consistently recovered. However, although Examples \ref{ex:rotated_square} and \ref{ex:shifted_rotated_square} have exactly the same geometry, the positioning of the trimmed elements with respect to the grid lines heavily impacts the spectral properties. While the largest eigenvalue and minimum ratio perfectly follow our theoretical estimates, the smallest eigenvalue now actually decays much faster than its predicted rate given by the minimum ratio. The mismatch is again probably due to a linear combination of bad basis functions and is a limitation of our theory. However, the jumble in \Cref{fig:Trimming2D_rotatedshiftedsquare_lambdan} is again partly attributed to convergence issues with the eigensolver.

    \begin{figure}[p]
        \centering
        \includegraphics[height=0.3\textheight]{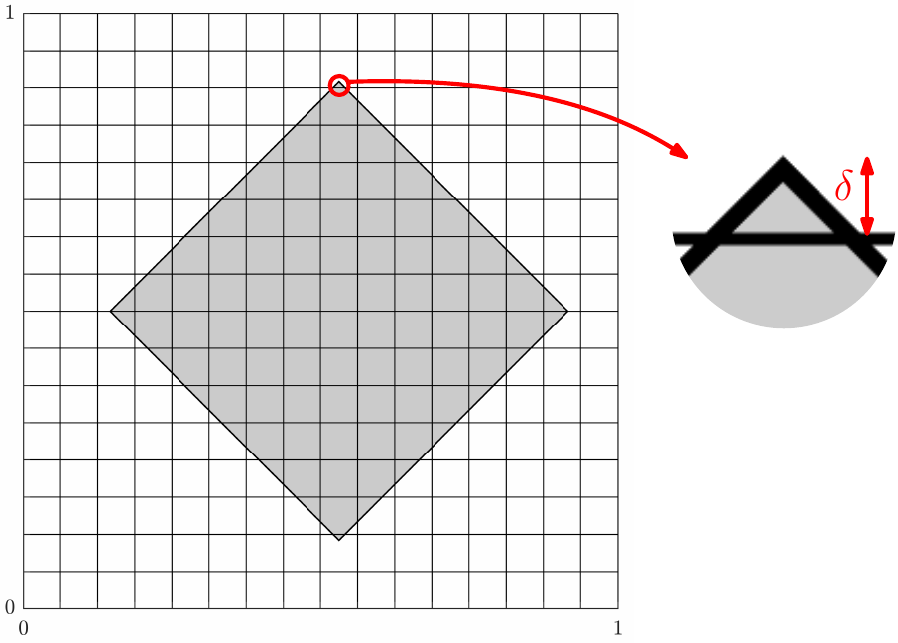}
        \caption{Fictitious (white) and trimmed (grey) domain for Example~\ref{ex:shifted_rotated_square}.}
        \label{fig:2D_trimming_rotatedshiftedsquare}
    \end{figure}

    \begin{figure}[p]
        \centering
        \begin{subfigure}[t]{0.49\textwidth}
            \centering
            \includegraphics[width=\textwidth]{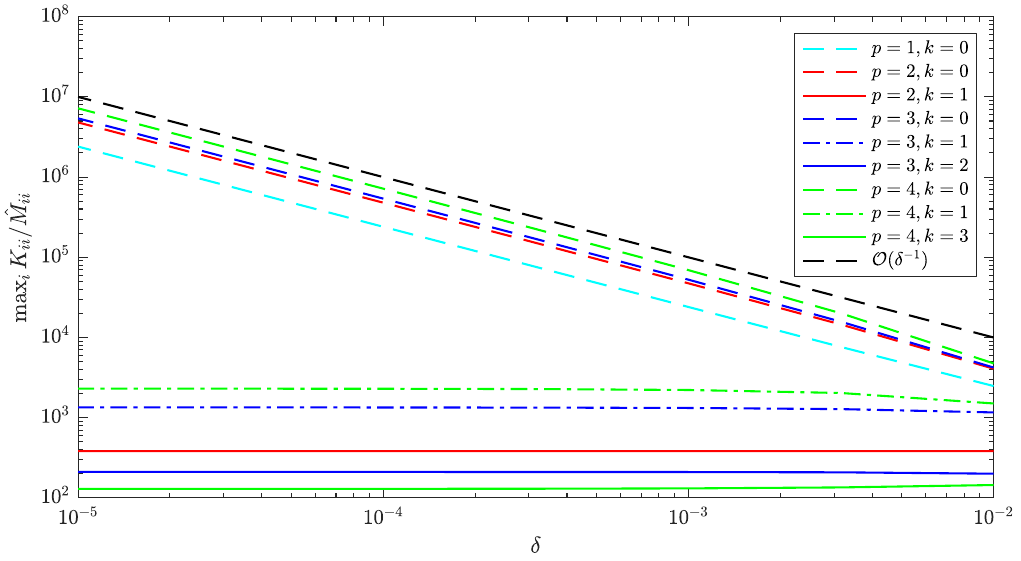}
            \caption{$\max_i K_{ii} / \LM_{ii}$.}
        \end{subfigure}
        \hfill
        \begin{subfigure}[t]{0.49\textwidth}
            \centering
            \includegraphics[width=\textwidth]{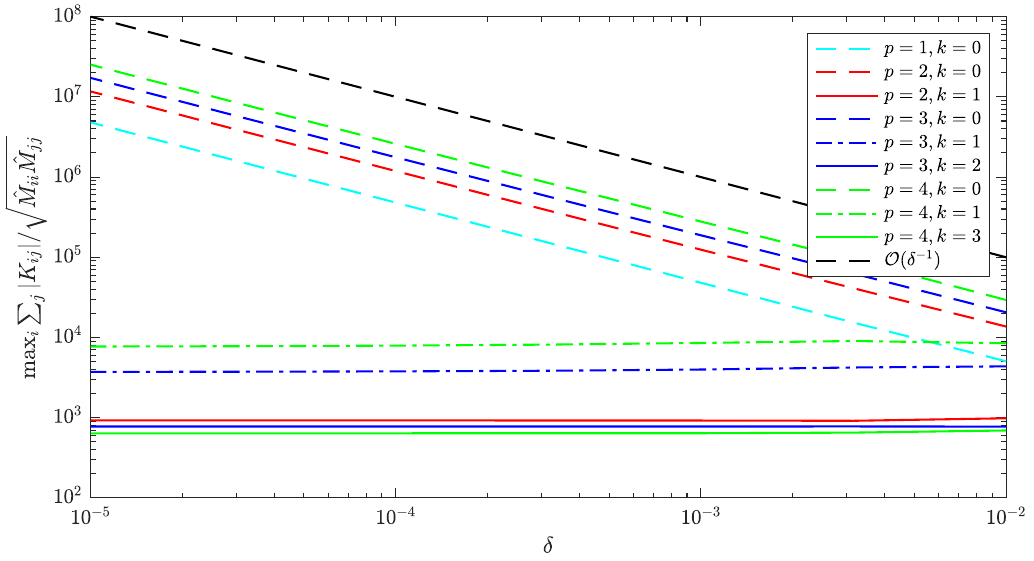}
            \caption{$\max_{i} \sum_{j} \abs{K_{ij}} / \sqrt{\LM_{ii}\LM_{jj}}$.}
        \end{subfigure}
        \hfill
        \begin{subfigure}[t]{0.49\textwidth}
            \centering
            \includegraphics[width=\textwidth]{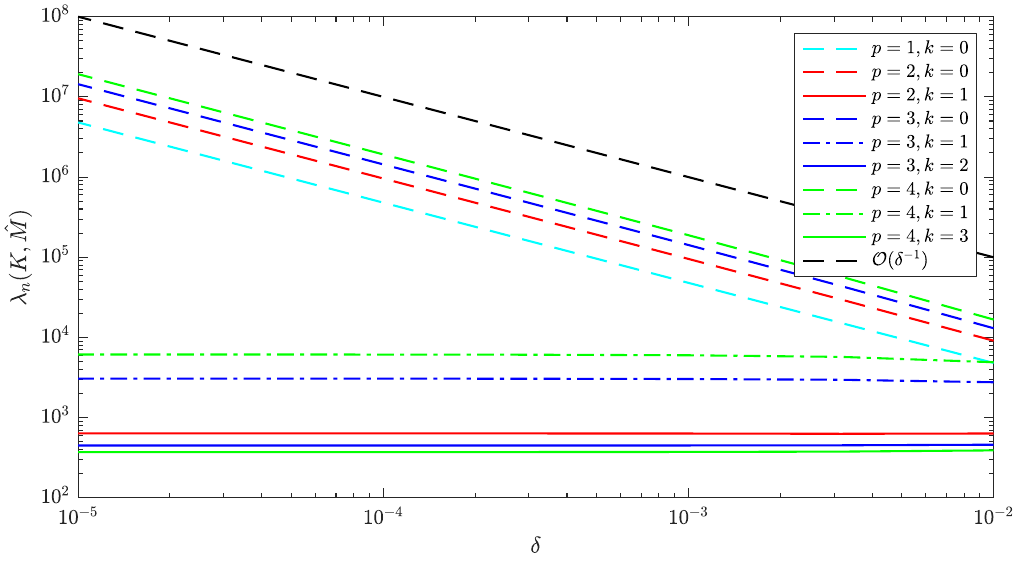}
            \caption{Largest eigenvalue with lumped mass.}
            \label{fig:Trimming2D_rotatedshiftedsquare_lambdan}
        \end{subfigure}
        \hfill
        \begin{subfigure}[t]{0.49\textwidth}
            \centering
            \includegraphics[width=\textwidth]{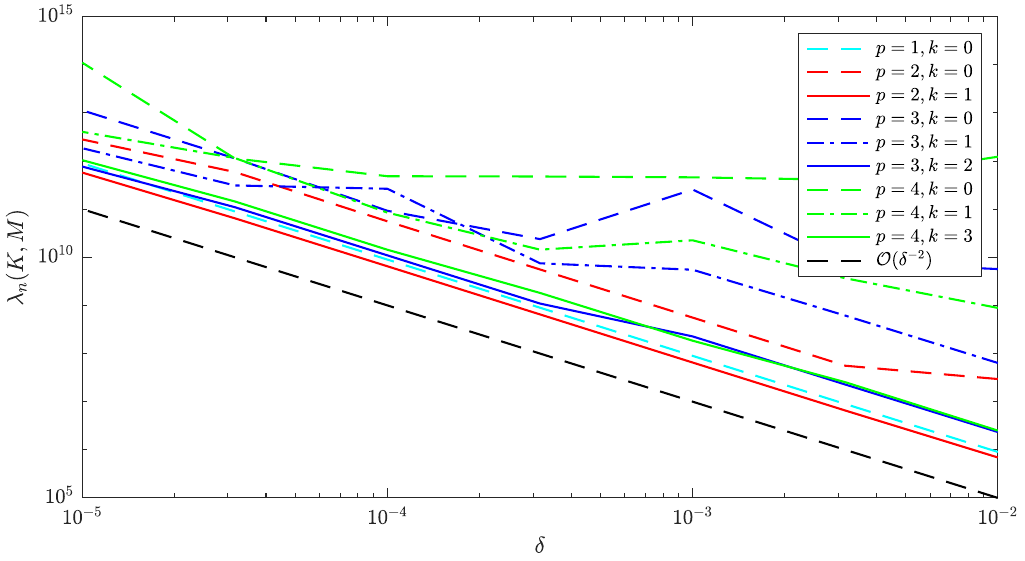}
            \caption{Largest eigenvalue with consistent mass.}
        \end{subfigure}
        \hfill
        \begin{subfigure}[t]{0.49\textwidth}
            \centering
            \includegraphics[width=\textwidth]{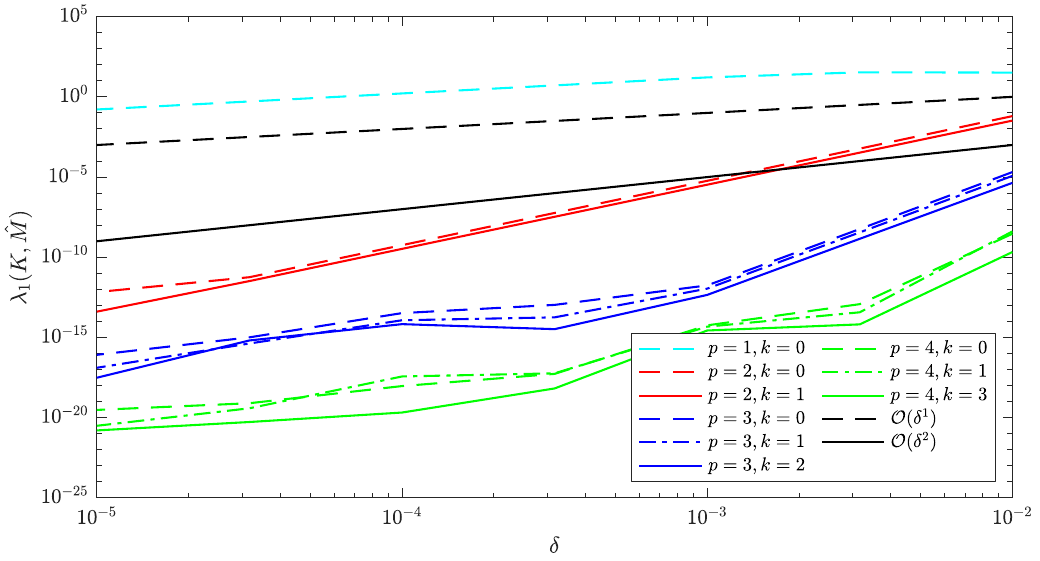}
            \caption{Smallest eigenvalue with lumped mass.}
        \end{subfigure}
        \hfill
        \begin{subfigure}[t]{0.49\textwidth}
            \centering
            \includegraphics[width=\textwidth]{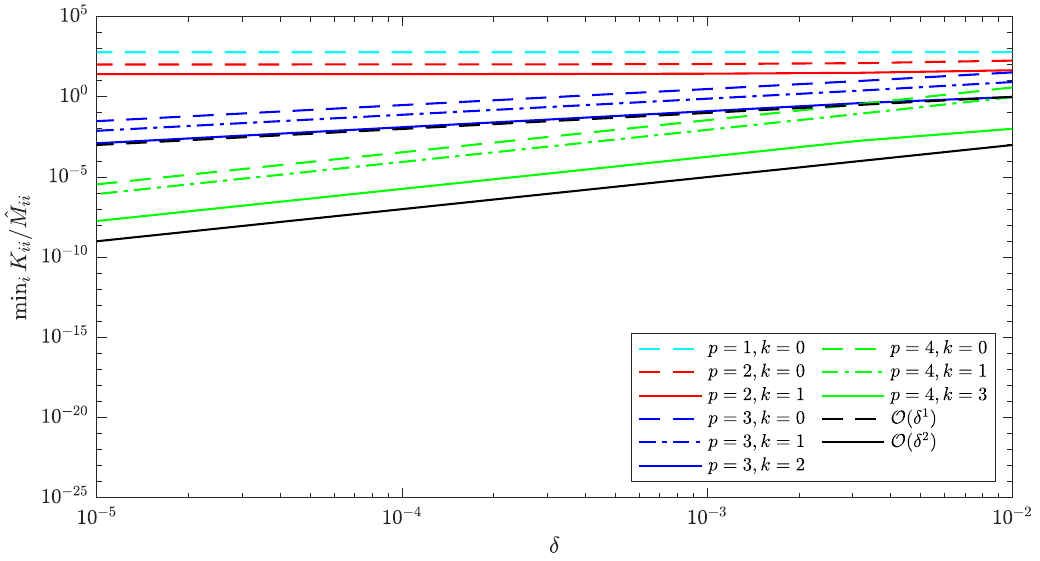}
            \caption{$\min_i K_{ii} / \LM_{ii}$.}
        \end{subfigure}
        \caption{Example~\ref{ex:shifted_rotated_square} of 2D trimming. }
        \label{fig:2Dtrimming_rotated_square_geo2}
    \end{figure}
\end{ex}

\begin{ex}[Configuration~\ref{fig:trimming2d_2directions}]
    \label{ex:curved}
    Let us now come to our last example, inspired from \cite{de2019preconditioning}, which isolates configuration~\ref{fig:trimming2d_2directions} from the two others. The physical domain is the unit square $[0,1]^2$ with a central circular hole of radius $R_{\mathrm{in}} = \sqrt{(1/5)^2 + h^2} - \delta$ and circular cut-outs of radius $R_{\mathrm{corner}} = 1/4 - \delta$ around each corner, where $h = 1/20$ is the mesh size. Several circular arcs cut through the corner of some basis functions' support thereby defining the trimming parameter $\delta$, as in \Cref{fig:2D_trimming_curved}, and resulting in a domain exclusively exhibiting trimming configuration~\ref{fig:trimming2d_2directions}. Finally, Dirichlet boundary conditions are strongly imposed on the edges aligned with those of the unit square. Since only configuration~\ref{fig:trimming2d_2directions} is active, \eqref{eq:case3_lambdan} implies that the largest eigenvalue remains bounded, regardless of the degree and continuity. Those results are confirmed in \Cref{fig:2D_trimming_curved}, where the curves for different continuities overlap. To the best of our knowledge, this is the first time a trimming configuration is exhibited for which the largest eigenvalue remains bounded, even for $C^0$ continuity. Moreover, configuration~\ref{fig:trimming2d_2directions} is responsible for a $\calO(\delta^{2p-2})$ decay of the minimum ratio, as shown in \Cref{fig:2D_trimming_curved_minratio}. The smallest eigenvalue follows the same decay rate until it approaches machine precision.

    \begin{figure}[p]
        \centering
        \includegraphics[height=0.3\textheight]{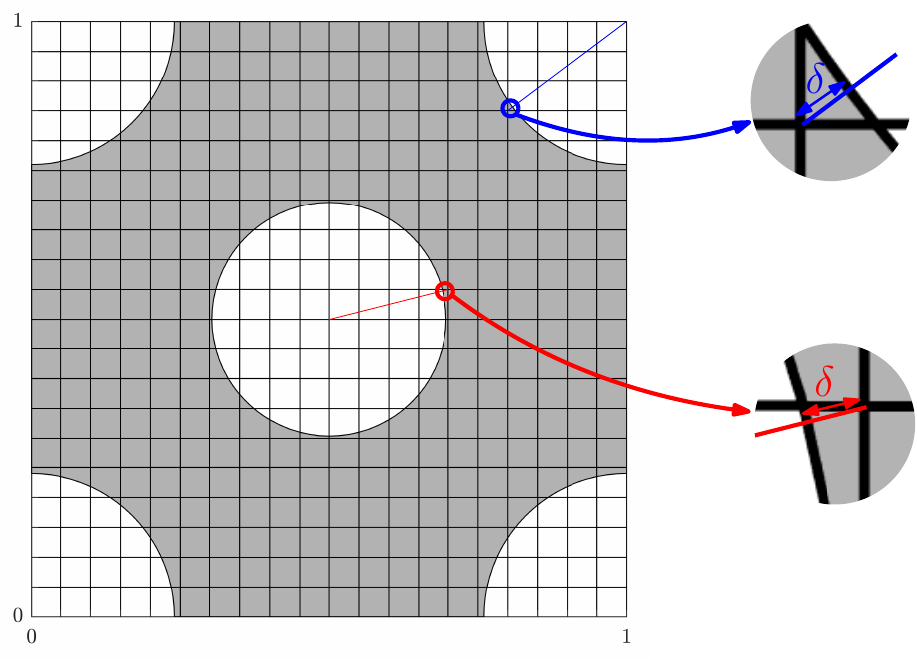}
        \caption{Fictitious (white) and trimmed (grey) domain for Example~\ref{ex:curved}.}
        \label{fig:2D_trimming_curved}
    \end{figure}

    \begin{figure}[p]
        \centering
        \begin{subfigure}[t]{0.49\textwidth}
            \centering
            \includegraphics[width=\textwidth]{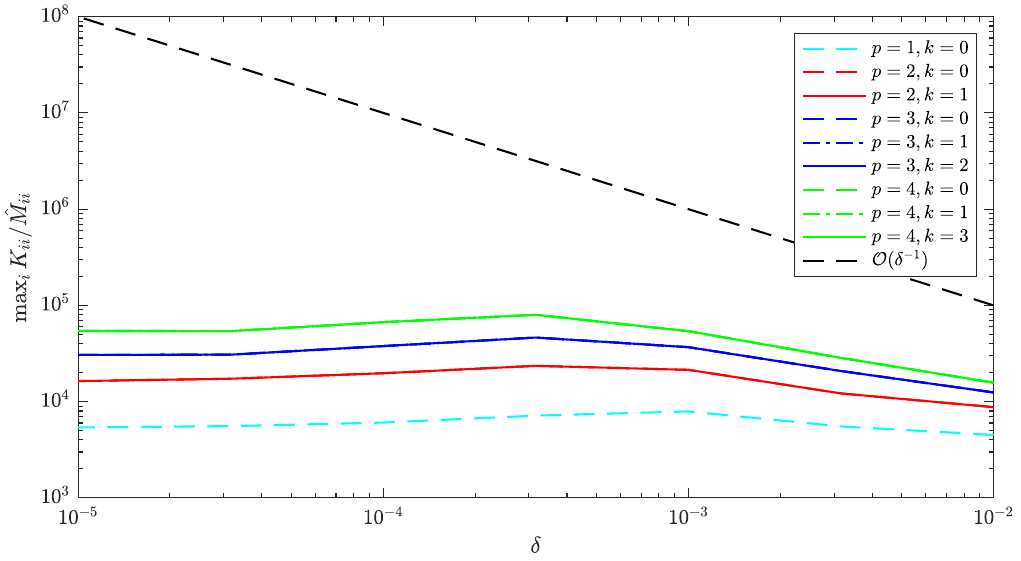}
            \caption{$\max_i K_{ii} / \LM_{ii}$.}
        \end{subfigure}
        \hfill
        \begin{subfigure}[t]{0.49\textwidth}
            \centering
            \includegraphics[width=\textwidth]{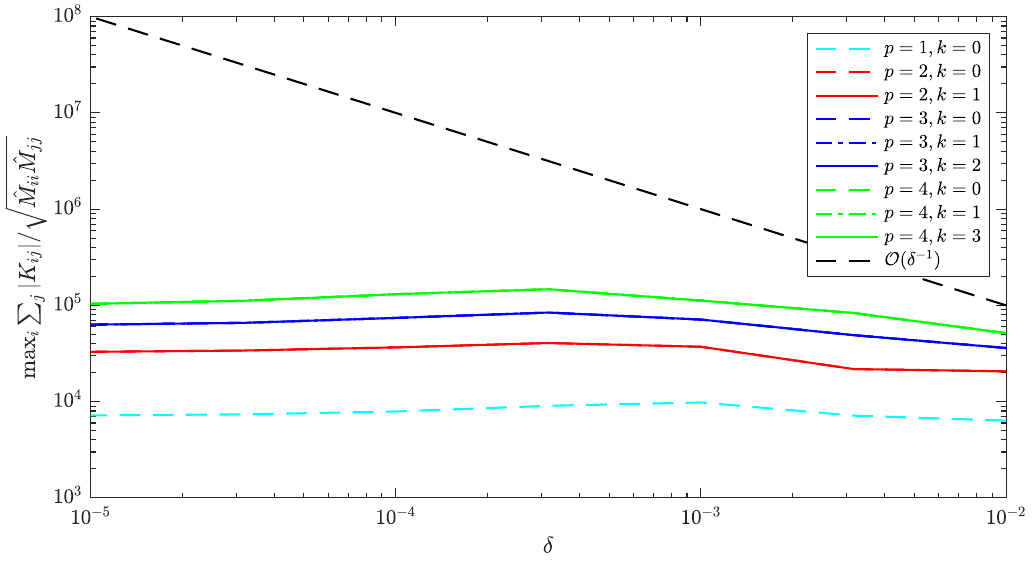}
            \caption{$\max_{i} \sum_{j} \abs{K_{ij}} / \sqrt{\LM_{ii}\LM_{jj}}$.}
        \end{subfigure}
        \hfill
        \begin{subfigure}[t]{0.49\textwidth}
            \centering
            \includegraphics[width=\textwidth]{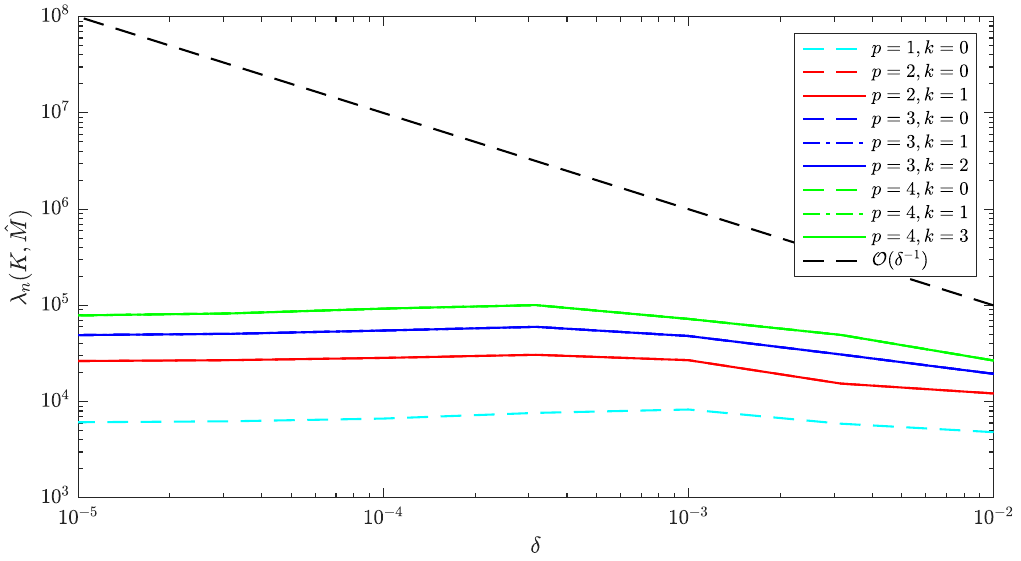}
            \caption{Largest eigenvalue with lumped mass.}
            \label{fig:Trimming2D_curved_lambdan}
        \end{subfigure}
        \hfill
        \begin{subfigure}[t]{0.49\textwidth}
            \centering
            \includegraphics[width=\textwidth]{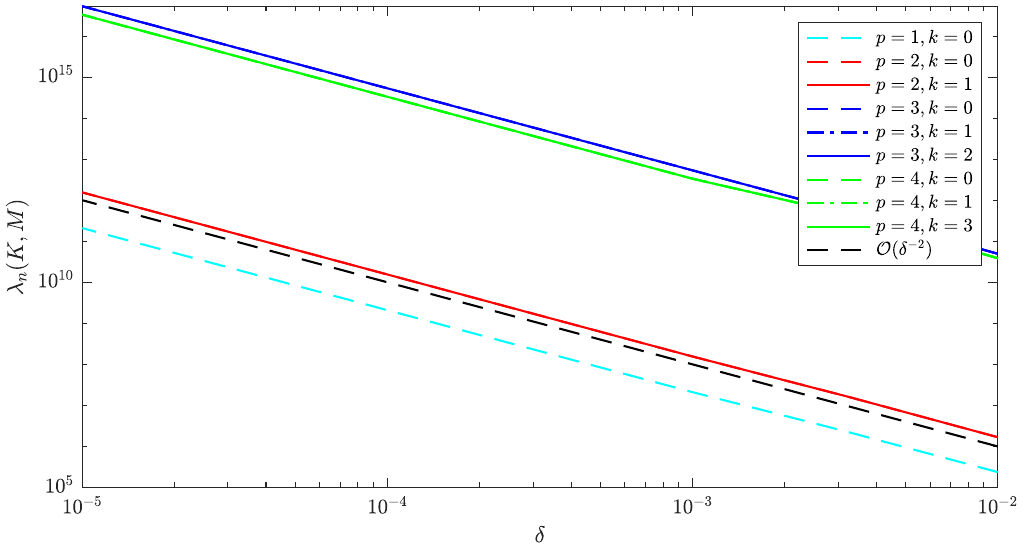}
            \caption{Largest eigenvalue with consistent mass.}
        \end{subfigure}
        \hfill
        \begin{subfigure}[t]{0.49\textwidth}
            \centering
            \includegraphics[width=\textwidth]{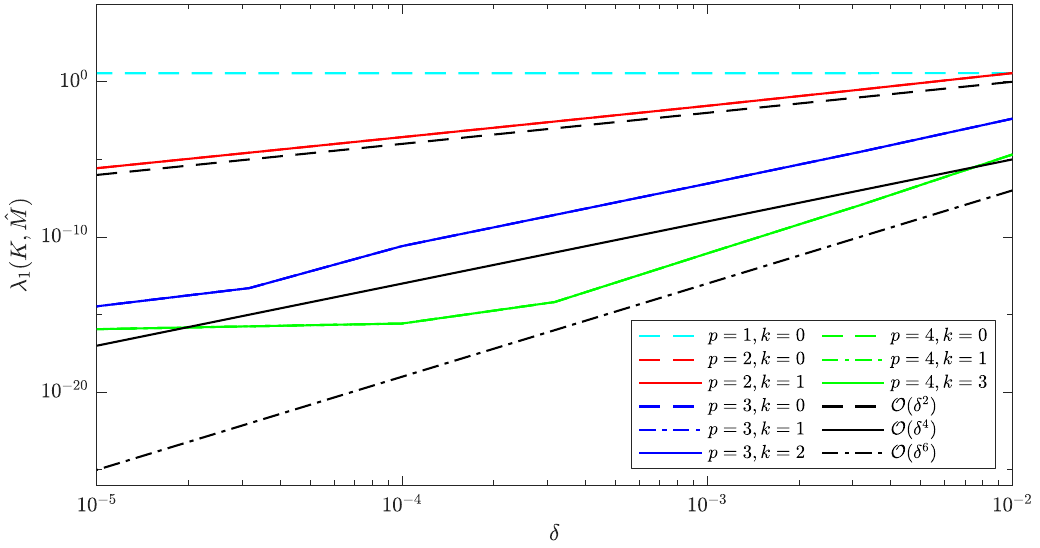}
            \caption{Smallest eigenvalue with lumped mass.}
            \label{fig:2D_trimming_curved_mineig}
        \end{subfigure}
        \hfill
        \begin{subfigure}[t]{0.49\textwidth}
            \centering
            \includegraphics[width=\textwidth]{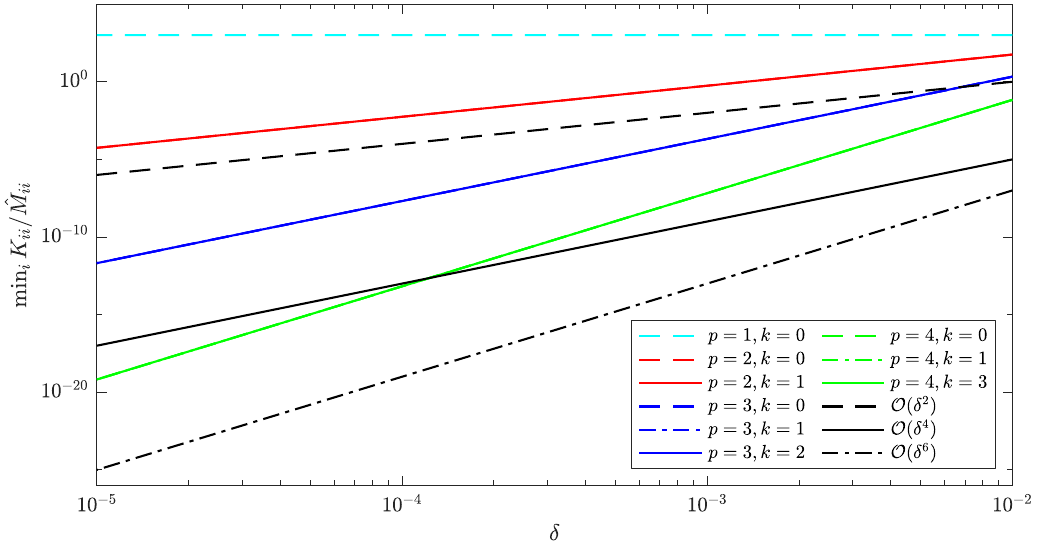}
            \caption{$\min_i K_{ii} / \LM_{ii}$.}
            \label{fig:2D_trimming_curved_minratio}
        \end{subfigure}
        \caption{Example~\ref{ex:curved} of 2D trimming.}
        \label{fig:2Dtrimming_curved_geo2}
    \end{figure}
\end{ex}

In summary, while the largest eigenvalue always behaves as the maximum ratio, the smallest eigenvalue may actually decay much faster than the minimum ratio. But in many configurations, the minimum ratio still captures the behavior of the smallest eigenvalue. Moreover, the associated eigenfunctions sometimes closely resemble basis functions. This is not entirely surprising given that the lumped mass matrix is constructed from the consistent mass in the B-spline basis. Therefore, some of the spectral properties are closely related to the basis.

\section{Conclusion}
\label{se:conclusion}
In this article, we have derived sharp lower and upper bounds on the largest generalized eigenvalue for a lumped mass approximation. Apart from their practical value in estimating the critical step size, these bounds also provide analytical estimates on the largest eigenvalue unraveling for the first time its behavior for trimmed geometries. Our estimates confirm that smoothness is the reason for the largest eigenvalue's boundedness as the trimmed elements get smaller. However, there is another side to the story. While the largest eigenvalue may remain bounded, the smallest one is driven down to zero as the trimmed elements get smaller. Although our estimates on the smallest eigenvalue were not always sharp, it converges to zero at least polynomially at a rate that depends on both the degree and trimming configuration. Thus, mass lumping on trimmed geometries might instead introduce spurious eigenvalues (and modes) in the low-frequency spectrum, similarly to those in the high frequency spectrum for a consistent mass approximation. Their impact on the solution of a transient simulation has already been investigated in \cite{voet2025stabilization}.

\begin{toappendix}
    \section{Improved upper bound on the smallest eigenvalue}
    \label{app: improved_bound}
    In this appendix, we compare the original upper bound \eqref{eq:bound_1_laplace} on the smallest eigenvalue given by the minimum ratio with the improved upper bound \eqref{eq: improved_bound} obtained by minimizing the Rayleigh quotient over all functions in $V_h^S$ (the space spanned by small basis functions). As shown in \Cref{fig:improved_bound}, contrary to the original bound, the improved one sharply captures the behavior of the smallest eigenvalue in \emph{all} cases, thereby substantiating our claim that the smallest eigenfunction is in general a linear combination of small basis functions.

    \begin{figure}[p]
        \centering
        \begin{subfigure}[t]{0.32\textwidth}
            \centering
            \includegraphics[width=\textwidth]{Trimming1D_lambda1.pdf}
        \end{subfigure}
        \hfill
        \begin{subfigure}[t]{0.32\textwidth}
            \centering
            \includegraphics[width=\textwidth]{Trimming1D_min.pdf}
            \caption{Example of 1D trimming.}
        \end{subfigure}
        \hfill
        \begin{subfigure}[t]{0.32\textwidth}
            \centering
            \includegraphics[width=\textwidth]{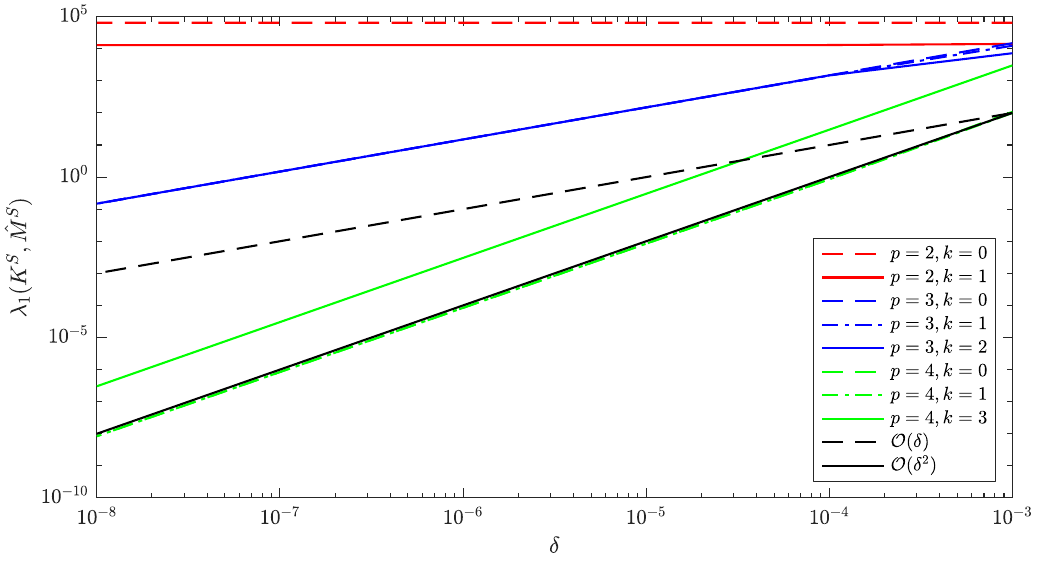}
        \end{subfigure}
        \hfill
        \begin{subfigure}[t]{0.32\textwidth}
            \centering
            \includegraphics[width=\textwidth]{Trimming2D_square_lambda1.pdf}
        \end{subfigure}
        \hfill
        \begin{subfigure}[t]{0.32\textwidth}
            \centering
            \includegraphics[width=\textwidth]{Trimming2D_square_min.pdf}
            \caption{Example~\ref{ex:trimmed_square_2_sides} of 2D trimming.}
        \end{subfigure}
        \hfill
        \begin{subfigure}[t]{0.32\textwidth}
            \centering
            \includegraphics[width=\textwidth]{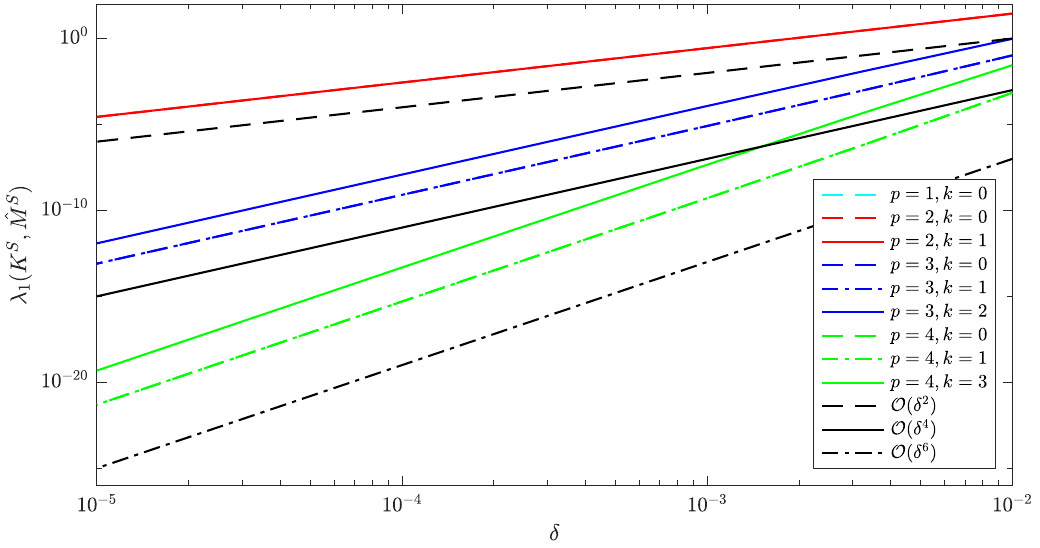}
        \end{subfigure}
        \hfill
        \begin{subfigure}[t]{0.32\textwidth}
            \centering
            \includegraphics[width=\textwidth]{Trimming2D_ex4_lambda1.pdf}
        \end{subfigure}
        \hfill
        \begin{subfigure}[t]{0.32\textwidth}
            \centering
            \includegraphics[width=\textwidth]{Trimming2D_ex4_min.pdf}
            \caption{Example~\ref{ex:house_c2_c3} of 2D trimming.}
        \end{subfigure}
        \hfill
        \begin{subfigure}[t]{0.32\textwidth}
            \centering
            \includegraphics[width=\textwidth]{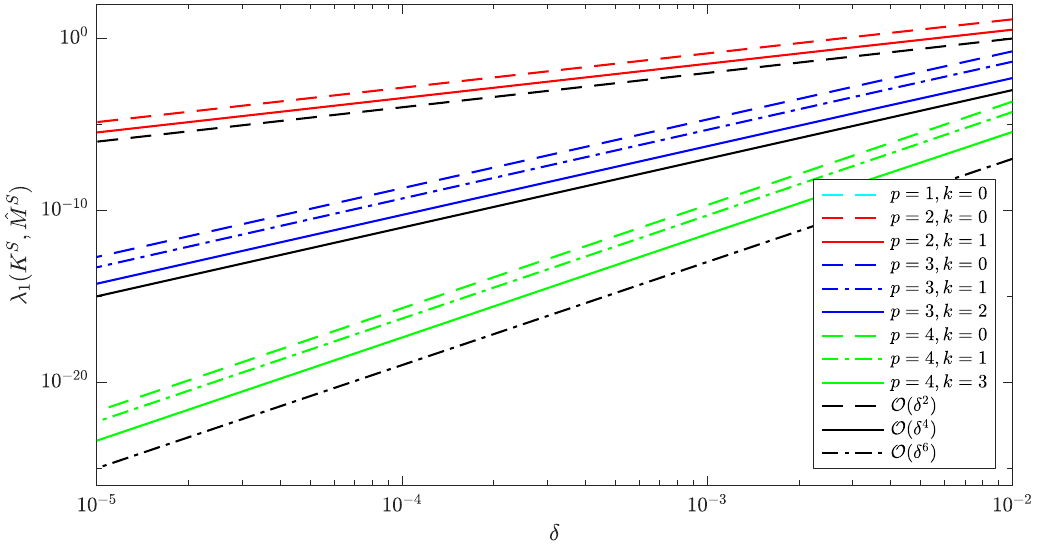}
        \end{subfigure}
        \hfill
        \begin{subfigure}[t]{0.32\textwidth}
            \centering
            \includegraphics[width=\textwidth]{Trimming2D_ex5_lambda1.pdf}
        \end{subfigure}
        \hfill
        \begin{subfigure}[t]{0.32\textwidth}
            \centering
            \includegraphics[width=\textwidth]{Trimming2D_ex5_min.pdf}
            \caption{Example~\ref{ex:house_c2} of 2D trimming.}
        \end{subfigure}
        \hfill
        \begin{subfigure}[t]{0.32\textwidth}
            \centering
            \includegraphics[width=\textwidth]{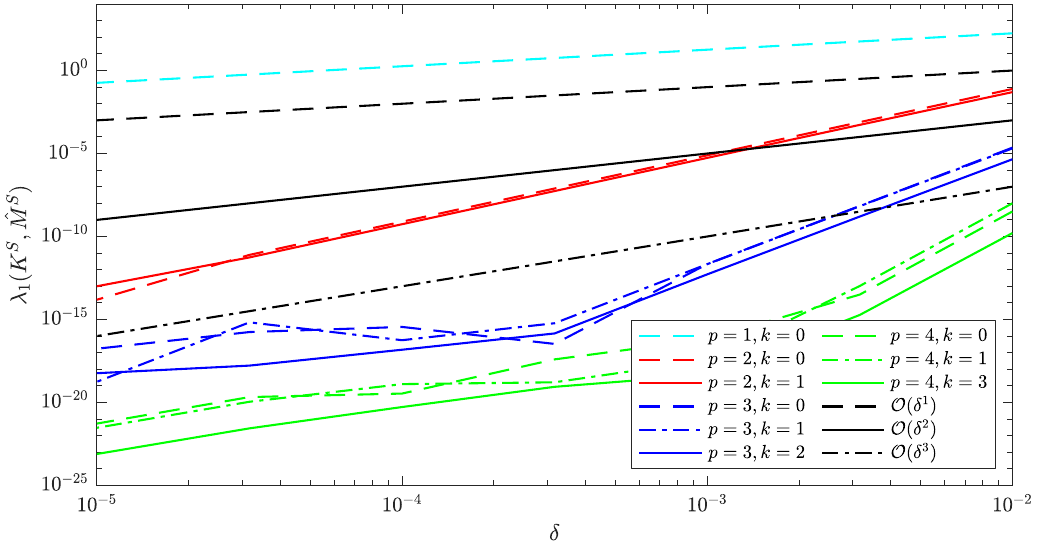}
        \end{subfigure}
        \hfill
        \begin{subfigure}[t]{0.32\textwidth}
            \centering
            \includegraphics[width=\textwidth]{Trimming2D_rotatedsquare_lambda1.pdf}
        \end{subfigure}
        \hfill
        \begin{subfigure}[t]{0.32\textwidth}
            \centering
            \includegraphics[width=\textwidth]{Trimming2D_rotatedsquare_min.pdf}
            \caption{Example~\ref{ex:rotated_square} of 2D trimming.}
        \end{subfigure}
        \hfill
        \begin{subfigure}[t]{0.32\textwidth}
            \centering
            \includegraphics[width=\textwidth]{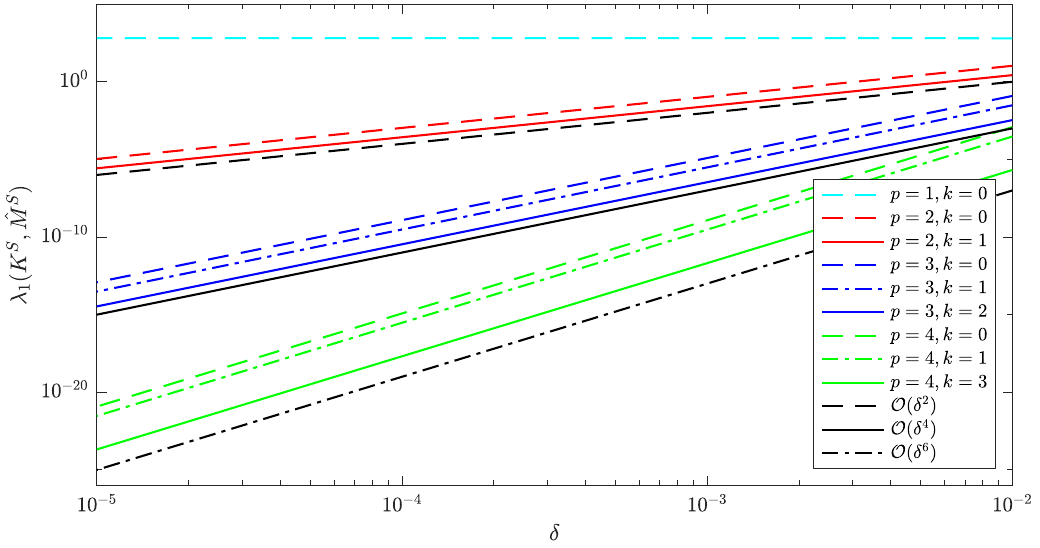}
        \end{subfigure}
        \hfill
        \begin{subfigure}[t]{0.32\textwidth}
            \centering
            \includegraphics[width=\textwidth]{Trimming2D_rotatedshiftedsquare_lambda1.pdf}
        \end{subfigure}
        \hfill
        \begin{subfigure}[t]{0.32\textwidth}
            \centering
            \includegraphics[width=\textwidth]{Trimming2D_rotatedshiftedsquare_min.pdf}
            \caption{Example~\ref{ex:shifted_rotated_square} of 2D trimming.}
        \end{subfigure}
        \hfill
        \begin{subfigure}[t]{0.32\textwidth}
            \centering
            \includegraphics[width=\textwidth]{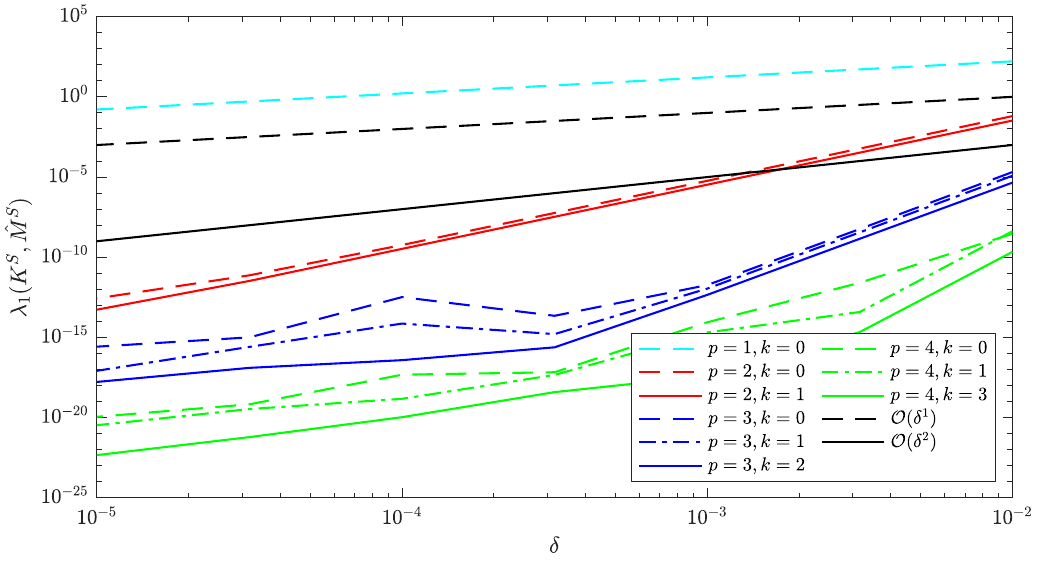}
        \end{subfigure}
        \hfill
        \begin{subfigure}[t]{0.32\textwidth}
            \centering
            \includegraphics[width=\textwidth]{Trimming2D_curved_lambda1.pdf}
        \end{subfigure}
        \hfill
        \begin{subfigure}[t]{0.32\textwidth}
            \centering
            \includegraphics[width=\textwidth]{Trimming2D_curved_min.pdf}
            \caption{Example~\ref{ex:curved} of 2D trimming.}
        \end{subfigure}
        \hfill
        \begin{subfigure}[t]{0.32\textwidth}
            \centering
            \includegraphics[width=\textwidth]{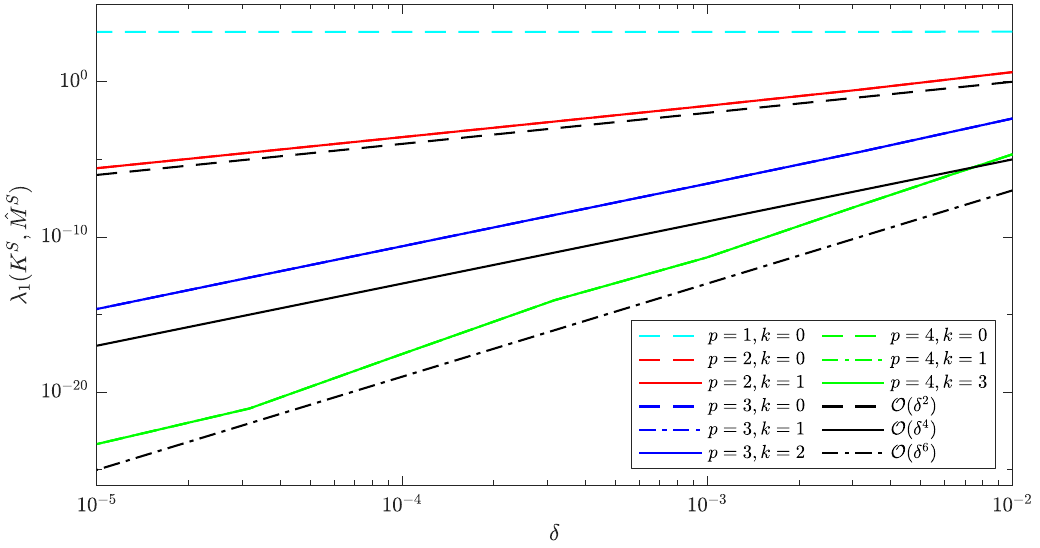}
        \end{subfigure}
        \hfill
        \caption{Comparison of the smallest eigenvalue (left column) with its original upper bound \eqref{eq:bound_1_laplace} (middle column) and its improved upper bound \eqref{eq: improved_bound} (right column) for all 1D and 2D examples.}
        \label{fig:improved_bound}
    \end{figure}
\end{toappendix}

\section*{CRediT authorship contribution statement}
\textbf{Ivan Bioli:} Conceptualization, Formal analysis, Investigation, Methodology, Software, Validation, Visualization, Writing – original draft, Writing – review \& editing.

\textbf{Yannis Voet:} Conceptualization, Formal analysis, Investigation, Methodology, Software, Validation, Visualization, Writing – original draft, Writing – review \& editing.

\section*{Declaration of competing interest}
The authors declare that they have no known competing financial interests or personal relationships that could have appeared to influence the work reported in this paper.

\section*{Data availability}
No data was used for the research described in the article.

\section*{Acknowledgments}
Ivan Bioli was partially supported by the EPFL - MNS - ``Numerical Modelling and Simulation'' group as a visiting ``Master's Valorisation''. This support is gratefully acknowledged.

Ivan Bioli was partially supported by the European Union’s Horizon Europe research and innovation programme under the Marie Skłodowska-Curie grant agreement No 101119556.

Both authors would like to express their gratitude to Giancarlo Sangalli and Espen Sande for carefully reading through this manuscript and to Stein Stoter for helpful discussions on the topic. Ivan Bioli also extends his thanks to Annalisa Buffa for her valuable support during his time at EPFL - MNS.

\bibliography{bibliography}

\begin{thebibliography}{51}
\providecommand{\natexlab}[1]{#1}
\providecommand{\url}[1]{\texttt{#1}}
\expandafter\ifx\csname urlstyle\endcsname\relax
  \providecommand{\doi}[1]{doi: #1}\else
  \providecommand{\doi}{doi: \begingroup \urlstyle{rm}\Url}\fi

\bibitem[Anitescu et~al.(2019)Anitescu, Nguyen, Rabczuk, and Zhuang]{anitescu2019isogeometric}
C.~Anitescu, C.~Nguyen, T.~Rabczuk, and X.~Zhuang.
\newblock {Isogeometric analysis for explicit elastodynamics using a dual-basis diagonal mass formulation}.
\newblock \emph{Computer Methods in Applied Mechanics and Engineering}, 346:\penalty0 574--591, 2019.

\bibitem[Bathe(2006)]{bathe2006finite}
K.-J. Bathe.
\newblock \emph{{Finite element procedures}}.
\newblock Klaus-Jurgen Bathe, 2006.

\bibitem[Bazilevs et~al.(2006)Bazilevs, Beirao~da Veiga, Cottrell, Hughes, and Sangalli]{bazilevs2006isogeometric}
Y.~Bazilevs, L.~Beirao~da Veiga, J.~A. Cottrell, T.~J. Hughes, and G.~Sangalli.
\newblock {Isogeometric analysis: approximation, stability and error estimates for h-refined meshes}.
\newblock \emph{Mathematical Models and Methods in Applied Sciences}, 16\penalty0 (07):\penalty0 1031--1090, 2006.

\bibitem[Borden et~al.(2014)Borden, Hughes, Landis, and Verhoosel]{borden2014higher}
M.~J. Borden, T.~J. Hughes, C.~M. Landis, and C.~V. Verhoosel.
\newblock {A higher-order phase-field model for brittle fracture: Formulation and analysis within the isogeometric analysis framework}.
\newblock \emph{Computer Methods in Applied Mechanics and Engineering}, 273:\penalty0 100--118, 2014.

\bibitem[Bressan and Sande(2019)]{bressan2019approximation}
A.~Bressan and E.~Sande.
\newblock {Approximation in FEM, DG and IGA: a theoretical comparison}.
\newblock \emph{Numerische Mathematik}, 143:\penalty0 923--942, 2019.

\bibitem[Buffa et~al.(2020)Buffa, Puppi, and V{\'a}zquez]{buffa2020minimal}
A.~Buffa, R.~Puppi, and R.~V{\'a}zquez.
\newblock {A minimal stabilization procedure for isogeometric methods on trimmed geometries}.
\newblock \emph{SIAM Journal on Numerical Analysis}, 58\penalty0 (5):\penalty0 2711--2735, 2020.

\bibitem[Burman et~al.(2023)Burman, Hansbo, Larson, and Larsson]{burman2023extension}
E.~Burman, P.~Hansbo, M.~G. Larson, and K.~Larsson.
\newblock {Extension operators for trimmed spline spaces}.
\newblock \emph{Computer Methods in Applied Mechanics and Engineering}, 403:\penalty0 115707, 2023.

\bibitem[Chan and Evans(2018)]{chan2018multi}
J.~Chan and J.~A. Evans.
\newblock {Multi-patch discontinuous Galerkin isogeometric analysis for wave propagation: Explicit time-stepping and efficient mass matrix inversion}.
\newblock \emph{Computer Methods in Applied Mechanics and Engineering}, 333:\penalty0 22--54, 2018.

\bibitem[Cocchetti et~al.(2013)Cocchetti, Pagani, and Perego]{cocchetti2013selective}
G.~Cocchetti, M.~Pagani, and U.~Perego.
\newblock {Selective mass scaling and critical time-step estimate for explicit dynamics analyses with solid-shell elements}.
\newblock \emph{Computers \& Structures}, 127:\penalty0 39--52, 2013.

\bibitem[Cohen et~al.(1994)Cohen, Joly, and Tordjman]{cohen1994higher}
G.~Cohen, P.~Joly, and N.~Tordjman.
\newblock {Higher-order finite elements with mass-lumping for the 1D wave equation}.
\newblock \emph{Finite elements in analysis and design}, 16\penalty0 (3-4):\penalty0 329--336, 1994.

\bibitem[Collier et~al.(2012)Collier, Pardo, Dalcin, Paszynski, and Calo]{collier2012cost}
N.~Collier, D.~Pardo, L.~Dalcin, M.~Paszynski, and V.~M. Calo.
\newblock {The cost of continuity: A study of the performance of isogeometric finite elements using direct solvers}.
\newblock \emph{Computer Methods in Applied Mechanics and Engineering}, 213:\penalty0 353--361, 2012.

\bibitem[Collier et~al.(2013)Collier, Dalcin, Pardo, and Calo]{collier2013cost}
N.~Collier, L.~Dalcin, D.~Pardo, and V.~M. Calo.
\newblock {The cost of continuity: performance of iterative solvers on isogeometric finite elements}.
\newblock \emph{SIAM Journal on Scientific Computing}, 35\penalty0 (2):\penalty0 A767--A784, 2013.

\bibitem[Coradello(2021)]{coradello2021accurate}
L.~Coradello.
\newblock \emph{{Accurate isogeometric methods for trimmed shell structures}}.
\newblock PhD thesis, {\'E}cole polytechnique f{\'e}d{\'e}rale de Lausanne, 2021.

\bibitem[Cottrell et~al.(2006)Cottrell, Reali, Bazilevs, and Hughes]{cottrell2006isogeometric}
J.~A. Cottrell, A.~Reali, Y.~Bazilevs, and T.~J. Hughes.
\newblock {Isogeometric analysis of structural vibrations}.
\newblock \emph{Computer methods in applied mechanics and engineering}, 195\penalty0 (41-43):\penalty0 5257--5296, 2006.

\bibitem[Cottrell et~al.(2007)Cottrell, Hughes, and Reali]{cottrell2007studies}
J.~A. Cottrell, T.~Hughes, and A.~Reali.
\newblock {Studies of refinement and continuity in isogeometric structural analysis}.
\newblock \emph{Computer methods in applied mechanics and engineering}, 196\penalty0 (41-44):\penalty0 4160--4183, 2007.

\bibitem[Cottrell et~al.(2009)Cottrell, Hughes, and Bazilevs]{cottrell2009isogeometric}
J.~A. Cottrell, T.~J. Hughes, and Y.~Bazilevs.
\newblock \emph{{Isogeometric analysis: toward integration of CAD and FEA}}.
\newblock John Wiley \& Sons, 2009.

\bibitem[de~Prenter et~al.(2017)de~Prenter, Verhoosel, van Zwieten, and van Brummelen]{de2017condition}
F.~de~Prenter, C.~V. Verhoosel, G.~J. van Zwieten, and E.~H. van Brummelen.
\newblock {Condition number analysis and preconditioning of the finite cell method}.
\newblock \emph{Computer Methods in Applied Mechanics and Engineering}, 316:\penalty0 297--327, 2017.

\bibitem[de~Prenter et~al.(2019)de~Prenter, Verhoosel, and Van~Brummelen]{de2019preconditioning}
F.~de~Prenter, C.~Verhoosel, and E.~Van~Brummelen.
\newblock {Preconditioning immersed isogeometric finite element methods with application to flow problems}.
\newblock \emph{Computer Methods in Applied Mechanics and Engineering}, 348:\penalty0 604--631, 2019.

\bibitem[de~Prenter et~al.(2023)de~Prenter, Verhoosel, van Brummelen, Larson, and Badia]{de2023stability}
F.~de~Prenter, C.~V. Verhoosel, E.~H. van Brummelen, M.~G. Larson, and S.~Badia.
\newblock {Stability and conditioning of immersed finite element methods: analysis and remedies}.
\newblock \emph{Archives of Computational Methods in Engineering}, pages 1--40, 2023.

\bibitem[Deng and Calo(2021)]{deng2021boundary}
Q.~Deng and V.~M. Calo.
\newblock {A boundary penalization technique to remove outliers from isogeometric analysis on tensor-product meshes}.
\newblock \emph{Computer Methods in Applied Mechanics and Engineering}, 383:\penalty0 113907, 2021.

\bibitem[Evans(2022)]{evans2022partial}
L.~C. Evans.
\newblock \emph{Partial differential equations}, volume~19.
\newblock American Mathematical Society, 2022.

\bibitem[Fried and Malkus(1975)]{fried1975finite}
I.~Fried and D.~S. Malkus.
\newblock {Finite element mass matrix lumping by numerical integration with no convergence rate loss}.
\newblock \emph{International Journal of Solids and Structures}, 11\penalty0 (4):\penalty0 461--466, 1975.

\bibitem[Gonz{\'a}lez and Park(2020)]{gonzalez2020large}
J.~A. Gonz{\'a}lez and K.~Park.
\newblock {Large-step explicit time integration via mass matrix tailoring}.
\newblock \emph{International Journal for Numerical Methods in Engineering}, 121\penalty0 (8):\penalty0 1647--1664, 2020.

\bibitem[Hiemstra et~al.(2025)Hiemstra, Nguyen, Eisentr{\"a}ger, Dornisch, and Schillinger]{hiemstra2025higher}
R.~Hiemstra, T.-H. Nguyen, S.~Eisentr{\"a}ger, W.~Dornisch, and D.~Schillinger.
\newblock Higher-order accurate mass lumping for explicit isogeometric methods based on approximate dual basis functions.
\newblock \emph{Computational Mechanics}, pages 1--22, 2025.

\bibitem[Hiemstra et~al.(2021)Hiemstra, Hughes, Reali, and Schillinger]{hiemstra2021removal}
R.~R. Hiemstra, T.~J. Hughes, A.~Reali, and D.~Schillinger.
\newblock {Removal of spurious outlier frequencies and modes from isogeometric discretizations of second-and fourth-order problems in one, two, and three dimensions}.
\newblock \emph{Computer Methods in Applied Mechanics and Engineering}, 387:\penalty0 114115, 2021.

\bibitem[Horn and Johnson(2012)]{horn2012matrix}
R.~A. Horn and C.~R. Johnson.
\newblock \emph{{Matrix analysis}}.
\newblock Cambridge university press, 2012.

\bibitem[Hsu et~al.(2015)Hsu, Kamensky, Xu, Kiendl, Wang, Wu, Mineroff, Reali, Bazilevs, and Sacks]{hsu2015dynamic}
M.-C. Hsu, D.~Kamensky, F.~Xu, J.~Kiendl, C.~Wang, M.~C. Wu, J.~Mineroff, A.~Reali, Y.~Bazilevs, and M.~S. Sacks.
\newblock {Dynamic and fluid--structure interaction simulations of bioprosthetic heart valves using parametric design with T-splines and Fung-type material models}.
\newblock \emph{Computational mechanics}, 55:\penalty0 1211--1225, 2015.

\bibitem[Hughes(2012)]{hughes2012finite}
T.~J. Hughes.
\newblock \emph{{The finite element method: linear static and dynamic finite element analysis}}.
\newblock Courier Corporation, 2012.

\bibitem[Hughes et~al.(2005)Hughes, Cottrell, and Bazilevs]{hughes2005isogeometric}
T.~J. Hughes, J.~A. Cottrell, and Y.~Bazilevs.
\newblock {Isogeometric analysis: CAD, finite elements, NURBS, exact geometry and mesh refinement}.
\newblock \emph{Computer methods in applied mechanics and engineering}, 194\penalty0 (39-41):\penalty0 4135--4195, 2005.

\bibitem[Hughes et~al.(2014)Hughes, Evans, and Reali]{hughes2014finite}
T.~J. Hughes, J.~A. Evans, and A.~Reali.
\newblock {Finite element and NURBS approximations of eigenvalue, boundary-value, and initial-value problems}.
\newblock \emph{Computer Methods in Applied Mechanics and Engineering}, 272:\penalty0 290--320, 2014.

\bibitem[Leidinger(2020)]{leidinger2020explicit}
L.~Leidinger.
\newblock \emph{{Explicit isogeometric B-Rep analysis for nonlinear dynamic crash simulations}}.
\newblock PhD thesis, Technische Universit{\"a}t M{\"u}nchen, 2020.

\bibitem[Li and Wang(2022)]{li2022significance}
X.~Li and D.~Wang.
\newblock {On the significance of basis interpolation for accurate lumped mass isogeometric formulation}.
\newblock \emph{Computer Methods in Applied Mechanics and Engineering}, 400:\penalty0 115533, 2022.

\bibitem[Li et~al.(2024)Li, Hou, and Wang]{li2024interpolatory}
X.~Li, S.~Hou, and D.~Wang.
\newblock {An interpolatory basis lumped mass isogeometric formulation with rigorous assessment of frequency accuracy for Kirchhoff plates}.
\newblock \emph{Thin-Walled Structures}, 197:\penalty0 111639, 2024.

\bibitem[Lorenzo et~al.(2019)Lorenzo, Hughes, Dominguez-Frojan, Reali, and Gomez]{lorenzo2019computer}
G.~Lorenzo, T.~J. Hughes, P.~Dominguez-Frojan, A.~Reali, and H.~Gomez.
\newblock {Computer simulations suggest that prostate enlargement due to benign prostatic hyperplasia mechanically impedes prostate cancer growth}.
\newblock \emph{Proceedings of the National Academy of Sciences}, 116\penalty0 (4):\penalty0 1152--1161, 2019.

\bibitem[Lyche and Morken(2018)]{lyche2018spline}
T.~Lyche and K.~Morken.
\newblock {Spline methods draft}.
\newblock Technical report, Department of Mathematics, University of Oslo, 2018.

\bibitem[Manni et~al.(2022)Manni, Sande, and Speleers]{manni2022application}
C.~Manni, E.~Sande, and H.~Speleers.
\newblock {Application of optimal spline subspaces for the removal of spurious outliers in isogeometric discretizations}.
\newblock \emph{Computer Methods in Applied Mechanics and Engineering}, 389:\penalty0 114260, 2022.

\bibitem[Manni et~al.(2023)Manni, Sande, and Speleers]{manni2023outlier}
C.~Manni, E.~Sande, and H.~Speleers.
\newblock {Outlier-free spline spaces for isogeometric discretizations of biharmonic and polyharmonic eigenvalue problems}.
\newblock \emph{Computer Methods in Applied Mechanics and Engineering}, page 116314, 2023.

\bibitem[Me{\ss}mer et~al.(2021)Me{\ss}mer, Leidinger, Hartmann, Bauer, Duddeck, W{\"u}chner, and Bletzinger]{messmer2021isogeometric}
M.~Me{\ss}mer, L.~F. Leidinger, S.~Hartmann, F.~Bauer, F.~Duddeck, R.~W{\"u}chner, and K.-U. Bletzinger.
\newblock {Isogeometric analysis on trimmed solids: a B-spline-based approach focusing on explicit dynamics}.
\newblock In \emph{Proceedings of the 13th European LS-DYNA conference}, 2021.

\bibitem[Me{\ss}mer et~al.(2022)Me{\ss}mer, Teschemacher, Leidinger, W{\"u}chner, and Bletzinger]{messmer2022efficient}
M.~Me{\ss}mer, T.~Teschemacher, L.~F. Leidinger, R.~W{\"u}chner, and K.-U. Bletzinger.
\newblock {Efficient CAD-integrated isogeometric analysis of trimmed solids}.
\newblock \emph{Computer Methods in Applied Mechanics and Engineering}, 400:\penalty0 115584, 2022.

\bibitem[Morganti et~al.(2015)Morganti, Auricchio, Benson, Gambarin, Hartmann, Hughes, and Reali]{morganti2015patient}
S.~Morganti, F.~Auricchio, D.~Benson, F.~Gambarin, S.~Hartmann, T.~Hughes, and A.~Reali.
\newblock {Patient-specific isogeometric structural analysis of aortic valve closure}.
\newblock \emph{Computer methods in applied mechanics and engineering}, 284:\penalty0 508--520, 2015.

\bibitem[Nguyen et~al.(2023)Nguyen, Hiemstra, Eisentr{\"a}ger, and Schillinger]{nguyen2023towards}
T.-H. Nguyen, R.~R. Hiemstra, S.~Eisentr{\"a}ger, and D.~Schillinger.
\newblock {Towards higher-order accurate mass lumping in explicit isogeometric analysis for structural dynamics}.
\newblock \emph{Computer Methods in Applied Mechanics and Engineering}, 417:\penalty0 116233, 2023.

\bibitem[Nitti et~al.(2020)Nitti, Kiendl, Reali, and de~Tullio]{nitti2020immersed}
A.~Nitti, J.~Kiendl, A.~Reali, and M.~D. de~Tullio.
\newblock {An immersed-boundary/isogeometric method for fluid--structure interaction involving thin shells}.
\newblock \emph{Computer Methods in Applied Mechanics and Engineering}, 364:\penalty0 112977, 2020.

\bibitem[Quarteroni(2009)]{quarteroni2009numerical}
A.~Quarteroni.
\newblock \emph{{Numerical models for differential problems}}, volume~2.
\newblock Springer, 2009.

\bibitem[Radtke et~al.(2024)Radtke, Torre, Hughes, D{\"u}ster, Sangalli, and Reali]{radtke2024analysis}
L.~Radtke, M.~Torre, T.~J. Hughes, A.~D{\"u}ster, G.~Sangalli, and A.~Reali.
\newblock {An analysis of high order FEM and IGA for explicit dynamics: Mass lumping and immersed boundaries}.
\newblock \emph{International Journal for Numerical Methods in Engineering}, page e7499, 2024.

\bibitem[Saad(2011)]{saad2011numerical}
Y.~Saad.
\newblock \emph{{Numerical methods for large eigenvalue problems: revised edition}}.
\newblock SIAM, 2011.

\bibitem[Sande et~al.(2020)Sande, Manni, and Speleers]{sande2020explicit}
E.~Sande, C.~Manni, and H.~Speleers.
\newblock {Explicit error estimates for spline approximation of arbitrary smoothness in isogeometric analysis}.
\newblock \emph{Numerische Mathematik}, 144\penalty0 (4):\penalty0 889--929, 2020.

\bibitem[Stoter et~al.(2023)Stoter, Divi, van Brummelen, Larson, de~Prenter, and Verhoosel]{stoter2023critical}
S.~K. Stoter, S.~C. Divi, E.~H. van Brummelen, M.~G. Larson, F.~de~Prenter, and C.~V. Verhoosel.
\newblock {Critical time-step size analysis and mass scaling by ghost-penalty for immersogeometric explicit dynamics}.
\newblock \emph{Computer Methods in Applied Mechanics and Engineering}, 412:\penalty0 116074, 2023.

\bibitem[Tkachuk and Bischoff(2014)]{tkachuk2014local}
A.~Tkachuk and M.~Bischoff.
\newblock {Local and global strategies for optimal selective mass scaling}.
\newblock \emph{Computational Mechanics}, 53\penalty0 (6):\penalty0 1197--1207, 2014.

\bibitem[Voet et~al.(2023)Voet, Sande, and Buffa]{voet2023mathematical}
Y.~Voet, E.~Sande, and A.~Buffa.
\newblock {A mathematical theory for mass lumping and its generalization with applications to isogeometric analysis}.
\newblock \emph{Computer Methods in Applied Mechanics and Engineering}, 410:\penalty0 116033, 2023.

\bibitem[Voet et~al.(2025{\natexlab{a}})Voet, Sande, and Buffa]{voet2025mass}
Y.~Voet, E.~Sande, and A.~Buffa.
\newblock Mass lumping and outlier removal strategies for complex geometries in isogeometric analysis.
\newblock \emph{Mathematics of Computation}, 2025{\natexlab{a}}.

\bibitem[Voet et~al.(2025{\natexlab{b}})Voet, Sande, and Buffa]{voet2025stabilization}
Y.~Voet, E.~Sande, and A.~Buffa.
\newblock Mass lumping and stabilization for immersogeometric analysis.
\newblock \emph{arXiv preprint arXiv:2502.00452}, 2025{\natexlab{b}}.

\end{thebibliography}

\end{document}